\documentclass[english,letter paper,12pt,reqno]{article}
\usepackage{etex} 
\usepackage{array}
\usepackage{varwidth}
\usepackage{bussproofs}
\usepackage{epigraph}
\usepackage{stmaryrd}
\usepackage{mathdots}
\usepackage{amsmath, amscd, amssymb, mathrsfs, accents, amsfonts,amsthm}
\usepackage[all]{xy}
\usepackage{mathtools} 
\usepackage{tikz-cd}
\usetikzlibrary{calc}
\usetikzlibrary{positioning}

\AtEndDocument{\bigskip{\footnotesize%
  James Clift \par
  \textsc{Department of Mathematics, University of Melbourne} \par
  \textit{E-mail address}: \texttt{j.clift3@student.unimelb.edu.au} \par
  \vspace{0.3cm}
  Daniel Murfet \par
  \textsc{Department of Mathematics, University of Melbourne} \par  
  \textit{E-mail address}: \texttt{d.murfet@unimelb.edu.au} \par
}}

\newenvironment{mathprooftree}
  {\varwidth{.9\textwidth}\centering\leavevmode}
  {\DisplayProof\endvarwidth}
  
\DeclarePairedDelimiter\ket{\lvert}{\rangle}
\DeclarePairedDelimiterX\braket[2]{\langle}{\rangle}{#1 \delimsize\vert #2}
\DeclarePairedDelimiterX\inner[2]{\langle}{\rangle}{#1,#2}

\definecolor{Myblue}{rgb}{0,0,0.6}
\usepackage[a4paper,colorlinks,citecolor=Myblue,linkcolor=Myblue,urlcolor=Myblue,pdfpagemode=None]{hyperref}

\SelectTips{cm}{}

\theoremstyle{definition}
\newtheorem{theorem}{Theorem}[section]
\newtheorem{proposition}[theorem]{Proposition}
\newtheorem{lemma}[theorem]{Lemma}
\newtheorem{corollary}[theorem]{Corollary}

	\newtheorem{definition}[theorem]{Definition}
	\newtheorem{example}[theorem]{Example}
	\newtheorem{remark}[theorem]{Remark}
	\newtheorem{setup}[theorem]{Setup}

\def\res{\operatorname{Res}}

\def\Hom{\operatorname{Hom}}

\def\vacu{\ket{\emptyset}}
\def\be{\begin{equation}}
\def\ee{\end{equation}}
\DeclareMathOperator{\derelict}{derelict}

\DeclareMathOperator{\Wr}{Wr}
\DeclareMathOperator{\Mv}{Mv}
\DeclareMathOperator{\std}{std}
\DeclareMathOperator{\inc}{inc}
\DeclareMathOperator{\nv}{nv}
\DeclareMathOperator{\prom}{prom}

\def\comp{\underline{\textup{comp}}}
\def\contract{\;\lrcorner\;}
                         
\newcommand{\Sum}{\sum\limits}

\begin{document}

\def\ScoreOverhang{1pt}

\makeatletter
\DeclareRobustCommand{\rvdots}{%
  \vbox{
    \baselineskip4\p@\lineskiplimit\z@
    \kern-\p@
    \hbox{}\hbox{.}\hbox{.}\hbox{.}
  }}
\makeatother

\newcommand{\proofvdots}[1]{\overset{\displaystyle #1}{\rvdots}}
\def\Res{\res\!}
\newcommand{\ud}[1]{\operatorname{d}\!{#1}}
\newcommand{\Ress}[1]{\res_{#1}\!}
\newcommand{\cat}[1]{\mathcal{#1}}
\newcommand{\lto}{\longrightarrow}
\newcommand{\xlto}[1]{\stackrel{#1}\lto}
\newcommand{\mf}[1]{\mathfrak{#1}}
\newcommand{\md}[1]{\mathscr{#1}}
\newcommand{\church}[1]{\underline{#1}}
\newcommand{\prf}[1]{\underline{#1}}
\newcommand{\den}[1]{\llbracket #1 \rrbracket}
\def\l{\,|\,}
\def\sgn{\textup{sgn}}
\def\cont{\operatorname{cont}}
\def\counit{\varepsilon}
\def\ptail{\underline{\operatorname{tail}}}
\def\phead{\underline{\operatorname{head}}}
\def\comp{\underline{\textup{comp}}}
\def\mult{\underline{\textup{mult}}}
\def\repeat{\underline{\textup{repeat}}}
\def\contract{\;\lrcorner\;}
\def\<{\langle} \def\>{\rangle}
\newcommand{\id}{\text{id}}
\newcommand{\del}{\partial}
\newcommand{\Inj}{\operatorname{Inj}}

\newcommand{\tTur}{\textbf{Tur}}
\newcommand{\tInt}{\textbf{int}}
\newcommand{\tBint}{\textbf{bint}}
\newcommand{\tList}{\textbf{list}}
\newcommand{\tTint}{\textbf{tint}}
\newcommand{\tBool}{\textbf{bool}}
\newcommand{\tMBool}{\textbf{${_m}$bool}}
\newcommand{\tNBool}{\textbf{${_n}$bool}}

\newcommand{\pLeft}{\text{\underline{left}}}
\newcommand{\pRight}{\text{\underline{right}}}
\newcommand{\pState}{\text{\underline{state}}}
\newcommand{\pRecomb}{\text{\underline{recomb}}}
\newcommand{\pConcat}{\text{\underline{concat}}}
\newcommand{\pTrans}{\text{\underline{trans}}}
\newcommand{\pDecomp}{\text{\underline{decomp}}}
\newcommand{\pHead}{\text{\underline{head}}}
\newcommand{\pTail}{\text{\underline{tail}}}
\newcommand{\pIntCopy}{\text{\underline{intcopy}}}
\newcommand{\pBintCopy}{\text{\underline{bintcopy}}}
\newcommand{\pNBoolCopy}{\text{${_n}$\underline{boolcopy}}}
\newcommand{\pStep}{\text{\underline{step}}}
\newcommand{\pStrstep}{\text{\underline{strstep}}}
\newcommand{\pExtract}{\text{\underline{extract}}}
\newcommand{\pPack}{\text{\underline{pack}}}
\newcommand{\pUnpack}{\text{\underline{unpack}}}
\newcommand{\pRelstep}{\text{\underline{relstep}}}
\newcommand{\pBoolstep}{\text{\underline{boolstep}}}
\newcommand{\pRead}{\text{\underline{read}}}
\newcommand{\pMultread}{\text{\underline{multread}}}
\newcommand{\pTensor}{\text{\underline{tensor}}}
\newcommand{\pComp}{\text{\underline{comp}}}
\newcommand{\pRepeat}{\text{\underline{repeat}}}
\newcommand{\pAdd}{\text{\underline{add}}}
\newcommand{\pMult}{\text{\underline{mult}}}
\newcommand{\pPred}{\text{\underline{pred}}}
\newcommand{\pIter}{\text{\underline{iter}}}
\newcommand{\pBooltype}{\text{\underline{booltype}}}
\newcommand{\pCast}{\text{\underline{cast}}}
\newcommand{\proj}{\operatorname{proj}}
\newcommand{\prob}{\bold{P}}
\newcommand{\probc}{\mathscr{P}}
\newcommand{\dkl}{D_{\operatorname{KL}}}

\newcommand{\nl}{\text{nl}} 
\newcommand{\dntn}[1]{\llbracket #1 \rrbracket} 
\newcommand{\dntntup}[1]{\langle\!\langle #1 \rangle\!\rangle} 
\newcommand{\dntnPT}[1]{\dntn{#1}_{\text{PT}}}
\newcommand{\dntnNL}[1]{\dntn{#1}_{\nl}}
\newcommand{\Sym}{\operatorname{Sym}}
\newcommand{\ND}{\text{ND}} 

\title{Derivatives of Turing machines in Linear Logic}
\author{James Clift, Daniel Murfet}

\maketitle

\begin{abstract} We calculate denotations under the Sweedler semantics of the Ehrhard-Regnier derivatives of various encodings of Turing machines into linear logic. We show that these derivatives calculate the rate of change of probabilities naturally arising in the Sweedler semantics of linear logic proofs. The resulting theory is applied to the problem of synthesising Turing machines by gradient descent.
\end{abstract}

\tableofcontents

\setlength{\epigraphwidth}{0.62\textwidth}
\epigraph{The situation is very much like that we meet in a truss bridge. To determine the strains that the different parts are carrying, we place a weight somewhere on the bridge and measure the deflection of element after element... In a similar way, for a system to show any effective causality, it must be possible to consider how this system would have behaved if it had been built up in a slightly different way.}{Norbert Wiener, \textit{Invention: The Care and Feeding of Ideas}}

\section{Introduction}

This paper is the fifth in a series \cite{murfet_coalg,murfet_ll,clift_murfet,clift_murfet2} studying the semantics of linear logic in vector spaces, where the cofree coalgebras introduced by Sweedler \cite{sweedler} are used to interpret the exponential connective of linear logic. We call this the \emph{Sweedler semantics}. The story so far reads as follows: in \cite{murfet_coalg} the explicit description of the cofree coalgebra was revisited, in \cite{murfet_ll} this was used to define the Sweedler semantics, in \cite{clift_murfet} it was explained how the primitive elements in cofree coalgebras give rise to a natural semantics of differential linear logic, and in \cite{clift_murfet2} the denotations of encodings of Turing machines were calculated. In this paper we take the proofs from \cite{clift_murfet2} encoding Turing machines, and compute using \cite{clift_murfet} the denotations of their derivatives in the sense of differential linear logic. Our aim at the beginning of this project was to, firstly, compute these ``derivatives of Turing machines'' and to, secondly, try to understand if the answer is computationally meaningful.

The differential lambda calculus is a system introduced by Ehrhard and Regnier \cite{difflambda} in which arbitrary algorithms (that is, lambda terms) may be ``differentiated''. This is a remarkable idea, originating in the semantics of linear logic \cite{ehrhard-kothe}. The differential lambda calculus and its cousin differential linear logic have good properties (such as confluence \cite[\S 3.1]{difflambda} and strong normalisation \cite[\S 5]{difflambda}) so we may view the derivative of an algorithm as an algorithm in its own right. This leads to the following basic
\begin{center}
\textbf{Question}: \emph{what is it} that the derivative of an algorithm computes? 
\end{center}
In this paper we use the Sweedler semantics to address this question, at least for a class of simple proofs in linear logic that we call component-wise plain proofs. Our answer is that the derivative of such a proof computes \emph{rates of change of naive probability}. The naive probability is, technically speaking, the result of applying the denotation of a proof in the Sweedler semantics to a group-like element corresponding to a probability distribution over proofs. More conceptually it may be described as the probability (in the Bayesian sense of a degree of belief) assigned by an observer of the operation of the algorithm, to an output value given some uncertainty about the input, where the observer reasons about the operation of the algorithm under certain (\emph{naive}) assumptions about the conditional independence of the variables involved. The same assumptions appear in \emph{naive Bayesian classifiers} \cite[\S 20.2.2]{russell_norvig} from which we borrow the name. 
\\

\textbf{Outline of the paper:} In order to make the paper accessible to a wider audience, we begin in Section \ref{section:motivationbig} with a high-level motivation for the idea of derivatives of algorithms, and linear logic in that context. In the main text our first goal is to define the naive probability associated to plain proofs, and explain how its rate of change is computed by denotations of the derivatives of these proofs in the Sweedler semantics (Theorem \ref{theorem:maind}). The second goal is to illustrate the behaviour of this probability using Turing machines (Section \ref{section: turing machines}). One possible area of application of the Ehrhard-Regnier derivative of algorithms is in machine learning, and in Section \ref{section:grad_descent_tm} we develop a theory of gradient descent based on derivatives of Turing machines.
\\

\textsl{Acknowledgements.} DM thanks Huiyi Hu for patiently explaining various aspects of machine learning, and collaboration on early versions of these ideas.

\section{Motivation}\label{section:motivationbig}

\subsection{Error propagation and algorithms}\label{section:motivation}

What is the derivative of an algorithm? It is not clear that such a thing should always exist, but when it does, it must surely result from making an infinitesimal variation in some part of the input, and measuring the infinitesimal variation in some part of the output. If the algorithm is numeric and computes a real-valued function of real-valued inputs, this works: there is an algorithm which computes the derivative of the output with respect to a given input, and the study of these algorithms is called \emph{automatic differentiation} \cite{autodiff}. 

However, for algorithms with discrete inputs and outputs the problem with this intuition is clear: it does not \emph{a priori} make sense to make such infinitesimal variations. In this part of the introduction we try to motivate our general point of view on this problem, which is the problem of \emph{propagation of error} (or \emph{uncertainty}) through algorithms.
\\

Given continuous random variables $X_1,\ldots,X_n$ (representing, for example, measurements) and a smooth function $f(x_1,\ldots,x_n)$ let $F$ be the output distribution, obtained for example by sampling from the input distributions and aggregating the results of applying $f$ to the samples. The error propagation law (see \cite[p.71]{bonamente}, \cite{ku}) states that in the case that the errors in the $X_i$ are independent, the variance of $F$ is approximately determined by the variance of the input variables $X_i$, according to the formula
\be\label{eq:lawoferrorprop}
\sigma^2_F \simeq \sum_{i=1}^n \Big[ \frac{\partial f}{\partial x_i} \Big]^2 \sigma^2_{X_i}
\ee
where the partial derivatives are evaluated by setting $x_i$ to be the mean $\mu_i$ of $X_i$. The significant point is that the variance of each input is weighted according to how strongly that particular input influences the output, with the degree of influence being measured by the partial derivative; see also \cite[\S 3.6]{sivia}. We may turn this around, and assuming that all the $X_i$ but the first are delta functions at some particular values $\mu_i$, obtain
\be\label{eq:derivative_smooth_f}
\Big\vert\frac{\partial f}{\partial x_1}(\mu_1,\ldots,\mu_n)\Big\vert \simeq \frac{\sigma_F}{\sigma_{X_1}}\,.
\ee
The right hand side of this equation gives us an alternative way to think about the partial derivatives of $f$ at arbitrary points of $\mathbb{R}^n$. As long as we know the extension of $f$ to a function of probability distributions $F = F(X_1,\ldots,X_n)$ we can estimate these partial derivatives by preparing suitable input distributions, computing the output distribution, and taking the ratio of a measure of the output and input errors.
\\

Now let us consider the case of a discrete algorithm $\psi$, and see to what extent we can emulate this point of view on the derivative. Suppose the input variable $x$ ranges over a discrete set of symbols $\Sigma$ and that an observer has some uncertainty about the true value of the input being $x = \sigma_0$, represented by a probability distribution
\[
(1-h) \cdot \sigma_0 + \sum_{\sigma \neq \sigma_0} u_\sigma h \cdot \sigma
\]
where the $u_\sigma$ give a probability distribution over $\Sigma \setminus \{ \sigma_0 \}$. We take $h$ as a measure of the error or uncertainty in the input value\footnote{This agrees with the standard deviation from $\sigma_0$ up to a constant factor, if we assume the symbols in $\Sigma$ are embedded in a metric space as equidistant points, as for example in the free vector space $\mathbb{R} \Sigma$.}. This uncertainty leads to uncertainty about the output of the algorithm, which we suppose lies in some set of symbols $\Lambda$. For the moment we do not take a position on precisely how the uncertainty about the input is propagated to uncertainty about the output. However it is calculated, the output uncertainty takes the form of a probability distribution
\[
(1- f(h)) \cdot \psi(\sigma_0) + \sum_{\lambda \neq \psi(\sigma_0)} V_\lambda(h) \cdot \lambda\,.
\]
We take $f(h)$ as a measure of the error in the output arising from an error of $h$ in the input, and assume this varies smoothly as a function of $h$. In the limit where our uncertainty about the input being other than $\sigma_0$ is an infinitesimal $\Delta h$, our uncertainty about the output being other than $\psi(\sigma_0)$ is an infinitesimal $\Delta f$, and the analogy with \eqref{eq:derivative_smooth_f} suggests that the derivative of the algorithm $\psi$ with respect to $x$ at $\sigma_0$ is approximated by the ratio
\be\label{eq:heuristic_derivative}
\frac{ \Delta f}{ \Delta h} \simeq \frac{f'(0) \Delta h}{\Delta h} = f'(0)\,.
\ee
The proposal we have arrived at is to define the derivative of an algorithm at some input symbol $x = \sigma_0$ to be the ratio of a measure of input and output errors, where the error is in the input $x$ and all other inputs are fixed. To formalise this, we have to specify how to extend $\psi$ to an algorithm which takes probability distributions as inputs and returns probability distributions as outputs, that is, we have to specify a \emph{probabilistic extension}. 

The probabilistic extension of $\psi$ dictated by standard probability does not give rise to a meaningful theory of derivatives of algorithms, because it depends only on the function $\Sigma \lto \Lambda$ encoded by $\psi$. We therefore do not expect any fundamental connection between standard probability and derivatives of algorithms. So the question becomes: how should we propagate error through algorithms? This is where we turn to linear logic.

\subsection{Differential linear logic}\label{section:dll}

The approach of this paper is coalgebraic and we begin with a brief overview of how this perspective charts a natural course from tensor algebra to differentiable computation. The starting point is the category $\cat{V}$ of vector spaces over an algebraically closed field, equipped with the structure of a closed symmetric monoidal additive category, that is, with the operations $\oplus, \otimes, \multimap$ where $V \multimap W$ is the space of linear maps from $V$ to $W$. This is the world of \emph{tensor algebra}, and it is computationally very meagre: we may only construct linear functions. The canonical nonlinear extension of this algebra is obtained by adding the operation ${!}$ of formation of cofree coalgebras \cite[Ch VI.]{sweedler} and the resulting language of constructions in the operators $\oplus, \; \otimes, \; \multimap, \; !$ is called linear logic \cite{girard_llogic}. 

What this means is that the types of linear logic are expressions in these operators, and the semantics $\dntn{-}$ maps such expressions to actual vector spaces, with for example
\[
\dntn{ {!} A \multimap B } = \Hom_k\!\big( \dntn{ {!} A }, \dntn{B} \big) = \Hom_k\!\big( {!} \dntn{A}, \dntn{B} \big)\,.
\]
To understand the ramifications of introducing cofree coalgebras into tensor algebra, we look first to the group-like \cite[p.57]{sweedler} and primitive elements \cite[p.199]{sweedler} of these coalgebras. Given a proof $\psi$ in linear logic of the sequent ${!} A \vdash B$, the denotation is a linear map 
\[
\dntn{\psi}: {!} \dntn{A} \lto \dntn{B}\,.
\]
We have for each proof $\alpha : A$ a vector $\dntn{\alpha} \in \dntn{A}$ and a group-like element $\vacu_{\dntn{\alpha}} \in {!} \dntn{A}$ (which we call the vacuum supported at $\dntn{\alpha}$) and it is the behaviour of the denotation of $\dntn{\psi}$ on these group-like elements which recovers the input-output behaviour of $\psi$, since
\[
\dntn{\psi} \vacu_{\dntn{\alpha}} = \dntn{ \psi(\alpha) }
\]
where $\psi(\alpha) : B$ is the output of the algorithm $\psi$ on input $\alpha$. But there is more information in the linear map $\dntn{\psi}$ than the input-output behaviour of $\psi$, and the two most obvious ways to probe this additional information are:
\begin{itemize}
\item[(i)] Lift $\dntn{\psi}$ to a morphism of coalgebras
\be
{!} \dntn{A} \lto {!} \dntn{B}\label{eq:lifting_psi_intro}
\ee
and evaluate this morphism on the \emph{primitive elements} of the coalgebra ${!} \dntn{A}$. Given proofs $\alpha, \beta : A$ there is an associated primitive element $| \dntn{\beta} \rangle_{\dntn{\alpha}}$ in ${!} \dntn{A}$.

\item[(ii)] Evaluate $\dntn{\psi}$ on vacuums at \emph{linear combinations} of proof denotations:
\be
\dntn{\psi} \vacu_{\sum_i \lambda_i \dntn{\alpha_i}} \in \dntn{B}\,.
\ee
\end{itemize}
Examining (i) one discovers the differentiable structure of algorithms in linear logic \cite{clift_murfet}. Indeed, the restriction of \eqref{eq:lifting_psi_intro} to primitive elements is a kind of tangent map (see Corollary \ref{cor:diaprimcomp}) and we refer to it as the \emph{coalgebraic derivative} of $\dntn{\psi}$ (Definition \ref{defn:coalg_derivative_plain}). On the other hand, examining (ii) one discovers a natural probabilistic semantics of linear logic, which is what we refer to as \emph{naive probability} in Section \ref{section:probexec}. These are naturally related, and in the present paper we explore this connection.

The connection is most apparent for proofs $\psi$ of the form
\begin{center}
\AxiomC{$\proofvdots{\pi}$}
\noLine\UnaryInfC{$n\, A \vdash B$}
\RightLabel{\scriptsize der}
\doubleLine\UnaryInfC{$n\, {!}A \vdash B$}
\RightLabel{\scriptsize ctr}
\doubleLine\UnaryInfC{${!}A \vdash B$}
\DisplayProof
\end{center}
which, viewed as algorithms, act by copying their input a fixed number of times and thereafter using these copies in a linear way, as described by $\pi$. We refer to such algorithms as \emph{plain proofs} \cite[Definition 3.1]{clift_murfet2}. The coalgebraic derivative of $\psi$ at a proof $\alpha : A$ in the direction of $\beta : A$ obtained using primitive elements agrees with a limit
\be\label{eq:limit_intro}
\llbracket \psi \rrbracket \Big| \dntn{\beta} - \dntn{\alpha} \Big\rangle_{\dntn{\alpha}} = \lim_{h \to 0} \frac{ \dntn{\psi} \vacu_{(1-h)\dntn{\alpha} + h \dntn{\beta}} - \llbracket \psi \rrbracket |\emptyset\rangle_{\llbracket \alpha \rrbracket}}{h}\,.
\ee
The left hand side is the notion of differentiation which is native to the world of nonlinear tensor algebra (that is, linear logic) and which is axiomitised by differential linear logic. The right hand side is a derivative in the sense of ordinary calculus. This is a more familiar object, however, in order to give it a clear computational interpretation we must understand the computational meaning of the values of $\dntn{\psi}$ on vacuums which are supported at \emph{probability distributions} 
\be\label{eq:distribution_intro}
(1-h) \dntn{\alpha} + h \dntn{\beta}\,.
\ee
The evaluation of $\dntn{\psi}$ at a vacuum supported at a distribution over proof denotations has the following interpretation: execute $\psi$ in such a way that at every step where the input value of type $A$ is to be used, we sample independently (of all other such steps) a value from the distribution. So we use $\alpha$ with probability $1-h$ and $\beta$ with probability $h$. 

The \emph{naive probabilistic extension} of $\psi$, denoted $\Delta \psi$, is the function which computes the distribution over output values of $\psi$ induced by the distribution \eqref{eq:distribution_intro} over input values in the manner just described (see Section \ref{section:probexec}). This probabilistic extension, which arises naturally in the context of the Sweedler semantics, gives our answer for how to propagate uncertainty through algorithms (at least those which are encoded as component-wise plain proofs in linear logic). The proposal in Section \ref{section:motivation} above to realise the derivative of $\psi$ using this probabilistic extension is given in Section \ref{section:probderive}. To be more precise, the derivative of $\psi$ that has already been defined by Ehrhard-Regnier is shown to coincide semantically with this rate of change of naive probability. 


\section{Background}\label{section:background}

Throughout $k$ is an algebraically closed field, and all vector spaces and coalgebras are defined over $k$. Our coalgebras are all coassociative, counital and cocommutative. Our reference for coalgebras is \cite{sweedler}. We write $\operatorname{Prim}(C)$ for the set of primitive elements in a coalgebra $C$. Throughout ``$\inc$'' always denotes an inclusion.

Our conventions for linear logic and its semantics are as in \cite[\S 2.1]{clift_murfet2}. By \emph{linear logic} we will always mean \emph{first-order intuitionistic linear logic} with connectives $\otimes, \&, \multimap, {!}$ and the corresponding introduction rules and cut-elimination transformations \cite{mellies,benton_etal} and $\dntn{-}$ denotes the Sweedler semantics \cite{clift_murfet, murfet_ll}. See \cite[\S 2.1]{clift_murfet2} for a brief introduction tailored to the present paper. Whenever we talk about a set of proofs $\cat{P}$ of a formula $A$ in linear logic, we always mean a set of proofs \emph{modulo the equivalence relation of cut-elimination}. Given a set of proofs $\cat{N}$ we write $\dntn{\cat{N}}$ for $\{ \dntn{\nu} \}_{\nu \in \cat{N}}$. If proofs $\pi,\pi'$ are equivalent under cut-elimination then $\dntn{\pi} = \dntn{\pi'}$, so the function $\dntn{-}: \cat{P} \lto \dntn{A}$ extends uniquely to a $k$-linear map
\be\label{eq:denotesec}
\xymatrix@C+2pc{ k \cat{P} \ar[r]^-{\dntn{-}} & \dntn{A} }
\ee
where $k \cat{P}$ is the free $k$-vector space generated by the set $\cat{P}$. If $\psi$ is a proof of ${!} A_1,\ldots,{!} A_r \vdash B$ and $\alpha_i$ is a proof of $A_i$ for $1 \le i \le r$ then $\psi(\alpha_1,\ldots,\alpha_r) : B$ denotes the (cut-elimination equivalence class of) the proof obtained by cutting $\psi$ against the promotion of each $\alpha_i$.
\\

A brief reference for common notation:
\begin{itemize}
\item $A, B, C$ are formulas of linear logic.
\item $\alpha, \beta, \psi, \pi, \rho, \zeta$ are proofs of sequents in linear logic.
\item $\cat{P}, \cat{Q}, \cat{R}$ are sets of proofs.
\item $u, v,w$ are vectors (in for example $k \cat{P}$, where $\cat{P}$ is a set).
\item $\bold{u}, \bold{v}, \bold{w}$ are sequences of vectors (in for example $\prod_i k \cat{P}_i$).
\item Since the set of probability distributions $\Delta \cat{P}$ is contained in the free vector space $\mathbb{R} \cat{P}$ we denote probability distributions by the same notation $v,w$ as vectors. Similarly sequences of distributions are denoted $\bold{v}, \bold{w}$.
\end{itemize}


\section{Coalgebraic derivatives of plain proofs}\label{section:derivative_plain}

In this section we recall some of the basic theory of coalgebraic derivatives from \cite{clift_murfet} and study in detail the special class of plain proofs for which these derivatives can be described in terms of explicit polynomials. The reader is encouraged to read \cite[\S 3]{clift_murfet2} before reading this section. Given any proof in linear logic
\[
\psi: {!} A_1, \ldots, {!} A_r \vdash B
\]
the denotation is a linear map
\[
\dntn{\psi}: {!} \dntn{A_1} \otimes \cdots \otimes {!} \dntn{A_r} \lto \dntn{B}\,.
\]
The promotion $\prom(\psi)$ is the proof
\begin{center}
\AxiomC{$\proofvdots{\psi}$}
\noLine\UnaryInfC{${!} A_1,\ldots,{!} A_r \vdash B$}
\RightLabel{\scriptsize prom}
\UnaryInfC{${!} A_1,\ldots,{!} A_r \vdash {!} B$}
\DisplayProof
\end{center}
which has for its denotation $\dntn{\prom(\psi)}$ the unique morphism of coalgebras making
\[
\xymatrix@C+4pc@R+1pc{
\bigotimes_{i=1}^r {!} \dntn{A_i} \ar[dr]_-{\dntn{\psi}}\ar[r]^-{\dntn{\prom(\psi)}} & {!} \dntn{B} \ar[d]^-{d}\\
& \dntn{B}
}
\]
commute. Any linear map $\kappa$ from a coalgebra into $\dntn{B}$ induces a morphism of coalgebras into ${!} \dntn{B}$, and in general we call this induced morphism the \emph{promotion} of $\kappa$ and denote it $\prom(\kappa)$. Thus $\dntn{\prom(\psi)} = \prom \dntn{\psi}$. Since $\dntn{\prom(\psi)}$ is a morphism of coalgebras it sends primitive elements to primitive elements, and there is a commutative diagram
\[
\xymatrix@C+2pc{
\bigotimes_{i=1}^r {!} \dntn{A_i} \ar[r]^-{\dntn{\prom(\psi)}} & {!} \dntn{B}\\
\operatorname{Prim}\big(\! \bigotimes_{i=1}^r {!} \dntn{A_i} \big)\ar[u]^{\inc} \ar[r] & \operatorname{Prim}\big( {!} \dntn{B} \big)\ar[u]_-{\inc}
}
\]

\begin{definition}\label{defn:coalg_derivative_plain} The \emph{coalgebraic derivative} of $\dntn{\psi}$ is the function induced on primitive elements by the promotion $\prom\dntn{\psi}$, that is, the bottom row in the above diagram.
\end{definition}

One should think of the coalgebraic derivative as being analogous to the amalgamation of all the tangent maps at points $x \in M$ of a smooth map $M \lto N$ (see Corollary \ref{cor:diaprimcomp}). We completely understand the primitive elements in cofree coalgebras. From the basic theory of coalgebras \cite{sweedler}, \cite[\S 2.3]{clift_murfet2} we know that there are bijections
\begin{gather*}
\xymatrix{
\prod_{i=1}^r \dntn{A_i}^2 \ar[r]^-{\cong} & \operatorname{Prim}\big(\!\bigotimes_{i=1}^r {!} \dntn{A_i} \big)}\\
\big( (v_i, w_i) \big)_{i=1}^r \longmapsto \sum_{i=1}^r \vacu_{v_1} \otimes \cdots \otimes \ket{w_i}_{v_i} \otimes \cdots \otimes \vacu_{v_r}
\end{gather*}
and similarly
\begin{gather*}
\xymatrix{ \dntn{B}^2 \ar[r]^-{\cong} & \operatorname{Prim}\big( {!} \dntn{B} \big) }\,\\
(v, w) \longmapsto \ket{w}_{v}\,.
\end{gather*}
The upshot is that the coalgebraic derivative of $\dntn{\psi}$ is some function $\prod_{i=1}^r \dntn{A_i}^2 \lto \dntn{B}^2$. In this section we explain how to calculate this function in the special case where the original proof $\psi$ is \emph{plain} \cite[Definition 3.1]{clift_murfet2}.

\begin{setup}\label{setup:psiandpi} In this section $\psi: {!} A_1, \ldots, {!} A_r \vdash B$ is a plain proof and $\pi$ is a proof such that
\begin{center}
\AxiomC{$\proofvdots{\pi}$}
\noLine\UnaryInfC{$n_1\, A_1,\ldots,n_r \, A_r \vdash B$}
\RightLabel{\scriptsize der}
\doubleLine\UnaryInfC{$n_1\, {!}A_1, \ldots, n_r \, {!} A_r \vdash B$}
\RightLabel{\scriptsize ctr/wk}
\doubleLine\UnaryInfC{${!}A_1,\ldots,{!}A_r \vdash B$}
\DisplayProof
\end{center}
is equivalent under cut-elimination to $\psi$. Given such a presentation we refer to $n_i$ as the \emph{$A_i$-degree}. Throughout $\cat{P}_i$ is a finite set of proofs of $A_i$ for $1 \le i \le r$ and $\cat{Q}$ is a finite set of proofs of $B$ such that $\{ \dntn{\nu} \}_{\nu \in \cat{Q}}$ is linearly independent in $\dntn{B}$ and such that
\be\label{eq:coll_comput}
\Big\{ \pi( X_1,\ldots,X_r ) \,\Big\vert\, X_i \in (\cat{P}_i)^{n_i} \text{ for } 1 \le i \le r \Big\} \subseteq \cat{Q}\,.
\ee
\end{setup}

As a consequence of the next result, quoted from \cite{clift_murfet2}, we know that the values of $\dntn{\psi}$ on vacuums can be computed by polynomials. We need to recall the following notation from \cite[\S 3]{clift_murfet2}: given a function $\gamma: \{1,\ldots,n_i\} \lto \cat{P}_i$ for some $i$, we write
\begin{align}
\rho_\gamma &= \gamma(1) \otimes \cdots \otimes \gamma(n_i) : A_i^{\otimes n_i} \label{eq:rhogamma}\\
\dntn{\rho_\gamma} &= \dntn{ \gamma(1) } \otimes \cdots \otimes \dntn{\gamma(n_i)} \in \dntn{A_i}^{\otimes n_i}\,.
\end{align}
Let $\iota$ denote the function
\begin{gather*}
\iota: \prod_{i=1}^r k\cat{P}_i \lto \bigotimes_{i=1}^r {!} \dntn{A_i}\,,\\
\iota\big( \omega_1, \ldots, \omega_r \big) = \bigotimes_{i=1}^r \vacu_{\dntn{\omega_i}}
\end{gather*}
where $k \cat{P}$ is the free vector space on $\cat{P}$.

\begin{proposition}\label{prop:fpi} There is a unique function $F_\psi$ making the diagram
\[
\xymatrix@C+3pc@R+1pc{
{!} \dntn{A_1} \otimes \cdots \otimes {!} \dntn{A_r} \ar[r]^-{\dntn{\psi}} & \dntn{B}\\
k\cat{P}_1 \times \cdots \times k\cat{P}_r \ar[u]^-{\iota}\ar[r]_-{F_\psi} & k\cat{Q} \ar[u]_-{\dntn{-}}
}
\]
commute. This function is induced by a morphism of $k$-algebras
\[
f_\psi : \Sym(k \cat{Q}) \lto \Sym( k\cat{P}_1 \oplus \cdots \oplus k \cat{P}_r )\,.
\]
More precisely, if we present the symmetric algebras as polynomial rings in variables
\[
\{ y_\tau \}_{\tau \in \cat{Q}}\,,\qquad \{ x^i_\rho \}_{1 \le i \le r, \rho \in \cat{P}_i}
\]
respectively, then the polynomial $f^\tau_\psi := f_\psi( y_\tau )$ is given by the formula
\be\label{eq:formulaforfpitau}
f_\psi^\tau = \sum_{\gamma_1,\ldots,\gamma_r} \delta_{\tau = \pi(\rho_{\gamma_1},\ldots,\rho_{\gamma_r})} \prod_{i=1}^r \prod_{j=1}^{n_i} x^i_{\gamma_i(j)}\,,
\ee
where $\gamma_i$ ranges over all functions $\{1,\ldots,n_i\} \lto \cat{P}_i$.
\end{proposition}
\begin{proof}
This is \cite[Proposition 3.8]{clift_murfet2}.
\end{proof}

The values of the linear map $\dntn{\psi}$ on general kets can be computed using the partial derivatives of $F_\psi$. To prove this we introduce some more notation. Given tuples
\begin{align*}
\bold{m}^i &= \{ m^i_\rho \}_{\rho \in \cat{P}_i} \in \mathbb{N}^{\cat{P}_i}
\end{align*}
we write $\bold{m} = \{ \bold{m}^i \}_{i=1}^r$ and let $\iota^{\bold{m}}$ denote the function
\begin{gather*}
\iota^{\bold{m}}: \prod_{i=1}^r k\cat{P}_i \lto \bigotimes_{i=1}^r {!} \dntn{A_i}\,,\\
\iota^{\bold{m}} \big( w_1, \ldots, w_r \big) = \bigotimes_{i=1}^r \Big| \prod_{\rho \in \cat{P}_i} \dntn{\rho}^{\otimes m^i_\rho} \Big\rangle_{\dntn{w_i}}
\end{gather*}
where we are using the notation of \eqref{eq:denotesec} in writing $\dntn{w_i}$.

\begin{proposition}\label{prop:diagramforF} For each tuple $\bold{m}$ there is a unique function $F^{\bold{m}}_\psi$ making
\[
\xymatrix@C+3pc@R+1pc{
{!} \dntn{A_1} \otimes \cdots \otimes {!} \dntn{A_r} \ar[r]^-{\dntn{\psi}} & \dntn{B}\\
k\cat{P}_1 \times \cdots \times k\cat{P}_r \ar[u]^-{\iota^{\bold{m}}}\ar[r]_-{F^{\bold{m}}_\psi} & k\cat{Q} \ar[u]_-{\dntn{-}}
}
\]
commute. This function is induced by a morphism of $k$-algebras
\[
f^{\bold{m}}_\psi : \Sym(k \cat{Q}) \lto \Sym( k\cat{P}_1 \oplus \cdots \oplus k \cat{P}_r )\,,
\]
and moreover the polynomials making up this morphism are computed by
\be\label{eq:fmpitau}
f^{\bold{m}}_\psi(y_\tau) = \Big\{ \prod_{i=1}^r \prod_{\rho \in \cat{P}_i} \Big[ \frac{\partial}{\partial x^i_{\rho}} \Big]^{m^i_{\rho}} \Big\} f^\tau_\psi\,.
\ee
\end{proposition}
\begin{proof}
For the duration of the proof we write $\rho$ rather than $\dntn{\rho}$ for the denotation of proofs, just to avoid complicating the notation. Since everything takes place at the semantic level this should not cause confusion. Set $w_i = \sum_\rho \lambda^i_\rho \rho$ as above, so that
\be\label{eq:thingtocompute}
\dntn{\psi}\Big( \bigotimes_{i=1}^r \Big| \prod_{\rho \in \cat{P}_i} \rho^{\otimes m^i_\rho} \Big\rangle_{w_i} \Big) = \dntn{\pi} \Big( \bigotimes_{i=1}^r d^{\otimes n_i} \Delta^{n_i-1} \Big| \prod_{\rho \in \cat{P}_i} \rho^{\otimes m^i_\rho} \Big\rangle_{w_i}\Big)\,.
\ee
Now $d^{\otimes k} \Delta^{k-1} \big| s_1, \ldots, s_l \rangle_{w}$ is zero if $l > k$ and for $l \le k$ it is the $k$-tensor which is the sum over all ways of taking a list of $k$ copies of $w$, replacing $l$ of the copies with the $s_i$'s (in any order), and then forming the tensor over the list. For example with $l = 2$ and $k = 4$ one of the tensors obtained by this procedure is
\[
w, w, w, w \longmapsto w, s_2, w, s_1 \longmapsto w \otimes s_2 \otimes w \otimes s_1\,.
\]
In particular,
\[
d^{\otimes n_i} \Delta^{n_i-1} \Big| \prod_{\rho \in \cat{P}_i} \rho^{\otimes m^i_\rho} \Big\rangle_{w_i}
\]
is the sum of tensors obtained from a list of $n_i$ copies of $w_i$ by replacing, for each $\rho$, $m^i_\rho$ copies of $w_i$ by $\rho$. Thus the input tensor to $\dntn{\pi}$ in \eqref{eq:thingtocompute} is obtained from the list
\be\label{eq:pureomegalist}
\md{L} = \big(w_1,\ldots,w_1, \ldots, w_r, \ldots, w_r\big)
\ee
in which $w_i$ appears $n_i$ times, by replacing for each $1 \le i \le r$ and $\rho \in \cat{P}_i$, $m^i_\rho$ copies of $w_i$ by $\rho$, tensoring together the elements of the list, and then summing over all ways of making such replacements. Given such a list of vectors let $x^{ij}_\rho$ stand for the coefficient of $\rho$ in the the vector which is $j$th among the positions in $\mathscr{L}$ occupied by $w_i$'s. We have calculated in \cite[\S 3.1]{clift_murfet2} that the result of applying $\dntn{\pi}$ to such a list is
\be\label{eq:outputlist}
\sum_{\tau \in \cat{Q}} g^\tau_{\pi}( \{ x^{ij}_\rho \}_{i,j,\rho}) \dntn{\pi}\,.
\ee
Starting with the list $\md{L}$ let us replace the $j$th copy of $w_i$ by $\rho$, for some fixed $i,j,\rho$, and call this new list $\md{L}'$. We can compute the value of $\dntn{\pi}$ on the tensor associated to the new list by setting $x^{ij}_\rho = 1$ and $x^{ij}_\zeta = 0$ for $\zeta \neq \rho$ in \eqref{eq:outputlist} and then evaluating all the variables $x^{i'j'}_\zeta = \lambda^{i'}_\zeta$. Thus we may compute $\dntn{\pi}$ on the new tensor by evaluating the derivative
\[
\sum_{\tau \in \cat{Q}} \frac{\partial}{\partial x^{ij}_\rho} g^\tau_{\pi} \dntn{\pi}
\]
at the same point $x^{i'j'}_\zeta = \lambda^{i'}_\zeta$. Here we use that for each $1 \le i \le r$ and $1 \le j \le n_i$ there is in each monomial of $g^\tau_\pi$ \emph{precisely one} variable from the set $\{ x^{ij}_\rho \}_{\rho \in \cat{P}_i}$. We conclude that if we begin with $\md{L}$ and replace for each $i$ a total of $m^i_\rho$ copies of $w_i$ by $\rho$, say in positions $1 \le j_1 < \cdots < j_{m^i_\rho} \le n_i$, and call the new list $\md{L}'$, then the result of applying $\dntn{\pi}$ to the tensor associated to this list is the result of evaulating
\[
\sum_{\tau \in \cat{Q}} \Big\{ \prod_{i=1}^r \prod_{a=1}^{n_i} \frac{\partial}{\partial x^{ij_a}_\rho} \Big\} g^\tau_{\pi} \dntn{\pi}
\]
at the point $x^{i'j'}_\zeta = \lambda^{i'}_\zeta$. This proves that \eqref{eq:thingtocompute} is equal to the evaluation of
\be\label{eq:wombatdig}
\sum_{\tau \in \cat{Q}} \sum_{\bold{j}} \Big\{ \prod_{i=1}^r \prod_{\rho \in \cat{P}_i} \prod_{a=1}^{m^i_\rho} \frac{\partial}{\partial x^{ij^\rho_a}_\rho} \Big\} g^\tau_{\pi} \dntn{\pi}
\ee
at the point $x^{i'j'}_\zeta = \lambda^{i'}_\zeta$, where $\bold{j}$ ranges over assignments of sequences to proofs, where we must assign to each pair $(i, \rho)$ consisting of $1 \le i \le r$ and $\rho \in \cat{P}_i$ an increasing sequence $j^\rho = (j^\rho_1,\ldots,j^\rho_{m^i_\rho})$ in $\{1,\ldots,n_i\}$ of length $m^i_\rho$, such that for any $i$ and distinct proofs $\rho, \rho' \in \cat{P}_i$ we have $j^\rho \cap j^{\rho'} = \emptyset$. Next observe that for $1 \le a \le r$ and $\zeta \in \cat{P}_i$ the diagram
\[
\xymatrix@C+2pc@R+1pc{
k\Big[ \{ x^{ij}_\rho \}_{i,j,\rho} \Big] \ar[d]_-{\sum_j \frac{\partial}{\partial x^{aj}_\zeta}} \ar[r]^-{C} & k\Big[ \{ x^{i}_\rho \}_{i,j,\rho} \Big] \ar[d]^-{\frac{\partial}{\partial x^a_\zeta}}\\
k\Big[ \{ x^{ij}_\rho \}_{i,j,\rho} \Big] \ar[r]_-{C} & k\Big[ \{ x^{i}_\rho \}_{i,j,\rho} \Big]
}
\]
commutes, using the map $C$ of \cite[\S 3.1]{clift_murfet2}. Hence the right hand side of \eqref{eq:fmpitau} is
\begin{align*}
\Big\{ \prod_{i=1}^r \prod_{\rho \in \cat{P}_i} \Big[ \frac{\partial}{\partial x^i_{\rho}} \Big]^{m^i_{\rho}} \Big\} C(g^\tau_\pi) &= C\Big( \Big\{ \prod_{i=1}^r \prod_{\rho \in \cat{P}_i} \Big[ \sum_{j=1}^{n_i} \frac{\partial}{\partial x^{ij}_{\rho}} \Big]^{m^i_{\rho}} \Big\} g^\tau_\pi \Big)\\
&= C\Big( \sum_{\bold{j}} \Big\{ \prod_{i=1}^r \prod_{\rho \in \cat{P}_i} \prod_{a=1}^{n_i} \frac{\partial}{\partial x^{ij^\rho_a}_\rho} \Big\} g^\tau_\pi \Big)\,.
\end{align*}
But substituting a point $x^{i'}_\zeta = \lambda^{i'}_\zeta$ into the polynomial output of $C$ in the final line is the same as substituting $x^{i'j'}_\zeta = \lambda^{i'}_\zeta$ to the input polynomial, which by \eqref{eq:wombatdig} computes the vector \eqref{eq:thingtocompute}. This proves that the polynomials \eqref{eq:fmpitau} assemble to a morphism of $k$-algebras $f^{\bold{m}}_\psi$ whose associated function $F^{\bold{m}}_\psi$ makes the necessary diagram commute.
\end{proof}

\begin{corollary}\label{corollary:summaryderiv} Given $w_i \in k \cat{P}_i$ for $1 \le i \le r$ and integers $\{ m^i_\rho \}_{1 \le i \le r, \rho \in \cat{P}_i}$ we have
\be\label{eq:main_comp1}
\dntn{\psi}\Big( \bigotimes_{i=1}^r \Big| \prod_{\rho \in \cat{P}_i} \dntn{\rho}^{\otimes m^i_{\rho}} \Big\rangle_{\dntn{w_i}} \Big) = \sum_{\tau \in \cat{Q}} \Big\{ \prod_{i=1}^r \prod_{\rho \in \cat{P}_i} \Big[ \frac{\partial}{\partial x^i_{\rho}} \Big]^{m^i_{\rho}} \Big\} f^\tau_\psi\Big\vert_{\bold{w}} \dntn{\tau}\,.
\ee
where $\bold{w} = (w_1,\ldots,w_r) \in \prod_{i=1}^r k \cat{P}_i$.
\end{corollary}


Returning to the context at the beginning of this section, if we assume $k = \mathbb{C}$ then $F_\psi$ restricts to a smooth (in fact polynomial) morphism of manifolds
\be\label{eq:realF}
F_\psi: \prod_{i=1}^r \mathbb{R} \cat{P}_i \lto \mathbb{R} \cat{Q}\,,
\ee
and we can relate the coalgebraic derivatives to the usual tangent maps of $F_\psi$.

\begin{corollary}\label{cor:diaprimcomp} If $k = \mathbb{C}$ the diagram
\[
\xymatrix@C+2pc{
\bigotimes_{i=1}^r {!} \dntn{A_i} \ar[r]^-{\prom \dntn{\psi}} & {!} \dntn{B}\\
\operatorname{Prim}\!\Big( \!\bigotimes_{i=1}^r {!} \dntn{A_i} \Big) \ar[u]^-{\inc} & \operatorname{Prim}\!\big( {!} \dntn{B} ) \ar[u]_-{\inc}\\
\Big[ \prod_{i=1}^r \dntn{A_i} \Big]^2 \ar[u]^-{\cong} & \dntn{B}^2 \ar[u]_-{\cong}\\
\Big[ \prod_{i=1}^r (\mathbb{R} \cat{P}_i) \Big]^2 \ar[u]^{\dntn{-}^{2r}} \ar[r]_-{\star} & (\mathbb{R} \cat{Q})^2 \ar[u]_-{\dntn{-}^2}
}
\]
commutes, where $\star$ is the map
\be
(\bold{v}, \bold{w}) \longmapsto \big( F_\psi(\bold{v}), T_{\bold{v}}( F_\psi )( \bold{w} ) \big)\,.
\ee
\end{corollary}
\begin{proof}
The image of $(\bold{v}, \bold{w})$ under the left vertical map is the primitive element
\[
\sum_{i=1}^r \vacu_{\dntn{v_1}} \otimes \cdots \otimes \ket{\dntn{w_i}}_{\dntn{v_i}} \otimes \cdots \otimes \vacu_{\dntn{v_r}} \in \bigotimes_{i=1}^r {!} \dntn{A_i}\,.
\]
By \cite[Theorem 5.5]{murfet_ll} and \cite[Remark 2.19]{murfet_coalg} applying $\prom\dntn{\psi}$ yields the primitive element
\[
\sum_{i=1}^r \Big|\dntn{\psi}\big( \vacu_{\dntn{v_1}} \otimes \cdots \otimes \ket{\dntn{w_i}}_{\dntn{v_i}} \otimes \cdots \otimes \vacu_{\dntn{v_r}} \big) \Big\rangle_{\dntn{\psi}\big( \otimes_{i=1}^r \vacu_{\dntn{v_i}} \big)}\,.
\]
If $w_i = \sum_{\rho \in \cat{P}_i} w_i^\rho \rho$ with $w_i^\rho \in \mathbb{R}$ then by Corollary \ref{corollary:summaryderiv} this is equal to
\[
\sum_{i=1}^r \Big| \sum_{\tau \in \cat{Q}} \sum_{\rho \in \cat{P}_i} w^\rho_i \frac{\partial f^\tau_\psi}{\partial x^i_{\rho}}\Big\vert_{\bold{v}} \dntn{\tau} \Big\rangle_{F_\psi(\bold{v})} = \sum_{\tau \in \cat{Q}} \sum_{i=1}^r \sum_{\rho \in \cat{P}_i} w^\rho_i \frac{\partial f^\tau_\psi}{\partial x^i_{\rho}}\Big\vert_{\bold{v}} \big| \dntn{\tau} \big\rangle_{\dntn{F_\psi(\bold{v})}}
\]
which is the image of $( F_\psi(\bold{v}), T_{\bold{v}}( F_\psi )( \bold{w} ) )$ under the right hand vertical map.
\end{proof}


\begin{remark}\label{eq:pathgamma} From a computational perspective, the most natural derivative of $\psi$ is the rate of change associated to this proof by the tangent vector which is parallel to the path from some sequence of proofs $\bold{v} \in \prod_i \cat{P}_i$ to another sequence of proofs $\bold{w} \in \prod_i \cat{P}_i$. we can introduce the path $\gamma$ from $\bold{v} = \gamma(0)$ to $\bold{w} = \gamma(1)$ given by
\begin{gather*}
\gamma: \mathbb{R} \lto \prod_i \mathbb{R} \cat{P}_i\,,\\
\gamma(h) = \bold{v} + h(\bold{w}-\bold{v}) = (1-h) \bold{v} + h \bold{w}\,.
\end{gather*}
With $F_\psi$ the smooth map from \eqref{eq:realF} the image under the tangent map $T_{\bold{v}} F_\psi$ of the tangent vector at $h = 0$ determined by the path $\gamma$ is the vector $\delta$ in $T_{F_\psi(\bold{v})}( \mathbb{R}\cat{Q} ) \cong \mathbb{R} \cat{Q}$ determined by the limit
\begin{align*}
\delta &:= \lim_{h \lto 0} \frac{F_\psi \gamma(h) - F_\psi \gamma(0)}{h}\,.
\end{align*}
This may be computed by a matrix product
\begin{align*}
\delta &= T_{\bold{v}} (F_\psi) T_0(\gamma)( \tfrac{\partial}{\partial h} )\\
&= T_{\bold{v}}(F_\psi)\Big( \sum_{i=1}^r \sum_{\rho \in \cat{P}_i} \frac{\partial}{\partial h}\big[ (1-h) \delta_{v_i = \rho} + h \delta_{w_i = \rho} \big] \frac{\partial}{\partial x^i_\rho} \Big)\\
&= \sum_{i=1}^r \sum_{\rho \in \cat{P}_i} \big[ - \delta_{v_i = \rho} +  \delta_{w_i = \rho} \big]T_{\bold{v}} (F_\psi)\Big( \frac{\partial}{\partial x^i_\rho} \Big)\\
&= \sum_{i=1}^r T_{\bold{v}}(F_\psi)\Big( \frac{\partial}{\partial x^i_{w_i}} - \frac{\partial}{\partial x^i_{v_i}} \Big)\\
&= \sum_{\tau \in \cat{Q}} \sum_{i=1}^r \Big\{ \frac{\partial}{\partial x^i_{w_i}} - \frac{\partial}{\partial x^i_{v_i}} \Big\} f^\tau_\psi\Big\vert_{\bold{v}} \cdot \frac{\partial}{\partial y^\tau}\,.
\end{align*}
Using Corollary \ref{corollary:summaryderiv} (and identifying $\tfrac{\partial}{\partial y^\tau}$ with $\dntn{\tau}$) we may rewrite this tangent vector in coalgebraic language as follows. We begin with the vector
\be\label{eq:partialirho}
\partial^{\,i}_\rho := \vacu_{\dntn{v_1}} \otimes \cdots \otimes \big| \dntn{\rho} \big\rangle_{\dntn{v_i}} \otimes \cdots \otimes \vacu_{\dntn{v_r}} \in \bigotimes_{j=1}^r {!} \dntn{A_j}\,.
\ee
defined for $1 \le i \le r$ and a proof $\rho$ of $A_i$. Then the above shows that
\be\label{eq:defnnu}
\delta = \sum_{i=1}^r \dntn{\psi}\big( \partial^{\,i}_{w_i} - \partial^{\,i}_{v_i} \big)\,.
\ee
\end{remark}

As an important special case, let us record the following:

\begin{corollary}\label{cor: limiting calculations work}
With $k = \mathbb{C}$ let $\psi: {!} A \vdash B$ be a plain proof with associated multiplicity $n$ and linear part $\pi$ as in Setup \ref{setup:psiandpi}, and let $\alpha, \beta$ be proofs of $A$ and $\cat{Q}$ a set of proofs of $B$ such that
\begin{itemize}
\item $\big\{ \pi(X) \l X \in \{ \alpha, \beta \}^n \big\} \subseteq \cat{Q}$, and
\item $\dntn{\cat{Q}}$ is linearly independent in $\dntn{B}$.
\end{itemize}
Then as vectors in $\dntn{B}$ we have
\[
\dntn{\psi}\ket{\dntn{\beta}-\dntn{\alpha}}_{\dntn{\alpha}} = \lim_{h \to 0}\frac{\dntn{\psi}\vacu_{(1-h)\dntn{\alpha} + h \dntn{\beta}} - \dntn{\psi}\vacu_{\dntn{\alpha}}}{h}
\]
where the limit is taken in $\operatorname{span}_{\mathbb{R}} \dntn{\cat{Q}} \cong \mathbb{R} \cat{Q}$.
\end{corollary}

\begin{remark}\label{remark:antipode} The coalgebraic derivatives that are directly available at the level of the syntax of differential linear logic are those of the form $\dntn{\psi}\ket{\dntn{\beta}}_{\dntn{\alpha}}$. This is inconvenient, because the derivatives appearing naturally in the present context are those of the form
\[
\dntn{\psi}\ket{\dntn{\beta}-\dntn{\alpha}}_{\dntn{\alpha}} = \dntn{\psi}\ket{\dntn{\beta}}_{\dntn{\alpha}} - \dntn{\psi}\ket{\dntn{\alpha}}_{\dntn{\alpha}}\,.
\]
One solution would be to axiomitise not just the bialgebra structure of ${!} V$, as is currently done \cite{ehrhard-survey}, but the full Hopf algebra structure \cite[\S 6.4]{sweedler}, since the antipode precisely encodes the necessary minus signs \cite[\S 3.1]{clift_murfet}.
\end{remark}

We have now given a complete treatment of the coalgebraic derivatives of plain proofs (and thus also of component-wise plain proofs by \cite[Remark 2.2]{clift_murfet2}). What we have seen emerge is the role of evaluations of proof denotations at linear combinations such as
\[
\dntn{\psi}\vacu_{(1-h)\dntn{\alpha} + h \dntn{\beta}}\,,
\]
that is, evaluations at probability distributions over proofs. 

While the Curry-Howard correspondence gives a computational meaning to cutting $\psi$ against inputs and reducing the result using cut-elimination, and thus to the evaluation of $\dntn{\psi}$ on vacuums, this does not automatically extend to a computational interpretation of the evaluation of $\dntn{\psi}$ on more general kets such as the ones involved in the coalgebraic derivatives. For example, to give a computational interpretation of the tangent vector $\delta$ of \eqref{eq:defnnu}, we must first have a computational interpretation of the polynomials $f^\tau_\psi$. This will be provided by the probabilistic semantics in the next section.

\section{Naive probability}\label{section:probexec}

We are interested in propagating uncertainty through algorithms in order to give a probabilistic interpretation of the Ehrhard-Regnier derivative. For various reasons, standard probability is not appropriate for this purpose (see Section \ref{section:motivation}) and in this section we explore a natural alternative based on realising algorithms as proofs in linear logic. We define for each plain proof $\psi$ a function $\Delta \psi$ which propagates uncertainty through $\psi$.

Throughout this section $k = \mathbb{C}$. When we study proof denotations in the Sweedler semantics, there are two natural ways to realise probability distributions over proofs. Assume that $0 \leq a_i \leq 1$ are real numbers with $\sum_i a_i = 1$, and interpret $a_1, ..., a_s$ as being a probability distribution over proofs $\alpha_1, ..., \alpha_s : A$. We can encode this distribution as a vector in ${!} \dntn{A}$ using either of the vectors
\be\label{eq:twovectors_dist}
\sum_i a_i \vacu_{\dntn{\alpha_i}} \,,\qquad \vacu_{\sum_i a_i \dntn{\alpha_i}}\,.
\ee
We refer to the former as the \emph{standard encoding} and the latter as the \emph{naive encoding}. If $\psi : {!} A \vdash B$ is a proof and we apply $\dntn{\psi}$ to the vectors in \eqref{eq:twovectors_dist} then the input will be copied using the comultiplication $\Delta$ on ${!} \dntn{A}$, which behaves differently on the two encodings. In the first case $\Delta$ reproduces the same distribution, but as a \emph{distribution over copies}
\be
\Sum_i a_i \vacu_{\dntn{\alpha_i}} \xmapsto{\quad\Delta\quad}\Sum_i a_i \vacu_{\dntn{\alpha_i}} \otimes \vacu_{\dntn{\alpha_i}},
\ee
while in the second case $\Delta$ \emph{copies the distribution}
\be
\vacu_{\sum_i a_i \dntn{\alpha_i}} \xmapsto{\quad\Delta\quad}\vacu_{\sum_i a_i \dntn{\alpha_i}} \otimes \vacu_{\sum_i a_i \dntn{\alpha_i}}\,.
\ee
We will show that the naive encoding gives rise to a probabilistic semantics of component-wise plain proofs, which we call \emph{naive probability}, whereas the standard encoding gives rise to standard probability (see Appendix \ref{section:stdprob}).
\\

For the next definition recall the standard $n$-simplex
\[
\Delta^n = \big\{ (x_0,\ldots,x_{n}) \in \mathbb{R}^{n+1} \l \sum_{i=0}^n x_i = 1 \text{ and } x_i \ge 0 \text{ for all $i$ } \big\}\,.
\]
More generally, given a set $Z$ we write
\[
\Delta Z = \Big\{ \sum_{z \in Z} \lambda_z z \in \mathbb{R} Z \l \sum_z \lambda_z = 1 \text{ and } \lambda_z \ge 0 \text{ for all } z \in Z \Big\}\,,
\]
which is the image of the standard $|Z|$-simplex under the isomorphism $\mathbb{R}^{|Z|} \cong \mathbb{R} Z$ induced by the basis $Z$ in the case this set is finite. Here $\mathbb{R} Z$ is the free vector space on $Z$, and in particular for an element of $\Delta Z$ only finitely many of the coefficients $\lambda_z$ are nonzero. There is a canonical inclusion of $Z$ as the vertices of the simplex
\begin{gather*}
\xymatrix@C+2pc{
Z \ar[r]^-{\inc} & \Delta Z
}\,\\
z \mapsto \sum_{z' \in Z} \delta_{z = z'} z'
\end{gather*}
and we usually identify $Z$ with its image as a subset of $\Delta Z$. Given a set of proofs $\cat{P}$ of a formula $A$ we define using \eqref{eq:denotesec} the denotation $\dntn{v}$ of a probability distribution $v$ over $\cat{P}$ to be its image under the map
\begin{gather*}
\xymatrix@C+2pc{
\Delta \cat{P} \ar[r]^-{\inc} & \mathbb{R} \cat{P} \ar[r]^-{\dntn{-}} & \dntn{A}
}\\
\sum_{\rho \in \cat{P}} v_\rho \cdot \rho \longmapsto \sum_{\rho \in \cat{P}} v_\rho \dntn{\rho}\,.
\end{gather*}


\begin{definition}\label{defn:iota_notation} Given $A_1,\ldots,A_r$ and for each $i$ a finite set of proofs $\cat{P}_i$ of $A_i$, define
\begin{gather*}
\iota: \Delta \cat{P}_1 \times \cdots \times \Delta \cat{P}_r \lto {!} \dntn{A_1} \otimes \cdots \otimes {!} \dntn{A_r}\\
\iota\big( v_1,\ldots,v_r\big) = \vacu_{\dntn{v_1}} \otimes \cdots \otimes \vacu_{\dntn{v_r}}\,.
\end{gather*}
\end{definition}


\begin{lemma}\label{lemma:injectiveiota} If $\{ \dntn{\rho} \}_{\rho \in \cat{P}_i}$ is linearly independent in $\dntn{A_i}$ for $1 \le i \le r$ then $\iota$ is injective.
\end{lemma}
\begin{proof} Given an index $1 \le i \le r$ and a vector (where $v_i \in \dntn{A_i}$)
\[
\vacu_{v_1} \otimes \cdots \otimes \vacu_{v_r} \in \bigotimes_{i=1}^r {!} \dntn{A_i}
\]
we may apply the counit ${!} \dntn{A_j} \lto k$ for $j \neq i$ and the dereliction $d: {!} \dntn{A_i} \lto \dntn{A_i}$ to recover $v_i$. From this and the hypothesis of linear independence the claim is clear.
\end{proof}

\begin{proposition}\label{prop:plain_prob} Suppose given a plain proof
\[
\psi: {!} A_1, \ldots, {!} A_r \vdash B
\]
constructed from $\pi$ as in Setup \ref{setup:psiandpi}, together with for $1 \le i \le r$ a finite set $\cat{P}_i$ of proofs of $A_i$, and a finite set of proofs $\cat{Q}$ of $B$ such that the indexed set $\{ \dntn{\nu} \}_{\nu \in \cat{Q}}$ is linearly independent in $\dntn{B}$ and the constraint \eqref{eq:coll_comput} is satisfied. Then there is a unique function
\[
\Delta \psi: \Delta \cat{P}_1 \times \cdots \times \Delta \cat{P}_r \lto \Delta \cat{Q}
\]
which makes the diagram
\[
\xymatrix@C+2.5pc@R+1pc{
{!} \dntn{A_1} \otimes \cdots \otimes {!} \dntn{A_r} \ar[r]^-{\dntn{\psi}} & \dntn{B}\\
\Delta \cat{P}_1 \times \cdots \times \Delta \cat{P}_r\ar[u]^-{\iota} \ar[r]_-{\Delta \psi} & \Delta \cat{Q} \ar[u]_-{\dntn{-}}
}
\]
commute.
\end{proposition}
\begin{proof}
By hypothesis the right-hand vertical map is injective (we do not need the other $\iota$ to be injective) so it is only necessary to show that $\dntn{\psi}$ restricted to sequences of probability distributions factors through $\Delta \cat{Q}$. For this we need only check that the polynomial function $F_\psi$ of Proposition \ref{prop:fpi} satisfies
\[
F_\psi\big( \Delta \cat{P}_1 \times \cdots \times \Delta \cat{P}_r \big) \subseteq \Delta \cat{Q}\,,
\]
from which we deduce $\Delta \psi$ is the restriction of $F_\psi$ to this domain. This follows directly from the formula \eqref{eq:formulaforfpitau} since
\[
\sum_\tau f^\tau_\psi = \sum_{\gamma_1,\ldots,\gamma_r} \prod_{i=1}^r \prod_{j=1}^{n_i} x^i_{\gamma_i(j)} = \prod_{i=1}^r \Big( \sum_{\rho \in \cat{P}_i} x^i_{\rho} \Big)^{n_i} = 1
\]
when $\sum_{\rho \in \cat{P}_i} x^i_{\rho} = 1$ for all $1 \le i \le r$.
\end{proof}

\begin{definition}\label{defn:prob_extend} We call $\Delta \psi$ the \emph{naive probabilistic extension} of $\psi$.
\end{definition}

\begin{remark} It follows from the definition of $\Delta \psi$ that the diagram
\[
\xymatrix@C+2pc@R+1pc{
\Delta \cat{P}_1 \times \cdots \times \Delta \cat{P}_r \ar[r]^-{\Delta \psi} & \Delta \cat{Q}\\
\cat{P}_1 \times \cdots \times \cat{P}_r \ar[u]^-{\inc} \ar[r]_-{\psi} & \cat{Q} \ar[u]_-{\inc}
}
\]
also commutes.
\end{remark}



Given a component-wise plain proof
\[
\psi: {!} A_1, \ldots, {!} A_r \vdash {!} B_1 \otimes \cdots \otimes {!} B_s
\]
the naive probabilistic extension is a function
\[
\Delta \psi: \Delta \cat{P}_1 \times \cdots \times \Delta \cat{P}_r \lto \Delta \cat{Q}_1 \times \cdots \times \Delta \cat{Q}_s
\]
depending on choices of sets of proofs $\cat{P}_i$ and $\cat{Q}_j$, which is defined in the same way as the naive probabilistic extension of a plain proof above (see Appendix \ref{section:appnv} for details).

\begin{remark} For the definition of the naive probabilistic extension to be applicable, we need a supply of proofs with linearly independent denotations. Integers and booleans are always linearly independent \cite[Proposition B.1]{clift_murfet2} and for any finite set of binary integers there is a type $A$ such that, when the binary integers are encoded as a set of proofs of type $\tBint_A$, their denotations are linearly independent; see \cite[Remark B.11]{clift_murfet2}.
\end{remark}

\subsection{Interpretation of naive probability}\label{section:interp_naive}

Given a plain proof $\psi$ with linear part $\pi$ as in Setup \ref{setup:psiandpi} and distributions $v_1,\ldots,v_r$ over proofs of $A_1,\ldots,A_r$, how are we to interpret the distribution $\Delta \psi(v_1,\ldots,v_r)$? The answer depends on our prior commitments to interpreting probabilities in general, so we will give both a \emph{frequentist} and a \emph{Bayesian} answer.
\\

\textbf{The frequentist view} revolves around an operational semantics of the algorithm $\psi$ and sampling from the various distributions. It suffices to explain in the case where $r = 1$ so $\psi : {!} A \vdash B$. Let $\alpha, \beta \in \cat{P}$ be proofs of $A$ and set
\begin{gather}
v = (1-h) \cdot \alpha + h \cdot \beta \in \Delta \cat{P}
\end{gather}
Let $n$ be the $A$-degree as in Setup \ref{setup:psiandpi}, let $\gamma$ range over functions $\gamma: \{1,\ldots,n\} \lto \{ 0,1 \}$ and set $\gamma_i = |\gamma^{-1}(i)|$ for $i \in \{0,1\}$. We write $d$ for the dereliction. Then
\begin{align*}
\dntn{\psi}\ket{\emptyset}_{\dntn{v} } &= \llbracket \psi \rrbracket| \emptyset \rangle_{(1-h) \llbracket \alpha \rrbracket + h \llbracket \beta \rrbracket}\\
&= \dntn{\pi} d^{\otimes n} \Delta^{n-1} | \emptyset \rangle_{(1-h) \llbracket \alpha \rrbracket + h \llbracket \beta \rrbracket}\\
&= \llbracket \pi \rrbracket\Big( (1-h) \llbracket \alpha \rrbracket + h \llbracket \beta \rrbracket \Big)^{\otimes n}\\
&= \sum_{\gamma} (1-h)^{\gamma_0} h^{\gamma_1} \llbracket \pi \rrbracket\Big( \llbracket \zeta^1_\gamma \rrbracket \otimes \cdots \otimes \llbracket \zeta^n_\gamma \rrbracket \Big)\\
&= \sum_{\gamma} (1-h)^{\gamma_0} h^{\gamma_1} \dntn{ \pi( \zeta^1_\gamma, \ldots, \zeta^n_\gamma ) }
\end{align*}
where $\zeta^i_\gamma$ is $\alpha$ if $\gamma(i) = 0$ and $\beta$ otherwise. Under the linear independence hypothesis of Setup \ref{setup:psiandpi} this result may be lifted uniquely to a distribution, with the result that
\[
\Delta \psi(v) = \sum_{\gamma} (1-h)^{\gamma_0} h^{\gamma_1} \cdot \pi( \zeta^1_\gamma, \ldots, \zeta^n_\gamma )
\]
This is a probability distribution over proofs of $B$ with the following interpretation: to evaluate $\psi$ on $v$ is to execute the algorithm probabilistically, such that every time the input is to be used, we sample from the distribution $v$ and obtain $\alpha$ with probability $1-h$ and $\beta$ with probability $h$. The expected distribution over the outputs, given an infinite number of runs, is then $\Delta \psi(v)$.
\vspace{0.2cm}

In the next two examples we present proofs using the term calculus of \cite{benton_etal}.

\begin{example} The Church numeral (with $q:{!}(A \multimap A), z: A$)
\be\label{eq:church_2_linearlambda}
\underline{2} = (\lambda q.(\textrm{copy $q$ as $q',q''$ in } (\lambda z. (\derelict(q'') \; (\derelict(q')\; z)))))
\ee
is a plain proof of $\tInt_A$, and to find its naive probabilistic extension we may take an input distribution $q$ and replace each of $\derelict(q'), \derelict(q'')$ with independent samples.
\end{example}

\begin{example} The binary integer
\[
\underline{001} = (\lambda r.(\lambda p.(\textrm{copy $r$ as $r',r''$ in } (\lambda z.(\derelict(p)\;(\derelict(r'')\;(\derelict(r')\;z)))))))
\]
is a plain proof of $\tBint_A$ \cite[\S 4.2]{clift_murfet}. The naive probabilistic extension of this proof maps a pair of distributions $r,p$ to the distribution induced by sampling twice independently from $r$ and once from $p$, and composing the sampled linear operators.
\end{example}

The operational content of the naive probabilistic extension is captured by the slogan: \emph{dereliction as sampling}. However, this needs to be understood with some care when we begin with a plain proof which is not cut-free (in which case we may not know explicitly how to present $\psi$ in terms of a linear part $\pi$). Notice that if we were to promote $\underline{2}$ (viewed as a proof of ${!}(A \multimap A) \vdash A \multimap A$) and cut its output against the first input of $\underline{001}$ (viewed as a proof of $2\,{!}(A \multimap A) \vdash A \multimap A$) we would obtain a plain proof \cite[Lemma 3.4]{clift_murfet2}.

The naive probabilistic execution of this proof can be described as follows: at each of the two points in $\underline{001}$ that $r$ is used (that is, derelicted), we independently sample the input $q$ to $\underline{2}$ twice so as to compute a single sample from the distribution $r$. Cut-elimination will push these calculations to the beginning, so that we would begin by independently sampling the first input four times.
\\

\textbf{The Bayesian view} revolves around the idea of an observer of the computational process of cut-elimination, who is reasoning about the current state of this process given some initial uncertainty. However, the observer reasons under the hypothesis that in any part of the algorithm where two or more outputs are computed from the same inputs, those outputs are \emph{conditionally independent} given the inputs. We will develop this point of view properly in the special case of Turing machines, so here we will be brief. Let us suppose for simplicity that $\psi$ is a component-wise plain proof of ${!} A \vdash {!} B_1 \otimes {!} B_2$. Then $\dntn{\psi}$ can be assembled from its components $\psi_i$ as (see \cite[Lemma 2.1]{clift_murfet2})
\[
\xymatrix@C+2pc{
{!} \dntn{A} \ar[r]^-{\Delta} & {!} \dntn{A} \otimes {!}\dntn{A} \ar[r]^-{\dntn{\psi_1} \otimes \dntn{\psi_2}} & {!} \dntn{B_1} \otimes {!} \dntn{B_2}
}
\]
which has the semantics, for $v \in \Delta \cat{P}$, $\cat{P}$ a set of proofs of $A$,
\[
\vacu_v \longmapsto \vacu_v \otimes \vacu_v \longmapsto \vacu_{\Delta\psi_1(v)} \otimes \vacu_{\Delta \psi_2(v)}\,.
\]
From this we see that the components of $\Delta \psi(v)$, even though they depend on the same input variable of type ${!} A$, are conditionally independent distributions. It is this conditional independence, which is reminiscent of the naive Bayesian classifier, that gives its name to naive probability. We refer to such an observer as a \emph{naive Bayesian observer}.
\\

In conclusion, we have now given a computational interpretation to the polynomial $f^\tau_\pi$ defined in Proposition \ref{prop:fpi}. This polynomial computes (in the frequentist view) the probability of the output being the proof $\tau$, when $\psi$ is executed according to the probabilistic semantics described above and (in the Bayesian view) the probability assigned by a naive Bayesian observer to the output being $\tau$. We call this the \emph{naive probability}.\footnote{A reasonable objection to the term \emph{naive probability} is that there is, by Cox's theorem \cite{coxpaper}, \cite[Ch.1]{jaynes} essentially only one way of calculating probabilities and it is therefore an abuse to term anything else ``probability''. Against this view, we observe that Cox's uniqueness theorem is derived from the fact that propositions form a Boolean algebra \cite[p.13]{coxpaper} and the argument does not apply if the underlying logic of propositions is changed, as is being implicitly done here; a similar point is made in \cite[p.549]{nguyen}.}

\subsection{Probability and derivatives}\label{section:probderive}

We now proceed to give our interpretation of the output values of the derivative of an algorithm, when that algorithm is encoded as a component-wise plain proof in linear logic. With $k = \mathbb{C}$ and $\psi$ a plain proof satisfying the hypotheses of Setup \ref{setup:psiandpi}, we have defined the naive probabilistic extension
\[
\Delta \psi: \prod_{i=1}^r \Delta \cat{P}_i \lto \Delta \cat{Q}\,.
\]
We view $\Delta \psi$ as providing a way of propagating uncertainty\footnote{We have given a frequentist interpretation of naive probability in Section \ref{section:interp_naive}, however in the present context this point of view is awkward, as we have to imagine making an infinitesimal change in input probabilities and then, after making an infinite number of trials, looking for an infinitesimal change in output probabilities. It seems more natural to consider an observer smoothly \emph{revising} their state of belief about the output of an algorithm given a change in their knowledge about the input. We therefore adopt a Bayesian perspective in this section.} about the inputs to uncertainty about the outputs. We now turn to the following question: given some infinitesimal change in one of the input distributions, what is the corresponding infinitesimal change in the output distribution? We give a \emph{syntactic} view, a \emph{coalgebraic} view, and an \emph{ordinary calculus} view on this question. These are all shown to be equivalent by Theorem \ref{theorem:maind}.

\vspace{0.2cm}

We begin with the calculus view. By construction $\Delta \psi$ is a smooth map of manifolds with corners \cite{joyce} as it is a restriction of the polynomial map $F_\psi$, and its tangent map at $\bold{w} = (w_1,\ldots,w_r) \in \prod_{i=1}^r \Delta \cat{P}_i$ is a linear map
\be\label{eq:tangentnvprob}
T_{\bold{w}}\big( \Delta \psi \big): T_{\bold{w}}\Big( \prod_i \Delta \cat{P}_i \Big) \lto T_{\Delta \psi(\bold{w})}( \Delta \cat{Q} )\,.
\ee
This function computes the infinitesimal change in the output \emph{from $\Delta \psi(\bold{w})$} that is induced by a given infinitesimal change in the input \emph{from $\bold{w}$}. These represent infinitesimal changes in the state (of belief) of the naive Bayesian observer.

Since this tangent map is the restriction of the tangent map of $F_\psi$, it is a consequence of Corollary \ref{cor:diaprimcomp} that the coalgebraic derivatives of $\dntn{\psi}$ (that is, the values of the promotion of this map on primitive elements) compute the tangent map \eqref{eq:tangentnvprob}. In this section we make this connection between derivatives and naive probability explicit. The first step to doing so is to compute the tangent spaces to the $\Delta \cat{P}_i$ and $\Delta \cat{Q}$.

For a finite set of proofs $\cat{P}$ of $A$ consider the affine subspace
\[
\widetilde{\Delta} \cat{P} = \Big\{ \sum_{\rho \in \cat{P}} \lambda_\rho \rho \in \mathbb{R} \cat{P} \,\Big\vert\, \sum_\rho \lambda_\rho = 1 \Big\}\,,
\]
which satisfies $\Delta \cat{P} \subseteq \widetilde{\Delta} \cat{P} \subseteq \mathbb{R} \cat{P}$. The space $\Delta \cat{P}$ is a manifold with corners \cite{joyce} and as such its tangent space is defined to be the tangent space of the affine space $\widetilde{\Delta} \cat{P}$, so
\be\label{eq:tangentspace_corners}
T_{\bold{w}}\big( \Delta \cat{P} ) := T_{\bold{w}} \big( \widetilde{\Delta} \cat{P} \big) \subseteq T_{\bold{w}}\big( \mathbb{R} \cat{P} \big)\,.
\ee
If we write $\{ x_\rho \}_{\rho \in \cat{P}}$ for the coordinates of $\mathbb{R} \cat{P}$ then $T_{\bold{w}}( \mathbb{R} \cat{P} )$ is spanned by the $\frac{\partial}{\partial x_\rho}$, and

\begin{lemma}\label{lemma:basistangent} For any pair $\zeta, \rho \in \cat{P}$ the tangent vector
\be
B^{\,\zeta}_\rho := \frac{\partial}{\partial x_\zeta} - \frac{\partial}{\partial x_\rho}
\ee
belongs to $T_{\bold{w}}(\Delta \cat{P})$ and for any $\rho \in \cat{P}$, the set $\big\{ B^\zeta_\rho \big\}_{\zeta \neq \rho}$ is a basis for $T_{\bold{w}} \big( \Delta \cat{P} )$.
\end{lemma}

The infinitesimal line segment corresponding to $B^{\,\zeta}_\rho$ on the space $\Delta \cat{P}$ of distributions starts from $\bold{w}$ and moves in the direction which infinitesimally decreases the probability mass for $\rho$ and increases it for $\zeta$, and leaves all other masses fixed. This path represents a \emph{revision of belief}, and the basis of the lemma represents the independent directions in which the observer's beliefs can be revised.


Arbitrarily choosing $\rho_i \in \cat{P}_i$ for each $i$ and $\tau \in \cat{Q}$,
\[
T_{\bold{w}}\Big( \prod_{i=1}^r \Delta \cat{P}_i \Big) \cong \bigoplus_{i=1}^r T_{w_i} \Delta \cat{P}_i \cong \bigoplus_{i=1}^r \bigoplus_{\zeta \neq \rho_i} \mathbb{R} B^{\zeta}_{\rho_i}
\]
and there is a commutative diagram
\[
\xymatrix@C+2pc{
T_{\bold{w}}\Big( \prod_i \Delta \cat{P}_i \Big) \ar[d]_-{\cong} \ar[r]^-{T_{\bold{w}} (\Delta \psi)} & T_{\Delta \psi(\bold{w})}\big( \Delta \cat{Q} \big) \ar[d]^-{\cong}\\
\bigoplus_{i=1}^r \bigoplus_{\zeta \neq \rho_i} \mathbb{R} B^{\zeta}_{\rho_i} \ar[r]_-{ \star } & \bigoplus_{\xi \neq \tau} \mathbb{R} B^\xi_\tau
}
\]
where $\star$ is, by definition of the tangent map, given by
\[
B^{\zeta}_{\rho_i} \longmapsto \sum_{\xi \neq \tau} \Big\{ \frac{\partial}{\partial x^i_{\zeta}} - \frac{\partial}{\partial x^i_{\rho_i}} \Big\} f^\xi_\psi\Big\vert_{\bold{w}} \cdot B^\xi_\tau
\]
The coalgebraic point of view is centered on the corresponding primitive element
\[
\vacu_{\dntn{w_1}} \otimes \cdots \otimes \big| \dntn{\zeta} - \dntn{\rho_i} \big\rangle_{\dntn{w_i}} \otimes \cdots \otimes \vacu_{\dntn{w_r}} \in \operatorname{Prim}\!\Big(\! \bigotimes_{i=1}^r {!} \dntn{A_i} \Big)\,.
\]
The punchline of this section is the comparison given in the next theorem, between the coalgebraic derivatives of $\dntn{\psi}$ on the one hand (which are the denotations of the syntactic derivatives axiomitised by differential linear logic) and the derivatives of the naive probability $\Delta \psi$, as encapsulated in its tangent map.

\begin{theorem}\label{theorem:maind} For any point $\bold{w} \in \prod_{i=1}^r \Delta \cat{P}_i$, and proofs $\zeta, \rho \in \cat{P}_i$,
\be\label{eq:theoremmaind}
\dntn{\psi}\Big( \vacu_{\dntn{w_1}} \otimes \cdots \otimes \big| \dntn{\zeta} - \dntn{\rho} \big\rangle_{\dntn{w_i}} \otimes \cdots \otimes \vacu_{\dntn{w_r}} \Big) = T_{\bold{w}}( \Delta \psi )( B^\zeta_{\rho} )
\ee
where we identify the right hand side as a vector in $\dntn{B}$ using the inclusions
\begin{gather*}
T_{\Delta \psi(\bold{w})}(\Delta \cat{Q}) \hookrightarrow T_{\Delta \psi(\bold{w})}( \mathbb{R} \cat{Q} ) \cong \mathbb{R} \cat{Q} \hookrightarrow \dntn{B}\,, \label{eq:someinclusions}\\
B^{\tau}_{\theta} \longmapsto \dntn{\tau} - \dntn{\theta}\,.
\end{gather*}
\end{theorem}
\begin{proof}
Let $\mu: \Delta \cat{Q} \lto \mathbb{R} \cat{Q}$ be the inclusion and observe that
\begin{align*}
T_{\Delta \psi(\bold{w})}(\mu) \circ T_{\bold{w}}(\Delta \psi)(B^\zeta_\rho) &= T_{\bold{w}}( \mu \circ \Delta \psi )(B^\zeta_\rho)\\
&= \sum_{\tau \in \cat{Q}} \Big\{ \frac{\partial}{\partial x^i_\zeta} - \frac{\partial}{\partial x^i_\rho}\Big\} (\mu \Delta \psi)_\tau\Big\vert_{\bold{w}} \tau\\
&= \sum_{\tau \in \cat{Q}} \Big\{ \frac{\partial}{\partial x^i_\zeta} - \frac{\partial}{\partial x^i_\rho}\Big\} f^\tau_\psi\Big\vert_{\bold{w}} \tau
\end{align*}
Sending $\tau$ to $\dntn{\tau}$ this agrees with the left hand side of \eqref{eq:theoremmaind} by Corollary \ref{corollary:summaryderiv}.
\end{proof}

The syntactic view on the question of infinitesimal changes is to take the proof $\psi$ in linear logic and form the proof $\frac{\partial}{\partial X_i} \psi$ in differential linear logic, given by the tree
\be\label{eq:derivative_tree}
\begin{mathprooftree}
\AxiomC{$\psi$}
\noLine\UnaryInfC{$\vdots$}
\def\extraVskip{5pt}
\noLine\UnaryInfC{${!} A_1, \ldots, {!} A_r \vdash B$}
\RightLabel{\scriptsize $\partial_i$}
\UnaryInfC{$A_i, {!} A_1, \ldots, {!} A_r \vdash B$}
\end{mathprooftree}
\ee
where $\partial_i$ is the new deduction rule of differential linear logic; see \cite{ehrhard-survey} and \cite[\S 3.1]{clift_murfet}. The comparison of the coalgebraic and calculus derivatives in \eqref{eq:theoremmaind} with the syntactic derivative \eqref{eq:derivative_tree} is complicated by the fact that we cannot express $w_i$ directly in the syntax. It is true by definition that the left hand side equals
\[
\dntn{\tfrac{\partial}{\partial X_i} \psi}\big( (\dntn{\zeta} - \dntn{\rho}) \otimes \vacu_{\dntn{w_1}} \otimes \cdots \otimes \vacu_{\dntn{w_r}} \big)
\]
but in the case where $w_i$ is actually a proof of $A_i$ and not just a general distribution over proofs, we can say more: there is a further equality of the vectors in  \eqref{eq:theoremmaind} with
\[
\dntn{\tfrac{\partial}{\partial X_i} \psi( \zeta, w_1,\ldots, w_r) } - \dntn{\tfrac{\partial}{\partial X_i} \psi( \rho, w_1,\ldots,w_r) }\,.
\]


\begin{example} Returning to the situation of Remark \ref{eq:pathgamma} let us take sequences of proofs $\bold{v}, \bold{w}$ in $\prod_{i=1}^r \cat{P}_i \subseteq \prod_{i=1}^r \Delta \cat{P}_i$. The tangent vector to the straight line from $\bold{v}$ to $\bold{w}$ is
\[
\delta := \sum_{i=1}^r B^{w_i}_{v_i} \in T_{\bold{v}}\Big( \prod_i \Delta \cat{P}_i \Big)
\]
which is encoded coalgebraically in the notation of \eqref{eq:partialirho} as the primitive element
\[
\sum_{i=1}^r\big( \partial^{\,i}_{w_i} - \partial^{\,i}_{v_i} \big) \in \bigotimes_{i=1}^r {!} \dntn{A_i}\,.
\]
The theorem then states that the vectors
\begin{gather*}
\sum_{\tau \in \cat{Q}} \sum_{i=1}^r \Big\{ \frac{\partial}{\partial x^i_{w_i}} - \frac{\partial}{\partial x^i_{v_i}} \Big\} f^\tau_\psi\Big\vert_{\bold{v}} \dntn{\tau}\,,\\
\sum_{i=1}^r \Big( \dntn{\tfrac{\partial}{\partial X_i} \psi( w_i, v_1, \ldots,v_r) } - \dntn{\tfrac{\partial}{\partial X_i} \psi( v_i, v_1, \ldots, v_r) } \Big)\,,\\
\sum_{i=1}^r \dntn{\psi}\big( \partial^{\,i}_{w_i} - \partial^{\,i}_{v_i} \big)
\end{gather*}
in $\dntn{B}$ are equal.
\end{example}


\subsection{Non-determinism}\label{section:non-det}

The formal sum of terms plays a central role in differential lambda calculus. The standard interpretation of the sum \cite[p.3]{difflambda}, \cite[p.2]{ehrhard-survey}, \cite[p.3]{ehrhard-kothe} is that it is a form of \emph{non-determinism}:
\begin{quote}
``It is worth noting that linear substitution is a non-deterministic operation, as soon as the substituted variable has several occurrences: one has to choose a linear occurrence of the variable to be substituted and there are several possible such choices. This fundamental non-determinism of the differential lambda-calculus might be an evidence of a link with process calculi...'' \cite[p.3]{difflambda}
\end{quote}
Here the linear substitution that is being referred to arises in the definition of the reduction rule of differential lambda calculus and the cut-elimination transformations of differential linear logic \cite[\S 1.4.3, \S 1.5.3]{ehrhard-survey}. We can make the point semantically as follows: in the setting of Corollary \ref{cor: limiting calculations work} we can easily calculate as in \cite[\S 3.1]{clift_murfet2} that
\be\label{eq:sumoflinearsub}
\dntn{\psi}\ket{\dntn{\beta}}_{\dntn{\alpha}} = \dntn{ \pi( \beta, \alpha, \ldots, \alpha ) } + \dntn{ \pi( \alpha, \beta, \alpha, \ldots, \alpha ) } + \cdots + \dntn{ \pi( \alpha, \ldots, \alpha, \beta ) }\,.
\ee
The left hand side is the denotation of the derivative $\frac{\partial}{\partial X} \psi$ as a proof in differential linear logic cut against $\beta, \alpha$, while the right hand side is the denotation of a formal sum of cuts of the linear part $\pi$ of $\psi$ against various inputs. It is this formal sum that is being referred to in the quote above as the linear substitution of $\beta$ for $\alpha$ in $\psi$, and the equality expresses the connection between derivatives and linear substitution.

For context let us recall that a \emph{non-deterministic Turing machine} $M$ \cite[\S 2.1.2]{arorabarak} has two transition functions and returns $M(x) = 1$ on an input $x$ if some sequence of choices of which transition function to apply in each step leads to the machine halting in a special state $q_{\text{accept}}$, while a \emph{probabilistic Turing machine} \cite[\S 7.1]{arorabarak} is syntactically the same, in the sense that it has two transition functions, but the output is viewed as a random variable $M(x)$ where $\prob( M(x) = 1 )$ is the fraction of the sequences of choices which lead to $q_{\text{accept}}$. Note that non-deterministic and probabilistic Turing machines are conceptually quite distinct, despite their technical similarity; see \cite[p.125]{arorabarak}.

According to this definition the non-deterministic point of view on \eqref{eq:sumoflinearsub} would emphasise, for each proof $\tau$ of $B$, the \emph{predicate} determining whether $\tau$ appears in the list
\[
\pi(\beta, \alpha, \ldots, \alpha)\,, \pi( \alpha, \beta, \alpha, \ldots, \alpha )\,, \ldots\,, \pi(\alpha, \ldots, \alpha, \beta)\,.
\]
However, the derivative $\frac{\partial}{\partial X} \psi(\beta, \alpha)$ contains more information than this predicate: it also knows the multiplicity with which each $\tau$ appears. The problem is to give a computational account for these multiplicities, which are, under the linear independence hypotheses we have adopted, precisely the coefficients in \eqref{eq:sumoflinearsub}. While in principle we could scale \eqref{eq:sumoflinearsub} appropriately and obtain a probability distribution, this does not seem very natural.
\\

In the Sweedler semantics it is natural to view the multiplicities in \eqref{eq:sumoflinearsub} as encoding the values of \emph{derivatives} of probabilities, rather than the probabilities themselves. To be precise, as in \cite[\S 3.1]{clift_murfet2} we may extract from the proof $\psi$ a family of polynomials $f^\tau_\psi$ in variables $x_\alpha, x_\beta$, namely
\be
\sum_{(\gamma_1,\ldots,\gamma_n) \in \{\alpha, \beta\}^n} \delta_{\tau = \pi(\gamma_1,\ldots,\gamma_n)} x_{\gamma_1} \cdots x_{\gamma_n}
\ee
so that by Corollary \ref{corollary:summaryderiv},
\begin{align*}
\dntn{\psi}\vacu_{\dntn{\alpha}} &= \sum_{\tau \in \cat{Q}} f^\tau_\psi\Big\vert_{x_\alpha = 1, x_\beta = 0} \dntn{\tau}\\
\dntn{\psi} \ket{\dntn{\beta}}_{\dntn{\alpha}} &= \sum_{\tau \in \cat{Q}} \frac{\partial f^\tau_\psi}{\partial x_\beta} \Big\vert_{x_\alpha = 1, x_\beta = 0} \dntn{\tau}\,.
\end{align*}
The connection to probabilistic computation comes from Section \ref{section:prob_ext_tm} which explains how $f^\tau_\psi$ is a component of the naive probabilistic extension $\Delta \psi$, in which role it computes the naive probability of observing the output $\tau$.\footnote{We note, however, that from this point of view it is the quantity $\dntn{\psi}\ket{\dntn{\beta} - \dntn{\alpha}}_{\dntn{\alpha}}$ which arises most naturally, rather than $\dntn{\psi}\ket{\dntn{\beta}}_{\dntn{\alpha}}$. See Remark \ref{remark:antipode}}

\section{Turing machines} \label{section: turing machines}

What does the Ehrhard-Regnier derivative of an algorithm compute? When \emph{algorithm} is interpreted as \emph{component-wise plain proof in linear logic} and the derivative is understood in the sense of Ehrhard-Regnier \cite{difflambda}, the general answer given in Section \ref{section:probderive} was that the derivative computes the rate of change of naive probability. However, this answer is not fully satisfactory, because proofs in linear logic (or lambda terms) are not a very intuitive model of algorithms. For this reason we explain in this section how to encode Turing machines in linear logic and then take their derivatives.


\subsection{Naive probability and Turing machines}\label{section:prob_ext_tm}

We recall the definition of a Turing machine to fix our conventions; see \cite{arorabarak}. Informally speaking, a Turing machine is a computer which possesses a finite number of internal states $Q$, and a one dimensional `tape' as memory. Our tapes are unbounded in both directions. The tape is divided into individual squares each of which contains some symbol from a fixed finite alphabet $\Sigma$. At any instant only one square is being read by the `tape head'. Depending on the symbol on this square and the current internal state, the machine will write a symbol from $\Sigma$ to the square under the tape head, possibly change the internal state, and then move the tape head either left or right.

\begin{definition} \label{defn: turing machine}
A \emph{Turing machine} $M = (\Sigma, Q, \delta)$ is a tuple where $Q$ is a finite set of states, $\Sigma$ is a finite set of symbols called the \emph{tape alphabet}, and \[\delta: \Sigma \times Q \to \Sigma \times Q \times \{\text{left, right}\}\] is a function, called the \emph{transition function}. For $i \in \{0,1,2\}$, we write $\delta_i = \proj_i \circ\, \delta$ for the $i$th component of $\delta$
\end{definition}

The set $\Sigma$ is assumed to contain some designated blank symbol $\Box$ which is the only symbol that is allowed to occur infinitely often on the tape. Often one also designates a starting state, as well as a special accept state which terminates computation if reached. The \emph{Turing configuration} of a Turing machine $M$ is an element $\big( (\sigma_u)_{u \in \mathbb{Z}}, q )$ of $\Sigma^{\mathbb{Z}} \times Q$ where $q$ is the current state and the symbol in the square in position $u$ relative to the head is $\sigma_u$ (so $\sigma_0$ is the symbol currently under the tape head, $\sigma_{-1}$ is the symbol immediately to the left of the tape head). Observe that the configuration of the tape actually lies in the smaller set of functions which are finitely supported, in the following sense:

\begin{definition}\label{defn:finsupstate} We write
\[
\Sigma^{\mathbb{Z}, \Box} = \{ f: \mathbb{Z} \lto \Sigma \l f(u) = \Box \text{ except for finitely many $u$} \}\,.
\]
\end{definition}

Associated to $\delta$ is the \emph{step function} or time evolution
\be\label{eq:stepfunction}
\textrm{step}: \Sigma^{\mathbb{Z}, \Box} \times Q \lto \Sigma^{\mathbb{Z},\Box} \times Q
\ee
where if $C = \big((\sigma_u)_{u \in \mathbb{Z}}, q\big)$ is the current configuration of the Turing machine then after one time step it will be in configuration $\textrm{step}( C)$. Our aim is to define the naive probabilistic extension of this step function. This is complicated by the fact that we do not give a \emph{single} proof encoding the step function, rather we give a \emph{family} of proofs which collectively give an encoding. These proofs each have a naive probabilistic extension, which piece together to define a function $\Delta \textrm{step}$. This is what we call the naive probabilistic extension of $\textrm{step}$.


The encoding we will use is taken from \cite[\S 5.1]{clift_murfet2}. The naive probabilistic extension of the function $\textrm{step}$ that we obtain depends in principle on the choice of encoding. However, we show in Appendix \ref{section:smooth_model} that the two encodings $\pBoolstep, \pRelstep$ of the step function as ``boolean to boolean'' proofs in \cite{clift_murfet2} give rise to the same naive probabilistic extension, which moreover is motivated on probabilistic grounds in Section \ref{section:bayesian}. Note that $\pBoolstep$ is based on Girard's encoding in \cite{girard_complexity}.

\vspace{0.2cm}

In this section we set $s = |\Sigma|, n = |Q|$ and make the identifications
\begin{gather}
\Sigma = \{ 0, \ldots, s - 1 \}\\
Q = \{ 0, \ldots, n - 1 \}
\end{gather}
such that the blank symbol $\Box$ is identified with $0$. Throughout $A$ is a fixed type and we write $\tBool$ for $\tBool_A$. Given $a \ge 1$ there is a proof \cite[Proposition 5.5]{clift_murfet2}
\be\label{eq:typecstep}
{_a} \pRelstep : (2a+1) \,{!} {}_s \tBool, {!} {_n} \tBool \vdash \big({!} {}_s \tBool\big)^{\otimes 2a + 3} \otimes {!} {_n} \tBool
\ee
such that
\[
{_a} \pRelstep( \sigma_{-a},\ldots,\sigma_a, q ) = ( \tau_{-a-1},\ldots,\tau_{a+1}, q' )
\]
if and only if the Turing machine $M$ initialised in state $q$ and with tape\footnote{The underline indicates the position of the head.}
\[
\ldots, \Box, \sigma_{-a},\ldots, \sigma_{-1}, \underline{\sigma_0}, \sigma_1, \ldots, \sigma_a, \Box, \ldots
\]
is after one time step in the state $q'$ with tape contents
\[
\ldots, \Box, \tau_{-a-1},\ldots, \tau_{-1}, \underline{\tau_0}, \tau_1, \ldots, \tau_{a+1}\,, \Box, \ldots\,.
\]
These proofs are component-wise plain and so the denotations are morphisms of coalgebras
\be\label{eq:psimorco}
{!} \dntn{{}_s \tBool}^{\otimes 2a+1} \otimes {!} \dntn{{}_n \tBool} \lto {!} \dntn{{}_s \tBool}^{\otimes 2a+3} \otimes {!} \dntn{ {}_n \tBool}\,.
\ee
One way to define the propagation of uncertainty through $M$ is to propagate error through the proofs that encode its transition function in linear logic. This is done by taking the naive probabilistic extensions of the proofs ${}_a \pRelstep$, for which we fix sets of proofs
\begin{gather}
\cat{P}_{tape} = \{ \underline{\sigma} \}_{\sigma \in \Sigma} \text{ of } {}_s \tBool\,, \label{eq:ptape}\\
\cat{P}_{state} = \{ \underline{q} \}_{q \in Q} \text{ of } {}_n \tBool\,. \label{eq:pstate}
\end{gather}
In the notation of Corollary \ref{cor:plain_prob} set
\begin{align}
\cat{P}_1 &= \cdots = \cat{P}_{2a+1} = \cat{P}_{tape}\,, \cat{P}_{2a+2} = \cat{P}_{state} \label{eq:setsofproofsnv1}\\
\cat{Q}_1 &= \cdots = \cat{Q}_{2a+3} = \cat{P}_{tape}\,, \cat{Q}_{2a+4} = \cat{P}_{state} \label{eq:setsofproofsnv2}\,.
\end{align}
It is clear that the hypotheses of Corollary \ref{cor:plain_prob} are satisfied, and so we have a naive probabilistic extension $\Delta {_a} \pRelstep$ which is the unique function making the diagram
\[
\xymatrix@C+5pc@R+1.5pc{
{!} \dntn{{}_s \tBool}^{\otimes 2a+1} \otimes {!} \dntn{ {}_n \tBool } \ar[r]^-{\dntn{{_a} \pRelstep}} &  {!} \dntn{{}_s \tBool}^{\otimes 2a+3} \otimes {!} \dntn{{}_n \tBool} \\
(\Delta \Sigma)^{2a+1} \times \Delta Q \ar[u]^-{\iota} \ar[r]_-{\Delta {_a} \pRelstep} & (\Delta \Sigma)^{2a+3} \times \Delta Q \ar[u]_-{\iota}
}
\]
commute, where we identify $\Sigma = \cat{P}_{tape}$ and $Q = \cat{P}_{state}$ and $\iota$ is from Definition \ref{defn:iota_notation}. Next we give the explicit formulas for this naive probabilistic extension. 

\begin{lemma}\label{lemma:explicitdeltastep} Given $\bold{x} \in (\Delta \Sigma)^{2a+1}$ and $y \in \Delta Q$
\be
\Delta {}_a \pRelstep\big( \bold{x}, y ) = (\bold{w}, z)
\ee
where, writing $\bold{x} = ( x^u )_{u=-a}^a$ and using the notation
\begin{align}
M^d &:= \sum_{\sigma \in \Sigma} \sum_{q \in Q} \delta_{\delta_3( \sigma, q ) = d} \, x^0_{\sigma} y_q && d \in \{ \textrm{left}, \textrm{right} \} \\
W^\sigma &:= \sum_{\sigma' \in \Sigma} \sum_{q \in Q} \delta_{\delta_1(\sigma',q) = \sigma}\, x^0_{\sigma'} y_q && \sigma \in \Sigma
\end{align}
we have
\begin{align}
z_q &= \sum_{\sigma \in \Sigma} \sum_{q' \in Q} \delta_{\delta_2(\sigma,q') = q}\, x^0_\sigma y_{q'}\\
w^u_\sigma &= M^{\textrm{right}} \Big( \delta_{u \neq -1} x^{u+1}_\sigma + \delta_{u = -1} W^\sigma \Big) \notag \\
&\quad + M^{\textrm{left}} \Big( \delta_{u \neq 1} x^{u-1}_\sigma + \delta_{u = 1} W^\sigma \Big)
\end{align}
where by convention $x^u = \Box$ for $u \notin [-a,a]$.
\end{lemma}
\begin{proof}
This follows from the polynomials given in \cite[Remark 5.6]{clift_murfet2}, when we add the constraint that the input linear combinations are probability distributions.
\end{proof}

\begin{lemma}\label{lemma:nvstepext} There is a unique function
\be
\Delta \textrm{step}: (\Delta \Sigma)^{\mathbb{Z}, \Box} \times \Delta Q \lto (\Delta \Sigma)^{\mathbb{Z}, \Box} \times \Delta Q
\ee
with the property that for any $a \ge 1$ the diagram
\[
\xymatrix@C+5pc@R+1pc{
(\Delta \Sigma)^{\mathbb{Z}, \Box} \times \Delta Q \ar[r]^-{\Delta \textrm{step}} & (\Delta \Sigma)^{\mathbb{Z}, \Box} \times \Delta Q\\
(\Delta \Sigma)^{2a+1} \times \Delta Q \ar[u]^{\textrm{fill}_a \times 1}\ar[r]_-{\Delta {_a} \pRelstep} & (\Delta \Sigma)^{2a+3} \times \Delta Q \ar[u]_-{\textrm{fill}_{a+1} \times 1}
}
\]
commutes, where for $c \ge 1$ the function
\begin{gather*}
\textrm{fill}_c: (\Delta \Sigma)^{2c+1} \lto (\Delta \Sigma)^{\mathbb{Z}, \Box}\\
\textrm{fill}_c\big( (\sigma_u)_{u=-c}^{c} \big)(v) = \begin{cases} \sigma_v & -c \le v \le c \\
\Box & \text{otherwise} \end{cases}
\end{gather*}
fills the tape outside $[-c,c]$ with blanks.
\end{lemma}
\begin{proof}
Any sequence in $(\Delta \Sigma)^{\mathbb{Z}, \Box} \times \Delta Q$ is in the image of $\textrm{fill}_a \times 1$ for some $a$, so there is at most one function making all such diagrams commute. To see that such a function exists, it suffices to observe that if $b \ge a$ then the diagram
\[
\xymatrix@C+5pc@R+1pc{
(\Delta \Sigma)^{2b+1} \times \Delta Q \ar[r]^-{\Delta {_b} \pRelstep} & (\Delta \Sigma)^{2b+3} \times \Delta Q\\
(\Delta \Sigma)^{2a+1} \times \Delta Q \ar[u]\ar[r]_-{\Delta {_a} \pRelstep} & (\Delta \Sigma)^{2a+3} \times \Delta Q \ar[u]
}
\]
commutes, where the vertical maps substitute blanks outside of $[-a,a]$ and $[-a-1,a+1]$, respectively and act as the identity on $\Delta Q$. But this is clear from Lemma \ref{lemma:explicitdeltastep}.
\end{proof}

Observe that by construction the following diagram commutes:
\be\label{eq:deltastepvsstep}
\xymatrix@C+2.5pc@R+1pc{
\Sigma^{\mathbb{Z}, \Box} \times Q \ar[d]_-{\inc} \ar[r]^-{\textrm{step}} & \Sigma^{\mathbb{Z},\Box} \times Q \ar[d]^-{\inc}\\
(\Delta \Sigma)^{\mathbb{Z}, \Box} \times \Delta Q \ar[r]_-{\Delta \textrm{step}} & (\Delta \Sigma)^{\mathbb{Z}, \Box} \times \Delta Q
}
\ee

\begin{definition}\label{defn:nvprobfullstep} We refer to $\Delta \textrm{step}$ as the \emph{naive probabilistic extension} of $\textrm{step}$.
\end{definition}


\begin{definition}\label{defn:prelstept} For $t \ge 1$ we define a proof
\be
{^t_a} \pRelstep : (2a+1) \,{!} {}_s \tBool, {!} {_n} \tBool \vdash \big({!} {}_s \tBool\big)^{\otimes 2a + 1 + 2t} \otimes {!} {_n} \tBool
\ee
which is the cut of $t$ copies of the encoding of the step function
\be
{^t_a} \pRelstep := {_{a+t-1}} \pRelstep \l \cdots \l {_{a+1}} \pRelstep \l {_a} \pRelstep\,.
\ee
\end{definition}

\begin{lemma}\label{lemma:prelsteptnv} For any $a,t \ge 1$ the diagram
\[
\xymatrix@C+5pc@R+1pc{
(\Delta \Sigma)^{\mathbb{Z}, \Box} \times \Delta Q \ar[r]^-{(\Delta \textrm{step})^t} & (\Delta \Sigma)^{\mathbb{Z}, \Box} \times \Delta Q\\
(\Delta \Sigma)^{2a+1} \times \Delta Q \ar[u]^{\textrm{fill}_a \times 1}\ar[r]_-{\Delta \big({^t_a} \pRelstep\big)} & (\Delta \Sigma)^{2a+1+2t} \times \Delta Q \ar[u]_-{\textrm{fill}_{a+t} \times 1}
}
\]
commutes.
\end{lemma}
\begin{proof}
By compositionality (Lemma \ref{lemma:cut_componentwiseplainnv}) we have
\[
\Delta\big( {^t_a} \pRelstep \big) = \Delta {_{a+t-1}} \pRelstep \circ \cdots \circ \Delta {_{a+1}} \pRelstep \circ \Delta {_a} \pRelstep
\]
so the claim follows from Lemma \ref{lemma:nvstepext}.
\end{proof}

 We may view the Turing machine as a discrete dynamical system, with state space $\Sigma^{\mathbb{Z}, \Box} \times Q$ and evolution rule given by the function $\textrm{step}$. Then $\Delta \textrm{step}$ gives an extension of this discrete dynamical system to a \emph{smooth} dynamical system (by which we mean a smooth self-map of a manifold with corners). We think of this latter dynamical system as the state evolution of an observer of the operation of the Turing machine.

We imagine that the observer knows the tape alphabet $\Sigma$, internal states $Q$ and instructions $\delta$ of the Turing machine, but is uncertain about the initial configuration: perhaps some of the squares on the tape are not visible, or only partially visible, to the observer while of course being perfectly visible to the machine. At each time step the observer updates their state of knowledge to reflect their state of belief about the new state of the machine. As explained in the next section, the observer described by $\Delta \textrm{step}$ updates their state of knowledge under the hypothesis that certain random variables are conditionally independent given the previous state.\footnote{If instead we use the standard probabilistic extension $\Delta^{\std} \textrm{step}$ of Appendix \ref{section:stdprob} then we get a different smooth extension of the Turing machine as a discrete dynamical system, and this smooth dynamical system is the state evolution of an ordinary Bayesian observer of the Turing machine.}

\subsection{The naive Bayesian observer}\label{section:bayesian}





The purpose of this section is to make precise the idea that the naive probabilistic extension $\Delta \textrm{step}$ of the Turing machine $M$ describes the time evolution of the state of belief of a ``naive'' Bayesian observer. We begin with an analysis of Turing machines using standard Bayesian probability, for which our references are \cite{coxbook, sivia, jaynes}. 

By the tape square in \emph{relative position $u \in \mathbb{Z}$} we mean the tape square which would be labelled $u$ if we were to enumerate all the squares on the tape in increasing order moving left to the right, with the square currently under the head assigned zero. 

\begin{definition} The random variables are (at times $t \ge 0$):
\begin{itemize}
\item $Y_{u,t}$: the content of the tape square at relative position $u$ at time $t$,
\item $S_t$: the internal state at time $t$,
\item $\Mv_t$: the direction to move in the transition from time $t$ to $t+1$,
\item $\Wr_t$: the symbol to be written in the transition from time $t$ to $t+1$,
\end{itemize}
taking values in $\Sigma, Q, \{ \operatorname{left}, \operatorname{right} \}$ and $\Sigma$ respectively. 
\end{definition}

We view the probability distributions of these random variables as describing the state of belief of a Bayesian observer (sometimes referred to as a \emph{standard} Bayesian observer) who has perfect knowledge of the transition function $\delta$. At time $t = 0$ for all but a finite number of tape positions the observer is certain that the position contains a blank symbol. 

\begin{lemma}\label{lemma:relBTMup} Let $C \in \Delta( \Sigma^{\mathbb{Z},\Box} \times Q )$ denote the initial probability distribution over configurations. With $\sigma, \sigma'$ ranging over $\Sigma$ and $q,q'$ range over $Q$, we have
\begin{align*}
\prob( \Mv_t = d \l C ) &= \sum_{\sigma,q} \delta_{\delta_3( \sigma, q ) = d} \prob( Y_{0,t} = \sigma \land S_t = q \l C )\\
\prob( \Wr_t = \sigma \l C) &= \sum_{\sigma', q} \delta_{\delta_1(\sigma',q) = \sigma} \prob( Y_{0,t} = \sigma' \land S_t = q \l C )\\
\prob( S_{t+1} = q \l C) &= \sum_{\sigma, q'} \delta_{\delta_2(\sigma,q') = q} \prob( Y_{0,t} = \sigma \land S_t = q' \l C )\\
\prob( Y_{u,t+1} = \sigma \l C ) &= \delta_{u \neq -1} \prob( Y_{u+1,t} = \sigma \land \Mv_t = \operatorname{right} \l C )\\
&\quad + \delta_{u = -1} \prob( \Wr_t = \sigma \land \Mv_t = \operatorname{right} \l C )\\
&\quad+ \delta_{u \neq 1} \prob( Y_{u-1,t} = \sigma \land \Mv_t = \operatorname{left} \l C)\\
&\quad + \delta_{u = 1} \prob( \Wr_t = \sigma \land \Mv_t = \operatorname{left} \l C)\,.
\end{align*}
\end{lemma}

The state of belief of the observer at any time $t$ is the joint distribution of the random variables $\{ Y_{u,t} \}_{u \in \mathbb{Z}}, S_t$. As is clear from the update equations given in the lemma, to determine the marginal distributions $\{ \prob( Y_{u,t+1} ) \}_{u \in \mathbb{Z}}, \prob( S_{t+1} )$ we require more than just these marginal distributions at the time $t$: we also require various joint distributions. In order to update the state of belief of the standard Bayesian observer we have to update the entire joint distribution, and this is computationally very expensive.

A radical way to avoid this expense is to rewrite the update equations from the lemma, pretending that the random variables which appear in the joint distributions on the right hand side are \emph{conditionally independent} given $C$. The pairs that appear together in joint distributions in the original update equations are (taken at equal times)
\be\label{eq:pairs_random}
P_1 = \{ Y_0, S \}\,, \qquad P_2 = \{ \Wr, \Mv \}\,, \qquad P_3 = \{ Y_u, \Mv \}_{u \neq 0}\,.
\ee
We will consider an observer of the Turing machine who, following this proposal, updates their state of belief during each time step \emph{as if} all of the pairs in \eqref{eq:pairs_random} are conditionally independent given $C$. In fact these random variables are \emph{not} conditionally independent, so this observer is not calculating probability in the standard sense. There is an antecedent for this kind of probability in the theory of \emph{naive Bayesian classifiers} (see \cite[p.310]{poole},\cite[\S 20.2.2]{russell_norvig}) and for this reason we refer to the observer as a \emph{naive Bayesian}.

\begin{definition}\label{defn:linearbtm} The state of belief of the \emph{naive Bayesian observer} is a tuple
\be\label{eq:state_of_belief_nv}
\big( \{ \prob_{\nv}(Y_{u,t}) \}_{u \in \mathbb{Z}}, \prob_{\nv}(S_t) \big) \in (\Delta \Sigma)^{\mathbb{Z},\Box} \times \Delta Q\,.
\ee
Given initial state $C \in (\Delta \Sigma)^{\mathbb{Z},\Box} \times \Delta Q$ the update equations are by definition
\begin{align}
\prob_{\nv}( \Mv_t = d \l C ) &= \sum_{\sigma,q} \delta_{\delta_3( \sigma, q ) = d} \prob_{\nv}( Y_{0,t} = \sigma \l C ) \prob_{\nv}( S_t = q \l C )\\
\prob_{\nv}( \Wr_t = \sigma \l C ) &= \sum_{\sigma', q} \delta_{\delta_1(\sigma',q) = \sigma} \prob_{\nv}( Y_{0,t} = \sigma' \l C) \prob_{\nv}( S_t = q \l C)\\
\prob_{\nv}( S_{t+1} = q \l C ) &= \sum_{\sigma, q'} \delta_{\delta_2(\sigma,q') = q} \prob_{\nv}( Y_{0,t} = \sigma \l C ) \prob_{\nv}( S_t = q' \l C) \label{eq:update_s_nv}\\
\prob_{\nv}( Y_{u,t+1} = \sigma \l C) &= \prob_{\nv}( \Mv_t = \operatorname{right} \l C)\Big( \delta_{u \neq -1} \prob_{\nv}( Y_{u+1,t} = \sigma \l C )\label{eq:Yu_nvb}\\
&\qquad\qquad + \delta_{u = -1} \prob_{\nv}( \Wr_t = \sigma \l C ) \Big)\nonumber\\
&\quad + \prob_{\nv}( \Mv_t = \operatorname{left} \l C )\Big( \delta_{u \neq 1} \prob_{\nv}( Y_{u-1,t} = \sigma \l C)\nonumber\\
&\qquad\qquad + \delta_{u = 1} \prob_{\nv}( \Wr_t = \sigma \l C) \Big)\,.\nonumber
\end{align}
\end{definition}



The \emph{update function} for the naive Bayesian observer is the function from $(\Delta \Sigma)^{\mathbb{Z},\Box} \times \Delta Q$ to itself sending the current state \eqref{eq:state_of_belief_nv} to the state determined by the formulas \eqref{eq:update_s_nv}, \eqref{eq:Yu_nvb}.

\begin{lemma}\label{lemma:probnv} The update function for the naive Bayesian observer is $\Delta \textrm{step}$.
\end{lemma}
\begin{proof}
By inspection of the formulas in Lemma \ref{lemma:explicitdeltastep}.
\end{proof}

\subsection{Derivatives of Turing machines}\label{section:der_tm}

The derivatives of a Turing machine $M$ are numerical measures of the relative magnitude of the change in an output (the contents of an individual tape square, or the internal state) caused by an infinitesimal variation in one of the inputs when we run the Turing machine for some fixed number of steps $t$. To make sense of these infinitesimal variations, we take as the relevant input and output quantities the states of belief of a naive Bayesian observer, and the derivative measures the rate of change of belief. 

Recall that the Turing machine is encoded as
\begin{itemize}
\item a family of proofs ${_a} \pRelstep$ of linear logic \eqref{eq:typecstep},
\item whose denotations $\dntn{{_a} \pRelstep}$ are morphisms of coalgebras \eqref{eq:psimorco}, 
\item from which we may extract the naive probabilistic extension $\Delta \textrm{step}$ of $\textrm{step}$.
\end{itemize}
In this section we relate three different views on the derivative of $M$: the \emph{syntactic} (meaning the Ehrhard-Regnier derivative of a proof in linear logic) the \emph{coalgebraic} (meaning the evaluation of a morphism of coalgebras on a primitive element) and the \emph{ordinary calculus} of the naive probability (meaning tangent maps of the function $\Delta \textrm{step}$). 

As explained in \cite[Remark 2.2]{clift_murfet2} and Corollary \ref{cor:diaprimcomp} the derivatives (in all three senses) may be understood component by component, so we may restrict our attention without loss of generality to the plain proofs
\begin{align*}
\psi_v: (2a+1)\,{!} {}_s \tBool, {!} {}_n \tBool &\vdash {}_s \tBool\,,\\
\varphi: (2a+1)\,{!} {}_s \tBool, {!} {}_n \tBool &\vdash {}_n \tBool
\end{align*}
which are the components of the proof ${^t_a} \pRelstep$ (see Definition \ref{defn:prelstept}). These proofs respectively compute the content of the tape square in relative position $v \in \mathbb{Z}$ and the state of the Turing machine after $t$ time steps. As in the previous section we fix a type $A$ and write $\tBool$ for $\tBool_A$. We must $v$ to lie between $-2a-2t-1$ and $2a+2t+1$, but this is not a real restriction because outside the range the tape is blank.

The denotations of the promotions of $\psi_v, \varphi$ are morphisms of coalgebras
\begin{align*}
\dntn{\prom(\psi_v)}: {!} \dntn{{}_s \tBool}^{\otimes 2a+1} \otimes {!} \dntn{ {}_n \tBool} \lto {!}\dntn{{}_s \tBool}\,,\\
\dntn{\prom(\varphi)}: {!} \dntn{{}_s \tBool}^{\otimes 2a+1} \otimes {!} \dntn{ {}_n \tBool} \lto {!}\dntn{ {}_n \tBool}
\end{align*}
and the naive probabilistic extensions are smooth maps
\begin{align*}
\Delta \psi_v: \big(\Delta \Sigma\big)^{2a+1} \times \Delta Q \lto \Delta \Sigma\,,\\
\Delta \varphi: \big(\Delta \Sigma\big)^{2a+1} \times \Delta Q \lto \Delta Q\,.
\end{align*}
In order to talk about the derivative of $M$ we have to consider an infinitesimal variation of the input data. This infinitesimal variation is made relative to some fixed initial state of belief about the tape and state, which we denote
\[
C = \big( \bold{x} = (x^u)_{u =-a}^a, y ) \in (\Delta \Sigma)^{[-a,a]} \times \Delta Q\,.
\]
Then $\Delta \psi_v(C), \Delta \varphi(C)$ represent the state of belief about the tape and internal state, after $t$ steps of the machine. Next we make an infinitesimal variation in $C$ and observe the infinitesimal variation this causes in $\Delta \psi_v(C)$ and $\Delta \varphi(C)$. 

We first arbitrarily choose a symbol $\sigma_0 \in \Sigma$ and state $q_0 \in Q$ and write $\Bar{\Sigma} = \Sigma \setminus \{ \sigma_0 \}$ and $\Bar{Q} = Q \setminus \{ q_0 \}$. Then by Lemma \ref{lemma:basistangent} for any $x \in \Delta \Sigma, y \in \Delta Q$ we have the following presentation of the tangent spaces
\begin{gather*}
T_x(\Delta \Sigma) \cong \bigoplus_{\sigma \in \Bar{\Sigma}} \mathbb{R} B^{\sigma}_{\sigma_0} \,, \qquad T_y(\Delta Q) \cong \bigoplus_{q \in \Bar{Q}} \mathbb{R} B^{q}_{q_0}
\end{gather*}
where we abuse notation and write $\sigma, q$ for the proofs $\underline{\sigma}, \underline{q}$. Here $B^{\sigma}_{\sigma_0}$ denotes the ``revision of belief'' tangent vector  which represents an infinitesimal decrease in the observer's probability of $\sigma_0$ and increase in the probability of $\sigma$. The space of all possible variations\footnote{The revision of belief which increases the probability of $\sigma'$ and decreases that of $\sigma$ is $B^{\sigma'}_{\sigma_0} - B^{\sigma}_{\sigma_0}$.} in the initial data $C$ is therefore spanned by the $B^q_{q_0}$ and one copy of $B^{\sigma}_{\sigma_0}$ for each tape location $-a + 1 \le u \le b$. We denote this copy by $B(u)^{\sigma}_{\sigma_0}$. The infinitesimal variation that these cause in the output $\Delta \psi_v(C), \Delta \varphi(C)$ are measured by the tangent maps
\begin{align*}
T_{C}\big( \Delta \psi_v \big): \bigoplus_{u=-a}^a T_{x^u}\big( \Delta \Sigma \big) \oplus T_{y}( \Delta Q ) \lto T_{\Delta \psi_v(C)}( \Delta \Sigma )\,,\\
T_{C}\big( \Delta \varphi \big): \bigoplus_{u=-a}^a T_{x^u}\big( \Delta \Sigma \big) \oplus T_{y}( \Delta Q ) \lto T_{\Delta \varphi(C)}( \Delta Q )\,.
\end{align*}
By Theorem \ref{theorem:maind} these output variations are computed by
\begin{gather}
T_{C}( \Delta \psi_v )( B(u)^{\sigma}_{\sigma_0} ) = \dntn{\psi_v}\Big( \vacu_{\dntn{x^{-a}}} \otimes \cdots \otimes \big| \dntn{\underline{\sigma}} - \dntn{\underline{\sigma_0}} \big\rangle_{\dntn{x^u}} \otimes \cdots \otimes \vacu_{\dntn{x^a}} \otimes \vacu_{\dntn{y}} \Big)\label{eq:theoremmaindtm1}\\
T_{C}( \Delta \varphi )( B^{q}_{q_0} ) = \dntn{\varphi}\Big( \vacu_{\dntn{x^{-a}}} \otimes \cdots \otimes \vacu_{\dntn{x^a}} \otimes \big| \dntn{\underline{q}} - \dntn{\underline{q_0}} \big\rangle_{\dntn{y}} \Big)\,.\label{eq:theoremmaindtm2}
\end{gather}
These two equations give the precise relationship between the coalgebraic derivatives of the encodings (that is, the evaluation of the proof denotations $\dntn{\psi_v}, \dntn{\varphi}$ on primitive elements) and the ordinary derivatives of the naive probability. We can expand the left hand sides of these equations as vectors in the ambient spaces $T_{\Delta \psi_v(C)}( \mathbb{R} \Sigma )$ and $T_{\Delta \varphi(C)}(\mathbb{R} Q)$ as follows:
\begin{align*}
T_{C}( \Delta \psi_v )( B(u)^{\sigma}_{\sigma_0} ) &= \sum_{\tau \in \Sigma} \lim_{h \to 0} \frac{f_{\psi_v}^\tau\big( (\ldots, x^u + h \cdot \sigma - h \cdot \sigma_0, \ldots), y\big) - f_{\psi_v}^\tau(C)}{h} \cdot \frac{\partial}{\partial z_\tau}
\end{align*}
where the $\tau$ component of $\Delta \psi_v$ we denote by $f_{\psi_v}^\tau$, and similarly
\begin{align*}
T_{C}( \Delta \varphi )( B^{q}_{q_0} ) &= \sum_{s \in Q} \lim_{h \to 0} \frac{f_{\varphi}^s\big( (\ldots, x^u, \ldots), y + h \cdot q - h \cdot q_0 \big) - f_{\varphi}^s(C)}{h} \cdot \frac{\partial}{\partial z_s}\,.
\end{align*}
Here $z_\tau$ and $z_s$ denote the coordinates on $\mathbb{R} \Sigma, \mathbb{R} Q$ respectively. For example, if $\lambda$ is the coefficient of the basis vector $\tfrac{\partial}{\partial z_s}$ in $T_C(\Delta \varphi)(B^q_{q_0})$ then an infinitesimal change from the initial state of belief $C$ of the observer, which increases the probability assigned to $q$ by $\Delta h$ and decreases the probability of $q_0$ by $\Delta h$, will cause a change of $\lambda \Delta h$ in the probability assigned to the state $s$ after $t$ steps. 

Finally, let us give the relation to the symbolic derivatives, by which we mean the derivatives of the proofs $\psi_v, \varphi$ in the language of differential linear logic. We can only make this connection in the case where all the $x^u$ and $y$ are distributions assigning probability $1$ to a single symbol or state, in which case there is a further equality of \eqref{eq:theoremmaindtm1} with
\begin{align*}
\dntn{\tfrac{\partial}{\partial X_u} & \psi_v( \sigma, x^{-a},\ldots, x^a,  y) }\\
 &- \dntn{\tfrac{\partial}{\partial X_u} \psi_v( \sigma_0,   x^{-a},\ldots, x^a,  y) }
\end{align*}
where $X_u$ denotes the input to the proof $\psi_v$ corresponds to the tape square $u$ in relative coordinates, and the proof $\frac{\partial}{\partial X_u} \psi_v$ is defined as in \eqref{eq:derivative_tree}. Similarly, there is a further equality of the vectors in \eqref{eq:theoremmaindtm2} with the difference
\begin{align*}
\dntn{\tfrac{\partial}{\partial Y} & \varphi( q, x^{-a},\ldots, x^a, y) }\\
 &- \dntn{\tfrac{\partial}{\partial Y} \varphi( q_0, x^{-a},\ldots, x^a,  y) }
\end{align*}
where $Y$ denotes the input to the proof $\varphi$ corresponding to the state.

\section{Gradient descent and Turing machines}\label{section:grad_descent_tm}

One motivation for differentiating algorithms comes from machine learning, where a basic problem is to minimise a loss function computed by some algorithm. To perform gradient descent one differentiates the outputs of the algorithm with respect to some set of inputs that one thinks of as parameters, and for this reason one restricts to numerical algorithms for which automatic differentiation is available \cite{autodiff,autodiffml}. Recurrent neural networks are often described as differentiable algorithms; see \cite{siegelmann,schmidhuber-thesis} and \cite[\S 2]{schmidhuberdl}, \cite[\S 3.1]{schmidhuberlearning}. Arguably, the notion of the derivative of an algorithm is a fundamental concept of machine learning.

In this section we use our (naive) probabilistic analysis of the Ehrhard-Regnier derivative to study the old problem of synthesising Turing machines. The idea of using Turing machines for machine learning is of course an old one, going back to Turing himself \cite{turing4, freer}. Turing machine synthesis has played an important foundational role in machine learning both theoretically \cite{solomonoff, hutter} and practically \cite{biermann, schmidhuber}. In recent years there has been significant activity in this area motivated by progress in deep learning; for a survey see \cite[\S 6.4]{gulwani}.

\vspace{0.2cm}

We consider the following learning problem:

\begin{setup}\label{setup:learning} A naive Bayesian observer is uncertain about the initial contents of the tape of a Turing machine $M$, which is then run for some number of steps. The observer is then told the contents of some number of squares on the tape. How should the observer vary their belief about the initial contents of the tape in light of this information?
\end{setup}

Let us make the problem more precise. Let $U \subseteq \mathbb{Z}$ denote the set of positions on the initial tape about which the observer has some uncertainty, and suppose the machine $M$ is run for $t$ steps, after which the observer is told the contents of tape squares in relative positions $V \subseteq \mathbb{Z}$. We suppose the observer is certain about the initial state $q_{\text{start}}$. The uncertainty about the initial tape is propagated by the function
\be\label{eq:deltastep1}
\Delta \textrm{step}^t(-, q_{\text{start}}): (\Delta \Sigma)^{\mathbb{Z}, \Box} \lto (\Delta \Sigma)^{\mathbb{Z}, \Box} \times \Delta Q
\ee
from Lemma \ref{lemma:nvstepext}. Let $U'$ be a finite set of tape positions disjoint from $U$ such that initially all the tape squares not in $U \cup U'$ are blank. By hypothesis the observer is certain about the initial contents $\bold{a} \in \Sigma^{U'}$ of the tape squares in $U'$. Let
\be\label{eq:deltastep2}
\Delta \textrm{step}^t(-, \bold{a}, q_{\text{start}}): (\Delta \Sigma)^{U} \lto (\Delta \Sigma)^{\mathbb{Z}, \Box} \times \Delta Q
\ee
denote the restriction of \eqref{eq:deltastep1} which substitutes blanks outside of $U \cup U'$ and $\bold{a}$ in $U'$. We denote the components of this map in relative tape positions $V \subseteq \mathbb{Z}$ by
\[
\Pi_V \Delta \textrm{step}^t(-, \bold{a}, q_{\text{start}}): (\Delta \Sigma)^{U} \lto (\Delta \Sigma)^{V}\,.
\]
In this notation the problem can be restated as follows: given some initial state of belief $\bold{h}_0 \in (\Delta \Sigma)^U$ of the observer, and information about the true contents $\bold{b} \in \Sigma^V$ how should the observer vary $\bold{h}_0$ so as to correct the discrepancy between $\bold{b}$ and
\be\label{eq:propunctoV}
\Pi_V \Delta \textrm{step}^t( \bold{h}_0, \bold{a}, q_{\text{start}}) \in (\Delta \Sigma)^V,
\ee
which represents the propagation of the uncertainty about the contents of $U$ to uncertainty about the contents of $V$. A standard choice of measure for the discrepancy between two distributions is the relative entropy:

\begin{definition}\label{defn:cross_entropy} Let $\Sigma$ be a finite set and $v,w \in \Delta \Sigma$ probability distributions with $w_\sigma \neq 0$ for all $\sigma \in \Sigma$. The \emph{relative entropy} or \emph{Kullback-Leiblier divergence} \cite[\S 2.6]{mackay} of $v,w$ is
\be
\dkl( v \,\vert\vert\, w ) = \sum_{\sigma \in \Sigma} v_\sigma \ln\left( \frac{v_\sigma}{w_\sigma}\right)
\ee
with the convention that $p \log(\frac{p}{q}) = 0$ if $p = 0$. This is a measure of the information gained when an observer revises their beliefs from the prior distribution $w$ to the posterior distribution $v$. This quantity is non-negative and zero if and only if $v = w$ \cite[p.44]{mackay}. If $\bold{v}, \bold{w} \in (\Delta \Sigma)^X$ for some finite set $X$, we define
\be
\dkl( \bold{v} \,\vert\vert\, \bold{w} ) = \sum_{x \in X} \dkl( v_x \,\vert\vert\, w_x )
\ee
again under the hypothesis that $w_{x,\sigma} \neq 0$ for all $\sigma \in \Sigma, x \in X$.
\end{definition}

An observer who varies their belief $\bold{h}$ about the initial contents of the tape in order to minimise the relative entropy between the sequence of distributions in \eqref{eq:propunctoV} and the sequence of symbols $\bold{b}$ is performing a form of conditional maximum likelihood estimation \cite[\S 5.5]{dlbook}. In the definition of the loss function below we also introduce a regularisation term which biases the maximum likelihood towards sequences $\bold{h}$ that lie in $\Sigma^{U} \subseteq (\Delta \Sigma)^U$.

In defining the loss we will need to perturb a given distribution $v \in \Delta \Sigma$, which may be on the boundary of the simplex, to a distribution which lies in the interior. To do this we move the point along the straight line between $v$ and the barycenter $\sum_{\sigma \in \Sigma} |\Sigma|^{-1} \sigma$. This is a standard trick in machine learning; see for example \cite{sklearn}.

\begin{definition} Given a finite set $\Sigma$ and $0 < \mu < 1$ define
\begin{gather*}
\varepsilon_\mu: \Delta \Sigma \lto \Delta \Sigma\,,\\
\sum_{\sigma \in \Sigma} v_\sigma \cdot \sigma \mapsto \sum_{\sigma \in \Sigma} \big( (1-\mu) v_\sigma + \mu |\Sigma|^{-1} \big) \cdot \sigma\,.
\end{gather*}
Clearly if $v \in \Delta \Sigma$ then $\varepsilon_\mu(v)_\sigma \neq 0$ for all $\sigma$. If $\bold{v} \in (\Delta \Sigma)^X$ for some finite set $X$, we define
\[
\varepsilon_\mu(\bold{v}) = ( \varepsilon_\mu( v_x ) )_{x \in X} \in (\Delta \Sigma)^X\,.
\]
\end{definition}

In the rest of this section we fix disjoint finite sets $U, U' \subseteq \mathbb{Z}$.

\begin{definition} A \emph{dataset} is a finite set $\mathscr{D}$ of tuples of the form $(\bold{a}, \bold{b}, V, t)$ where $V \subseteq \mathbb{Z}$ is finite and nonempty, $\bold{a} \in \Sigma^{U'}$, $\bold{b} \in \Sigma^V$ and $t \ge 1$.
\end{definition}

\begin{definition}\label{defn:loss} The \emph{loss function} associated to a dataset $\mathscr{D}$ is the smooth map
\begin{gather*}
L^{\lambda,\mu}_{\mathscr{D}}: (\Delta \Sigma)^U \lto \mathbb{R}\\
L^{\lambda,\mu}_{\mathscr{D}}(\bold{h}) = \sum_{(\bold{a}, \bold{b}, V, t) \in \mathscr{D}} \dkl\Big( \bold{b} \, \big\vert\big\vert\, \varepsilon_\mu \Pi_{V} \Delta \textrm{step}^{t}( \bold{h}, \bold{a},q_{\text{start}} ) \Big) + \lambda R(\bold{h})
\end{gather*}
where $0 < \lambda, \mu < 1$ are parameters and $R$ is a regularisation term
\be
R(\bold{h}) = \sum_{u \in U} \sum_{\sigma \in \Sigma} h_{u,\sigma}^2(1-h_{u,\sigma})^2\,.
\ee
Usually we drop the parameters from the notation and write $L_{\mathscr{D}}(\bold{h})$ for the loss.
\end{definition}

\begin{remark}\label{remark:state_loss} Consider the modified form of the learning problem in Setup \ref{setup:learning} where the observer is also told the final \emph{state} of the machine. An \emph{extended dataset} is a finite set of tuples $(\bold{a},\bold{b},V,t,q)$ where the first four entries are as before, and $q \in Q$ is a state. If we denote by $\Pi_{\text{state}} \Delta \textrm{step}^t(-, \bold{a}, q_{\text{start}}) \in \Delta Q$ the state component of \eqref{eq:deltastep1} then we modify the loss function by adding a term
\[
\dkl( q \,\vert\vert\, \varepsilon_\mu\Pi_{\text{state}} \Delta \textrm{step}^t ( \bold{h}, \bold{a}, q_{\text{start}}) \big)
\]
for each tuple in the extended dataset.
\end{remark}

The biased maximum likelihood estimator for $\bold{h}$ given a dataset $\mathscr{D}$ is
\be\label{eq:mllike}
\bold{h}^* = \underset{\bold{h}}{\operatorname{argmin}}\, \lim_{\mu \to 0} L^{\lambda,\mu}_{\mathscr{D}}(\bold{h})\,.
\ee
The standard way to search for the value $\bold{h}^*$ is to begin with some initial state of belief $\bold{h}_0$ and use gradient descent \cite[\S 4.3]{dlbook} or stochastic gradient descent if $\mathscr{D}$ is large \cite[\S 5.9]{dlbook}. We use the metric induced by the embedding $\Delta \Sigma \subseteq \mathbb{R} \Sigma$ and define $\nabla L_{\mathscr{D}}$ to be the vector field on $(\Delta \Sigma)^U$ dual to the exterior derivative $d( L_{\mathscr{D}} )$. This is easiest to describe in the case where $\Sigma = \{\sigma_1,\sigma_0\}$, where we may choose a coordinate
\begin{gather*}
\xymatrix{ [0,1] \ar[r]^-{\cong} & \Delta \Sigma }\,,\\
c_u \mapsto c_u \cdot \sigma_1 + (1-c_u) \cdot \sigma_0
\end{gather*}
for the copy of $\Delta \Sigma$ in relative position $u \in U$. Then 
\[
\nabla L_{\mathscr{D}}(\bold{h}) = \Bigg( \frac{\partial L_{\mathscr{D}}}{\partial c_u} \Big\vert_{\bold{h}} \Bigg)_{u \in U}\,.
\]

\begin{definition}\label{defn:grad_descent} The \emph{gradient descent flow} with initial condition $\bold{h}_0 \in (\Delta \Sigma)^U$ generated by $\mathscr{D}$ is the flow generated by the vector field $- \nabla L_{\mathscr{D}}$, see \cite[Ch. IV, Definition 3.11]{boothby}. 
\end{definition}


\begin{remark}
Let us assume for the moment that $\Sigma = \{0,1\}$ so that the loss is defined on $(\Delta \{0,1\})^U \cong [0,1]^U$. The defining feature of the vector field $\nabla L_{\mathscr{D}}$ is that the coordinate $c_u$ corresponding to the $u$th bit is effectively scaled by \emph{the degree to which it is used} in the algorithm $M$, where the degree has the meaning explained in Setup \ref{setup:psiandpi}. In particular, the gradient flow will tend to prioritise bits which are used more often; we give a concrete example of this in Section \ref{section:shift_machine} and Section \ref{section:attribution}. This is typical of linear logic.
\end{remark}

\subsection{Universal Turing machines}\label{section:utm}

Consider an observer of a repetitive process that generates a sequence of outputs $\bold{y}$ from a sequence of inputs $\bold{x}$. The assumption that this process is \emph{computable} is a strong prior which implies that for any Universal Turing Machine $\cat{U}$ there exists a code $\bold{h} \in \{0,1\}^*$ describing a Turing machine $M$ with $M(x) = y$ for all observed input-output pairs $(x,y)$. Determining this code $\bold{h}$ is the problem of \emph{inductive inference} \cite{solomonoff,gold,blum}.

Once a UTM $\cat{U}$ has been fixed the uncertainty is only about \emph{which} Turing machine is being simulated, and an approach by maximum likelihood estimation is available once we decide how to propagate the observer's uncertainty about the code $\bold{h}$ to uncertainty about the outputs $\cat{U}( \bold{h}, x )$ of the UTM. Propagating uncertainty using standard probability leads to an algorithm for finding $\bold{h}$ by gradient descent which has (per gradient descent step) exponential time complexity as a function of the code length $|\bold{h}|$. We now approach the problem using naive probability as a special case of Setup \ref{setup:learning}. The corresponding gradient descent has polynomial time complexity per step (Proposition \ref{prop:timecomp}).

\vspace{0.2cm}



We will need an upper bound on the time needed for the UTM to simulate a Turing machine to the point of halting (if it halts) and for this it is convenient to adopt a multi-tape model of the UTM, as given in \cite[Theorem 1.9]{arorabarak}. Consequently we will adopt the conventions of \cite[\S 1.4.1]{arorabarak} and take our UTM, which we denote $\cat{U}$, to have tape alphabet\footnote{We chose in Section \ref{section:probexec} a bijection of $\Sigma$ with $\{0,1,2,3\}$ in order to represent elements of the tape alphabet as proofs, and we assumed that under this bijection $\Box$ is sent to $0$. Since $0$ also appears explicitly in the tape alphabet, there is the possibility for confusion. To forestall this, let us remark that our chosen bijection is $(\rhd, \Box, 0, 1) = (0,1,2,3)$ and note that in this section any $0$ which appears denotes the tape alphabet symbol, which is represented in linear logic as the proof $\underline{2}$ of ${}_4 \tBool_A$ for some $A$.} $\Sigma = \{ \Box, \rhd, 0, 1 \}$ with a total of five tapes, as shown in Figure \ref{fig:tapesofU}.

\begin{figure}
\includegraphics[scale=0.55]{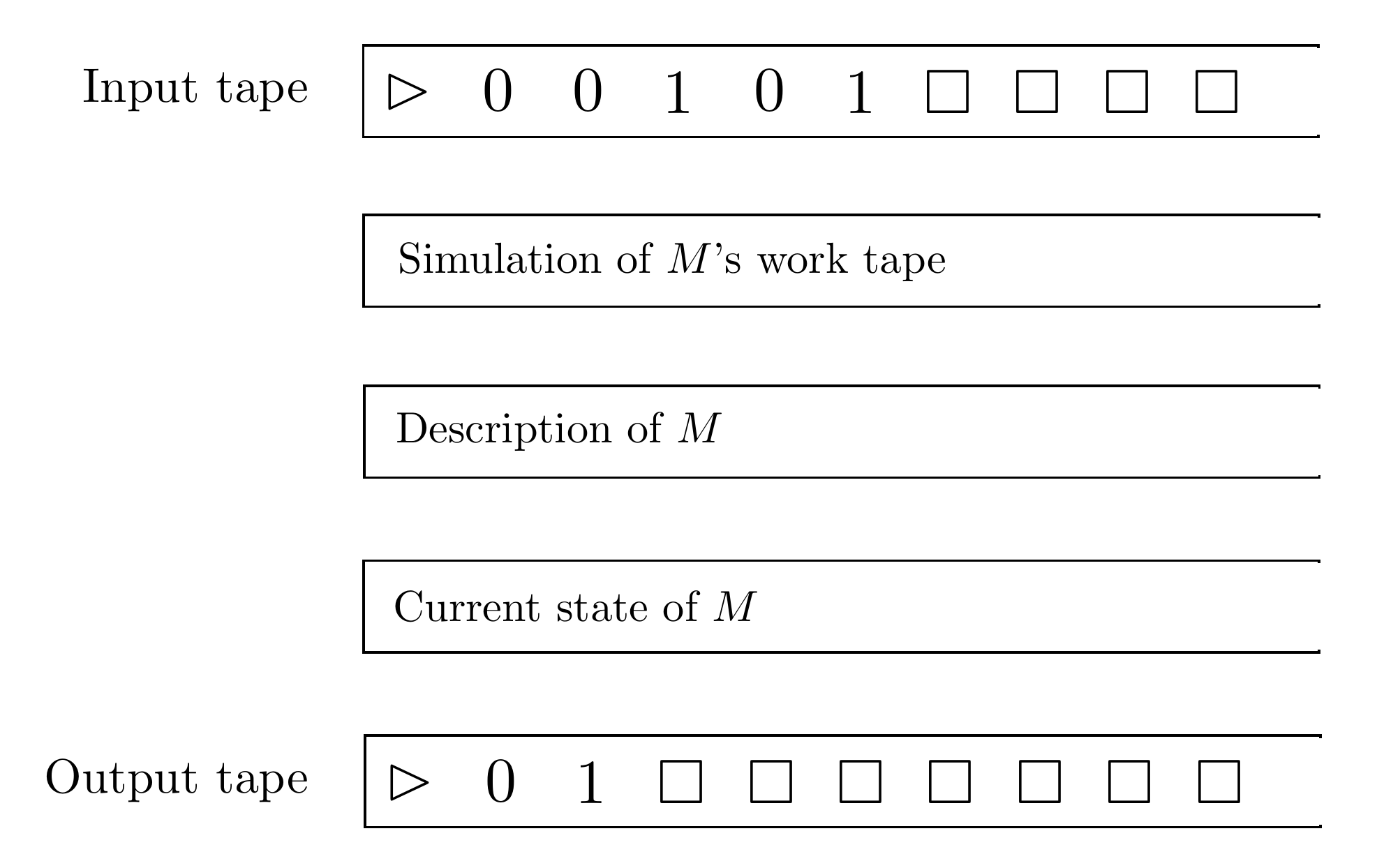}
\centering
\caption{Depiction of the tapes of $\cat{U}$.}\label{fig:tapesofU}
\end{figure}

These tapes are used, respectively, to store the input, simulate a given Turing machine $M$'s work tape, to store the description code of $M$, to store an encoded version of the current state of $M$, and to simulate the output tape of $M$. The heads of all the tapes are initially at the leftmost position, containing $\rhd$. We allow the UTM to choose at each time step to move the head left, right or not at all, and since we are also using an extended tape alphabet and a multi-tape machine, for our encoding of the step function of $\cat{U}$ as a linear logic proof we use the extensions of \cite[Appendix A]{clift_murfet2}. The material in the previous sections defines in the same way the naive probabilistic extension
\[
\Delta \textrm{step}: \Big[ (\Delta \Sigma)^{\mathbb{Z}, \Box} \Big]^5 \times \Delta Q \lto \Big[ (\Delta \Sigma)^{\mathbb{Z}, \Box} \Big]^5 \times \Delta Q
\]
where we have one copy of $(\Delta \Sigma)^{\mathbb{Z}, \Box}$ for each tape. In the context of the learning problem, the relative tape locations $U, U', V$ are now subsets of the disjoint union
\[
\mathbb{Z}_{\textrm{input}} \sqcup \mathbb{Z}_{\textrm{work}} \sqcup \mathbb{Z}_{\textrm{desc}} \sqcup \mathbb{Z}_{\textrm{state}} \sqcup \mathbb{Z}_{\textrm{output}}
\]
of five copies of $\mathbb{Z}$, one for each tape. In fact, as we will detail below, $U$ is always a subset of $\mathbb{Z}_{\textrm{desc}}$, $U'$ is a subset of $\mathbb{Z}_{\textrm{input}}$ and $V$ is a subset of $\mathbb{Z}_{\textrm{state}} \sqcup \mathbb{Z}_{\textrm{output}}$.

\begin{definition} Given a Turing machine $M$ let $c(M) \in \{0,1\}^*$ denote the code for $\cat{U}$.
\end{definition}

We can choose $\cat{U}$ such that if $M$ halts on input $x \in \{0,1\}^*$ within $T$ steps then $\cat{U}$ on input $(c(M), x)$ halts within $C f(T)$ steps where $C$ is a constant independent of $|x|$ and depending only on the alphabet size of $M$, its number of tapes and number of states, and $f: \mathbb{N} \lto \mathbb{N}$ is a fixed increasing function \cite[Theorem 1.9]{arorabarak}. For simplicity we will assume without loss of generality \cite[Claim 1.5, 1.6]{arorabarak} that all the Turing machines that we will simulate have one work tape (in addition to a read-only input tape, and the output tape) and the same alphabet $\Sigma$ as $\cat{U}$ (but we allow the number of states to vary), in which case $f(T) = T$ will do and $C$ is a function only of the number of states. We assume that in the encoding chosen for the states of a Turing machine, $q_{\text{halt}}$ is encoded as the same binary string $c(q_{\text{halt}})$ for every Turing machine with the same number of states.\footnote{We do not need to assume as in \cite[\S 1.4]{arorabarak} that every sequence is the code for a Turing machine, but we assume that if $c \in \{0,1\}^*$ is invalid, the UTM halts with an empty output tape. We also assume that our Turing machines (including $\cat{U}$) having special states $q_{\text{start}}$ and $q_{\text{halt}}$, that all machines are initialised in the state $q_{\text{start}}$ and that valid machines in state $q_{\text{halt}}$ only print the symbol they are reading and do not move the head. We use $\rhd$ to recognise the beginning of the tape, so any Turing machine which has a tuple printing $\rhd$ is considered invalid. We assume that once $\cat{U}$ has scanned the description of $M$ and found that it has halted, it updates the fourth tape to write the encoding $c(q_{\text{halt}})$, moves the head of the fourth tape to the original position containing $\rhd$, and halts.} Thus the transition function of $M$ is of the form
\be\label{eq:transition_M_UTM}
\delta_M: \Sigma^3 \times Q_M \lto \Sigma^2 \times Q_M \times \{ \textrm{left}, \textrm{right}, \textrm{stay} \}^3
\ee
where $Q_M$ depends on $M$ but $\Sigma$ is fixed.
\vspace{0.2cm}

We consider the following class of learning problems:

\begin{definition} A \emph{family of prefix equations} is a finite set $\mathscr{E}$ of tuples of the form
\[
(x,y,t) \in \{0,1\}^* \times \{0,1\}^* \times \mathbb{N}
\]
where $x,y$ are nonempty binary strings and $t > 0$. A \emph{solution} to $\mathscr{E}$ is a Turing machine $M$ such that for all $(x,y,t) \in \mathscr{E}$ when $M$ is initialised on input $x$ it halts in at most $t$ steps with $y$ as the initial segment on its output tape, and with the head on the output tape at the initial position (i.e. the unique square containing $\rhd$). We allow there to be other non-blank squares on the output tape.
\end{definition}


\vspace{0.2cm}

Our aim is to take a family of prefix equations $\mathscr{E}$ and formulate the search for solutions as a gradient descent problem. In order to simulate a Turing machine $M$ for $t$ steps, we will need to run $\cat{U}$ for at least $C t$ steps, where $C$ depends on the number of states $s = |Q_M|$ of $M$. Since the number of steps we will run the UTM forms part of the training data, this means our learning problems are also parametrised by the number of states $s$.

The region $U$ is supposed to be an initial portion of the tape of $\cat{U}$ which stores the description $c(M)$ of the Turing machine. We have to specify how much of this tape to use so our problem is in principle also parametrised over description lengths. However, we can describe the length of $c(M)$ in terms of $s$ as follows. Any state in $Q$ can be encoded using $\lfloor \log(s) \rfloor$ bits, and the number of tuples in the transition function $\delta_M$ is $|\Sigma|^3 |Q_M| = 4^3 s$. Stored in the standard way (see e.g. \cite[\S 7.2]{minsky}) the description $c(M)$ therefore has length
\be
a := 4^3 s( 9 + 2 \lfloor \log(s) \rfloor )
\ee
Throughout this section $a = a(s)$ denotes this function of $s$.



\begin{definition} The extended dataset $\mathscr{D}( \mathscr{E}, s )$ associated to a family of prefix equations $\mathscr{E}$ and number of states $s$ is defined as follows. We set
\begin{align*}
U &= \{ 1, \ldots, a \} \subseteq \mathbb{Z}_{\textrm{desc}}\\
U' &= \{ 1, \ldots, \max_{(x,y,t) \in \mathscr{E}} |x| \} \subseteq \mathbb{Z}_{\textrm{input}}\,.
\end{align*}
For each tuple $(x,y,t)$ in $\mathscr{E}$ we add the tuple $(\bold{a},\bold{b}, V, T, q_{\text{halt}})$ to $\mathscr{D}(\mathscr{E},s)$, where:
\begin{itemize}
\item $\bold{a} = x$
\item $V = V_{\textrm{state}} \cup V_{\textrm{output}}$ where
\begin{gather*}
V_{\textrm{state}} = \{0, 1, \ldots, |c(q_{\text{halt}})| \} \subseteq \mathbb{Z}_{\textrm{state}}\,,\\
V_{\textrm{output}} = \{0, 1, \ldots, |y| \} \subseteq \mathbb{Z}_{\textrm{output}}\,.
\end{gather*}
\item $\bold{b}$ is made up of a string $\bold{b}_{\textrm{state}}$ and a string $\bold{b}_{\textrm{output}}$, where
\begin{align*}
\bold{b}_{\textrm{state}} &= \rhd c(q_{\text{halt}})\\
\bold{b}_{\textrm{output}} &= \rhd y\,.
\end{align*}
\item $T$ is $C t$, where $C = C(s)$ depends on $s$.
\end{itemize}
\end{definition}

\begin{definition} The \emph{loss function} of the pair $(\mathscr{E}, s)$ is the smooth map
\[
L^{\mathscr{E}}_{s} := L_{\mathscr{D}(\mathscr{E},s)}: (\Delta \{0,1\})^a \lto \mathbb{R}
\]
associated to the extended dataset $\mathscr{D}(\mathscr{E}, s)$. More explicitly
\begin{align}
L^{\mathscr{E}}_{s}(\bold{h}) &= \sum_{(x,y,t) \in \mathscr{E}} \dkl\Big( \rhd c(q_{\text{halt}}) \, \big\vert\big\vert\, \varepsilon_\mu \Pi_{V_{\textrm{state}}} \Delta \textrm{step}^{Ct}( \bold{h}, x,q_{\text{start}} ) \Big) \nonumber\\
&+ \sum_{(x,y,t) \in \mathscr{E}} \dkl\Big( \rhd y \, \big\vert\big\vert\, \varepsilon_\mu \Pi_{V_{\textrm{output}}} \Delta \textrm{step}^{Ct}( \bold{h}, x,q_{\text{start}} ) \Big)\label{eq:defnlossprefix}\\
&+ \sum_{(x,y,t) \in \mathscr{E}} \dkl\Big( q_{\textrm{halt}} \, \big\vert\big\vert\, \varepsilon_\mu \Pi_{\textrm{state}} \Delta \textrm{step}^{Ct}( \bold{h}, x,q_{\text{start}} ) \Big) + \lambda R(\bold{h})\,. \nonumber
\end{align}
\end{definition}

Roughly speaking a Turing machine $M$ is a solution of $\mathscr{E}$ if and only if it is a zero of the loss function $L^{\mathscr{E}}_s$. However, this cannot be precisely true, because the definition of the loss involves a parameter $\mu$ which moves points (such as $c(M)$) away from the boundary of the simplex. The precise reformulation of this statement is instead:

\begin{theorem}\label{theorem:zeros} Let $M$ be a Turing machine with $s$ states. Then $M$ is a solution of the system of prefix equations $\mathscr{E}$ if and only if
\[
\lim_{\mu \to 0} L^{\mathscr{E}}_{s}( c(M) ) = 0\,.
\]
If $M$ is \emph{not} a solution, then $\lim_{\mu \to 0} L^{\mathscr{E}}_{s}( c(M) ) = \infty$.
\end{theorem}
\begin{proof}
In the special case where $\bold{h} = c(M) \in \Sigma^U$ is a sequence of tape symbols the regularisation term in the loss vanishes, and we also have by commutativity of \eqref{eq:deltastepvsstep}
\begin{align*}
\Delta \textrm{step}^{Ct}( \bold{h}, x,q_{\text{start}} ) &= \textrm{step}^{Ct}( \bold{h}, x, q_{\textrm{start}} ) \in \Sigma^{\mathbb{Z}, \Box} \times Q\,.
\end{align*}
By construction $M$ is a solution if and only if for every $(x,y,t) \in \mathscr{E}$,
\begin{itemize}
\item $\rhd c(q_{\textrm{halt}}) = \Pi_{V_{\textrm{state}}} \textrm{step}^{Ct}( c(M), x, q_{\textrm{start}} )$ in $\Sigma^*$, and
\item $\rhd y = \Pi_{V_{\textrm{output}}} \textrm{step}^{Ct}( c(M), x, q_{\textrm{start}} )$ in $\Sigma^*$, and
\item $q_{\textrm{halt}} = \Pi_{\textrm{state}} \textrm{step}^{Ct}( c(M), x, q_{\textrm{start}} )$ in $Q$.
\end{itemize}
Hence the claim follows from the fact that for $\sigma, \tau \in \Sigma$ (and similarly for $Q$),
\[
\lim_{\mu \to 0} \dkl( \tau \,\vert\vert\, \varepsilon_\mu(\sigma) ) = \begin{cases} \infty & \sigma \neq \tau \\ 0 & \sigma = \tau\end{cases} 
\]
and from the fact that the summands in the loss are all non-negative.
\end{proof}

\begin{corollary} We have $\lim_{\mu \to 0} L^{\mathscr{E}}_{s}(\bold{h}) = 0$ if and only if $\bold{h}$ is the code of a Turing machine $M$ with $s$ states which is a solution of $\mathscr{E}$.
\end{corollary}
\begin{proof}
Suppose $\bold{h} \in (\Delta \Sigma)^a$ has $\lim_{\mu \to 0} L^{\mathscr{E}}_{s}(\bold{h}) = 0$. Then considering the regularisation term we see that $\bold{h} \in \Sigma^a$. If this 
were \emph{not} the code of a valid Turing machine, then by hypothesis $\cat{U}$ halts on $\bold{h}$ with an empty output tape. Since the $y$'s appearing in $\mathscr{E}$ are all nonempty, this would force $\lim_{\mu \to 0} L^{\mathscr{E}}_{s}(\bold{h}) = \infty$. So $\bold{h} = c(M)$ for some Turing machine $M$, and then the vanishing of the limit implies that it is a solution.
\end{proof}


We have defined a family of manifolds with corners and vector fields
\begin{gather*}
\cat{M}_s = (\Delta \{0,1\})^a \cong [0,1]^a\\
\nabla L^{\mathscr{E}}_{s} \in \Gamma( \cat{M}_s, T \cat{M}_s )\,.
\end{gather*}
By the theorem, flow along $- \nabla L^{\mathscr{E}}_{s}$ will locally diverge from vertices representing codes of Turing machines that are not solutions, and locally converge to vertices that represent solutions. The loss $L^{\mathscr{E}}_{s}$ at a point $\bold{h} \in \cat{M}_s$ is computed from the function
\be\label{eq:dynamic_gradient}
\Pi_V \Delta \textrm{step}^t( \bold{h}, -, q_{\text{start}} ): \Sigma^{U'} \lto (\Delta \Sigma)^V\,.
\ee
While the inputs to this function are binary sequences the outputs are in general sequences of distributions, because of the uncertainty in the code $\bold{h}$. In the language of Section \ref{section:der_tm} we can view \eqref{eq:dynamic_gradient} as the result of applying time evolution to a dynamical system, namely, the naive Bayesian observer of the operation of $\cat{U}$ who has some initial uncertainty about the contents of the description tape, and thus uncertainty about \emph{which} Turing machine $M$ is being simulated. If $\bold{h} = c(M) \in \{0,1\}^a$ has no uncertainty, the naive Bayesian observer (as a dynamical system) is indistinguishable by \eqref{eq:deltastepvsstep} from the UTM itself simulating the machine with code $\bold{h}$. However, for general $\bold{h}$, at every stage of the operation of $\cat{U}$ where the UTM scans the description tape to decide what operation to execute, additional uncertainty will be introduced. From the frequentist point of view \eqref{eq:dynamic_gradient} is the operation for $t$ steps of a probabilistic deformation of $\cat{U}$ which, upon consulting the description tape, samples from the various distributions and acts according to the results of these samples (keeping in mind the unusual probabilistic semantics described in Section \ref{section:interp_naive}). 

\vspace{0.1cm}

We conclude this section with an analysis of the time complexity of evaluating the loss function $L^{\mathscr{E}}_{s}$ as a function of the number of states $s$.

\begin{proposition}\label{prop:timecomp} Given a fixed system of prefix equations $\mathscr{E}$ and assuming floating point operations and logarithms have time complexity $O(1)$ and that $\bold{h} \in [0,1]^a$ is represented by a list of $a$ floating point numbers, the time complexity with respect to the number of states $s$ of evaluating $L^{\mathscr{E}}_s(\bold{h})$ is $O( s^2 \log(s)^2 )$.
\end{proposition}
\begin{proof}
First we need to calculate a bound on $C = C(s)$. This is standard, but for lack of a reference we give the calculation using the UTM sketched in \cite[\S 1.4.1]{arorabarak} with various details filled in following Minsky \cite[\S 7.1]{minsky} with the obvious changes (for example Minsky's machines have tape alphabet $\{0,1\}$ while ours effectively have alphabet $\{\Box, 0,1\}$). We assume the description code $c(M)$ is as in \cite[\S 7.2]{minsky}. To simulate each step of $M$ the UTM begins with the heads on the description tape (tape three) and state tape (tape four) in their original positions; it then executes a \emph{locate} phase and a \emph{copy} phase. 

In the locate phase, the UTM moves the head of the description tape to the right, comparing the $3$-tuple of symbols on input, work and output tapes with the corresponding entries in each tuple on the description tape, and comparing bitwise the encoded state with the contents of the state tape. Locating the correct tuple takes at most $a + 4^3 s \lfloor \log(s) \rfloor$ steps\footnote{We use $a$ steps to scan each square of the description tape, plus $\lfloor \log(s) \rfloor$ steps per tuple to rewind the head of the state tape to its original position after the bitwise comparison.}. After this tuple is found the UTM writes symbols to the work and output tapes, and then enters the copy phase, during which it copies the state encoded in the located tuple onto the state tape. This takes at most $2 + 2\lfloor \log(s) \rfloor$ steps including the time to rewind the head on the state tape. Finally it needs $3$ steps to read the three directions and move the input, work and output tapes. At this point the UTM has finished simulating the given time step of $M$ and it just needs $a$ additional steps to rewind the head on the description tape to be back in its standard state. In total the UTM needs at most
\[
C(s) := 5 + 2a + (4^3 s + 2) \lfloor \log(s) \rfloor
\]
steps to simulate each step of $M$. This is $O(s \log(s))$. 

For evaluating $L^{\mathscr{E}}_s(\bold{h})$ we assume given a machine which performs floating point operations and logarithms in constant time, and that has constant time random access to a sufficiently large memory. We call this machine the \emph{processor}. Set $t_{max} = \max_{(x,y,t) \in \mathscr{E}} t$ and $T_{max} = C(s) t_{max}$. Considering the form of \eqref{eq:defnlossprefix} it suffices to find the time complexity of the processor evaluating $\Delta \textrm{step}^{T_{max}}(\bold{h}, x, q_{\textrm{start}})$ using the update rules of Definition \ref{defn:linearbtm}. To update the distributions for $\Mv, \Wr, S$ is $O(1)$ and we have to perform the calculation in \eqref{eq:Yu_nvb} for every tape square adjacent to a non-blank square. Initially there are at most
\[
E := \max_{(x,y,t) \in \mathscr{E}} |x| + a + \lfloor \log(s) \rfloor 
\]
non-blank tape squares, and in each time step at most two additional squares are written to, so over $T_{max}$ time steps we have at most $E + 2 T_{max}$ non-blank squares and thus per time step of the UTM updating all the $Y_u$ requires at most $4|\Sigma|(E + 2 T_{max})$ time steps of the processor. This shows the time complexity of the processor simulating each time step of the UTM is $O(s \log(s))$. Since the number of time steps is $O(s \log(s))$, this proves that the time complexity of the processor evaluating $L^{\mathscr{E}}_s(\bold{h})$ is $O(s^2 \log(s)^2)$.
\end{proof}

\begin{remark} We have described an approach to inductive inference based on gradient descent and the naive probabilistic extension of the step function of a UTM. Propagating naive probability provides a \emph{continuous relaxation} of the discrete step function (the formulas are given in Section \ref{section:bayesian}) so this is an instance of a well-known strategy: attempting to solve a discrete optimisation problem by continuous relaxation \cite[\S 2]{aardal}. 

There are other continuous relaxations of discrete algorithms being used for gradient descent in the recent literature on program synthesis: in \cite{terpret} the authors describe a continuous relaxation defined by propagating uncertainty through computational graphs using a Forward Marginals Gradient Descent (FMGD) approximation \cite[p.15]{terpret} which is (somewhat remarkably) identical to the conditional independence assumption that arises naturally from the semantics of linear logic, and which we have called here ``naive probability''. However in the end the details of how uncertainty is propagated through Turing machines \cite[\S B.1]{terpret} differs from the naive probabilistic extension $\Delta \pStep$, and so our synthesis process also differs from theirs.

Another example is \cite{evans} which introduces \emph{differentiable inductive logic programming} via a continuous relaxation of the forward chaining algorithm of logic programming, once again based on propagating uncertainty. The principal choice made in defining this propagation is the use of multiplication $xy$ of variables $x,y$ as the continuous relaxation of logical conjunction $x \land y$ \cite[4.5.1]{evans}. This choice of relaxation is ``correct'' from the point of view of linear logic, as it corresponds to viewing tensor products as the linear form of conjunction. This is ultimately the source of the multiplications appearing in the formulas of Lemma \ref{lemma:explicitdeltastep} and Definition \ref{defn:linearbtm} defining the naive probabilistic extension. The semantics of differential linear logic therefore provides a principled justification for some of the features of these continuous relaxations that otherwise appear \emph{ad hoc}.
\end{remark}

\subsection{The shift machine}\label{section:shift_machine}

In the previous section we considered a rather complicated example of the gradient descent flow of Definition \ref{defn:grad_descent}. In this section we consider instead a minimal example, which still exhibits some of the key features. The Turing machine $M$ has alphabet and states
\begin{align*}
\Sigma &= \{ \Box, A, B, 0, \ldots, 9 \}\\
Q &= \{ q_{\textrm{start}}, q_{\textrm{halt}}, \text{goR}, \text{goLA}, \text{goLB} \}\,.
\end{align*}
This is a single tape Turing machine, with ``stay'' as an allowed head movement. We give the transition function $\delta$ formally below, but informally when $M$ is initialised on a tape of the following form (the underline indicates head position)
\be\label{eq:shiftmachinetape}
\Box\, \underline{n}\, \text{B}\, \text{A}\, \text{B}\, \text{B}\, \text{A}\, \text{B} \,\text{A}\, \text{B}\, \Box
\ee
it will move the string of $A$'s and $B$'s leftwards by $n$ steps, filling in the right hand end of the string with $A$'s so that the string length is invariant. We refer to $n$ as as the \emph{counter} and the string of $A$'s and $B$'s as the \emph{input string}. For instance, if the machine is initialised in the above state with $n = 2$ the tape will eventually read
\[
\Box\, \underline{0}\, \text{B} \,\text{B}\, \text{A} \,\text{B}\, \text{A} \,\text{B}\, \text{A}\, \text{A}\, \Box\,.
\]
Once the counter reaches zero, the machine halts. The rules are:
\begin{align*}
\delta(\sigma, q_{\textrm{halt}}) &= (\sigma, q_{\textrm{halt}}, \textrm{stay})\\
\delta(n, q_{\textrm{start}}) &= (n-1,\text{goR}, \textrm{right}) && 0 < n \le 9\\
\delta(0,q_{\textrm{start}}) &= (0, q_{\textrm{halt}}, \textrm{stay})\\
\delta(A, q) &= (A, \text{goR}, \textrm{right})\\
\delta(B, q) &= (B, \text{goR}, \textrm{right})\\
\delta(\Box, \text{goR}) &= (\Box, \text{goLA}, \textrm{left})\\
\delta(i, \text{goL}j) &= (j, \text{goL}i, \textrm{left}) && i,j\in \{A,B\}\\
\delta(n, \text{goL}i) &= (n-1, \text{goR}, \textrm{right}) && i \in \{A,B\}, 1 < n \le 9\\
\delta(1, \text{goL}i) &= (0, q_{\textrm{halt}}, \textrm{stay}) && i \in \{A,B\}
\end{align*}
We suppose that our naive Bayesian observer is uncertain about whether the counter is initially $0$ or $2$, and also about whether the first symbol of the input string is $A$ or $B$. The observer's state of knowledge is therefore represented by distributions
\begin{align*}
v_{\text{counter}} &:= (1-h) \cdot 0 + h \cdot 2 \in \Delta \{0,2\}\,,\\
v_{\text{string}} &:= (1-k) \cdot B + k \cdot A \in \Delta \{A,B\}\,.
\end{align*}
In the context of gradient descent our weight vector is therefore
\[
(v_{\text{counter}},v_{\text{string}}) \in \Delta \{0,2\} \times \Delta\{A,B\}\,,
\]
which we identify in the obvious way with the point $\bold{h} = (h,k) \in [0,1]^2$. In our examples below we also assume the input string has length $3$, so the initial state of the tape is
\be\label{eq:shiftmachinetapeex}
\Box\, \underline{v_{\text{counter}}}\, v_{\text{string}}\, a_2\, a_3\, \Box
\ee
where $a_i \in \{A, B\}$ and we write $\{ a_j, \bar{a}_j \} = \{ A, B \}$. 
Suppose now the observer is told that that the machine eventually halts with the first square of the input string equal to $A$. With $U = \{0,1\}, U' = \{2,3\}$ the dataset is $\mathscr{D} = \{(\bold{a}, \bold{b}, V, t)\}$ with
\begin{align*}
V &= \{ 1 \}\,,\\
\bold{a} &= (a_2,a_3) \in ( \Delta \{A, B\} )^{\{2,3\}}\,,\\
\bold{b} &= A \in (\Delta \{A, B\})^V\,,
\end{align*}
and $t$ sufficiently large that the machine has finished modifying the input string and halted with the head on the counter. The loss $L_{\mathscr{D}}(\bold{h})$ associated to this dataset is computed as follows. The naive probability distribution which describes the first character of the input string when the machine has halted is
\[
S := (1-h)^2 (1-k) \cdot B + (1-h)^2 k \cdot A + \sum_{i=2}^{3} \binom{2}{i-1} h^{i-1} (1-h)^{3-i} \cdot a_i\,.
\]
Define $S_A, S_B$ by $S = S_A \cdot A + S_B \cdot B$. Then\footnote{In principle we should compute the relative entropy using $\varepsilon_\mu(S)$ which would introduce some small probability that the tape symbol is blank or an integer. To avoid the unnecessary complexity this introduces, we implicitly assume throughout that $S_A$ is nonzero and take the $\mu \to 0$ limit.}
\begin{align*}
L_{\mathscr{D}}(h,k) &= \dkl( A \,\vert\vert\, S ) + \lambda R = -\ln(S_A) + \lambda R\,.
\end{align*}
In the following examples we consider various initial values of the pair $(a_2,a_3)$ and look at the induced gradient descent path on the coordinates $(h,k) \in [0,1]^2$ assuming that initially $h = k$ are equal and small. This encodes the situation where the naive Bayesian observer is originally close to certain both that the counter is $0$ and that the first symbol of the input string is a $B$. They must revise this belief in light of the information in $\mathscr{D}$.

\begin{example}\label{example:shiftmachinesecond} Suppose $(a_2,a_3) = (A,A)$. Then
\begin{gather*}
S = \Big[ 1 - (1-h)^2(1-k) \Big] \cdot A + (1-h)^2(1-k) \cdot B\,\\
L_{\mathscr{D}} = - \ln\big( 1 - (1-h)^2(1-k) \big) + \lambda R\,.
\end{gather*}
With $h$ on the horizontal axis and $\lambda = 2$, a representative set of flows along $-\nabla L_{\mathscr{D}}$ is given in Figure \ref{fig:flow1}.
\begin{figure}
\includegraphics[scale=0.4]{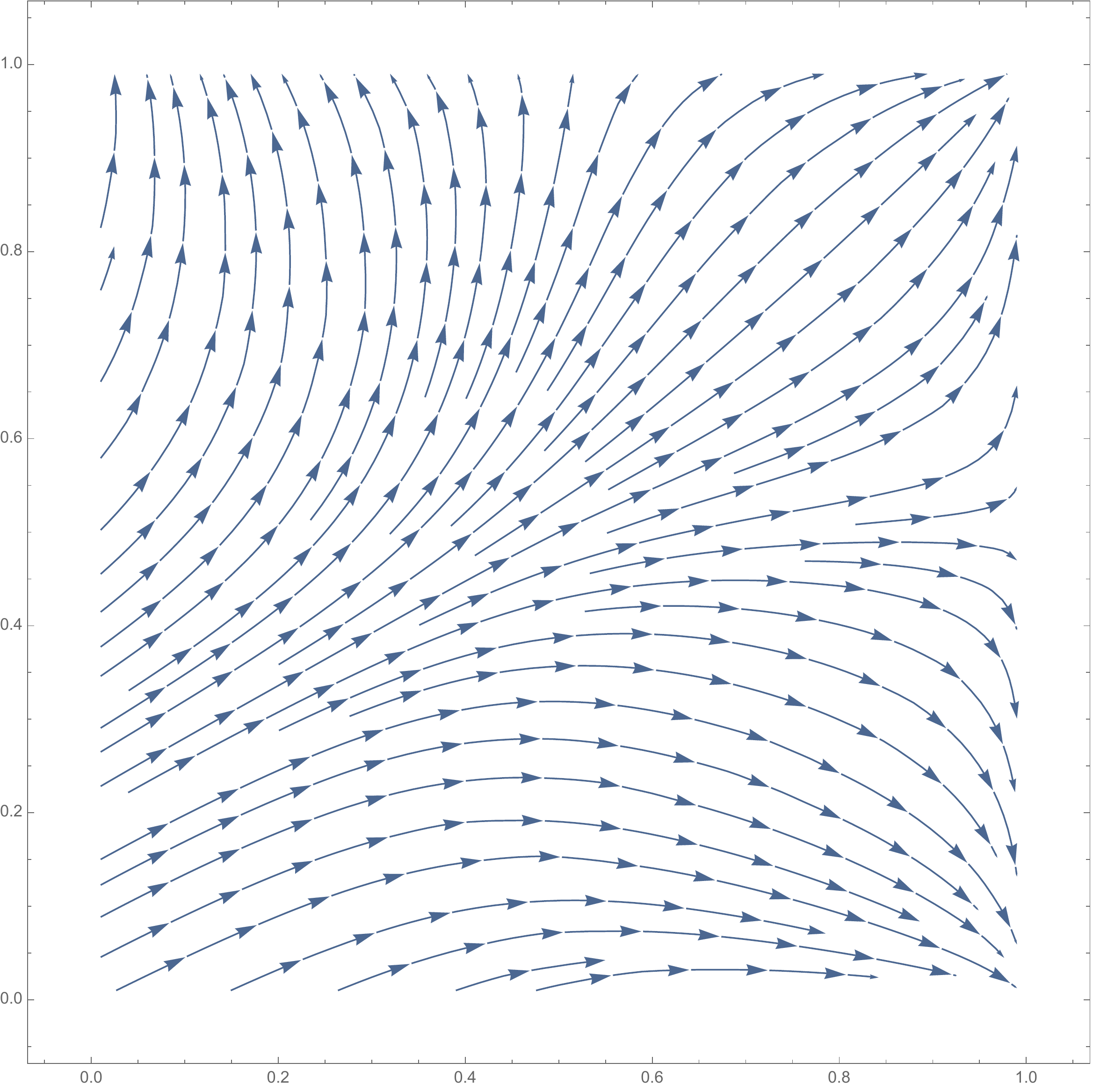}
\centering
\caption{Gradient descent flowlines for Example \ref{example:shiftmachinesecond}.}\label{fig:flow1}
\end{figure}
The paths beginning with $h,k$ both small all converge to $(1,0)$, representing a revision of belief in favour of the counter being originally a $2$, while leaving the belief that the input string began with an $A$ unchanged. We can contrast this with the probability distribution for the first symbol of the input string calculated by an observer using the normal rules of Bayesian probability:
\begin{align*}
S^{\std} &= (1-h)\left[ (1-k) \cdot B + k \cdot A \right] + h \cdot A\\
&= \big[ (1-h)k + h \big] \cdot A + (1-h)(1-k) \cdot B\,.
\end{align*}
If we were to compute the loss with this distribution instead, we obtain
\[
L_{\mathscr{D}} = -\ln( 1 - (1-h)(1-k) ) + \lambda R\,.
\]
The gradient descent path beginning at $(0.05,0.05)$ is simply a straight line to the upper right hand corner: the observer using standard probability therefore assigns causality to both the counter and first input string symbol, and changes both of them.
\end{example}


\begin{example}\label{example:second_flow} Suppose $(a_2,a_3) = (A,B)$. Then
\[
S = \Big[ (1-h)^2 k + 2h(1-h) \Big] \cdot A + \Big[ (1-h)^2(1-k) + h^2 \Big] \cdot B\,.
\]
With the same conventions as before, the flows of $- \nabla L_{\mathscr{D}}$ are given in Figure \ref{fig:flow2}.
\begin{figure}
\includegraphics[scale=0.4]{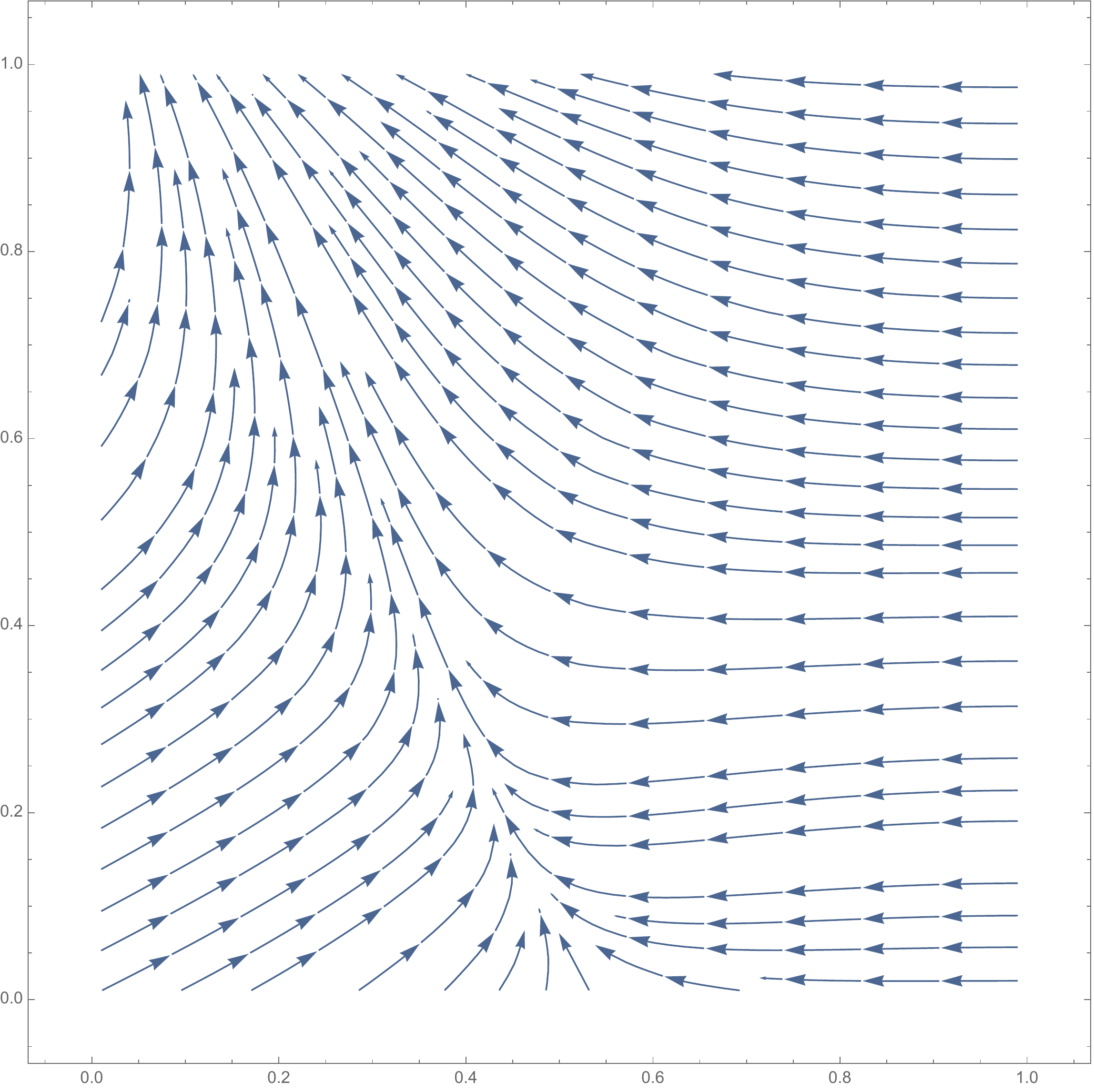}
\centering
\caption{Gradient descent flowlines for Example \ref{example:second_flow}.}\label{fig:flow2}
\end{figure}
Note that when $h,k$ are initially small, the observer will begin by revising their belief in favour of the original value of the counter $n$ being a $2$, even though this is not a valid explanation of the final state of the machine (since $a_3 = B$). This strange behaviour is caused by the term $h^2 \cdot B$ being initially too small to affect the flow; however once $h$ grows sufficiently large this term drives the flow to the top left.
\end{example}

\subsection{Naive causality}\label{section:attribution}

We observed in Example \ref{example:shiftmachinesecond} that, given an error in the output and two possible explanations the naive Bayesian observer selected a single cause while the standard observer selected both causes, in the sense that the gradient descent flow following the naive and standard vector fields led to flipping either just for the bit for the counter (in the naive case) or both the counter and the input string (in the standard case). In this section we make this into a general point, assuming the following very simple situation:
\begin{itemize}
\item[(a)] Suppose $\mathscr{D} = \{ (\bold{a}, \bold{b}, V, t) \}$ where $V = \{ v \}$ is a single output location.
\item[(b)] We only consider binary input and output distributions, or more formally, we only consider distributions $\bold{h}$ of the form (writing $n = |U|$)
\[
\bold{h} = \big( (1-h_1) \cdot \sigma_0^1 + h_1 \cdot \sigma_1^1, \ldots, (1-h_n) \cdot \sigma_0^u + h_n \cdot \sigma_1^n \big) \in (\Delta \Sigma)^U\,,
\]
where $\{ \sigma_0^i, \sigma_1^i \}$ is the same set of two symbols for each $i$, and we assume that for all such distributions $\Pi_V \Delta \textrm{step}^t( \bold{h}, \bold{a}, q_{\textrm{start}} )$ lies in $\Delta \{ \lambda_0, \lambda_1 \}$ for some symbols $\lambda_0, \lambda_1$.
\item[(c)] With $\underline{\sigma}_0 = (\sigma_0^1, \ldots, \sigma_0^n) \in \Sigma^U$ we assume that for all $\underline{\sigma} \in \Sigma^U$
\[
\Pi_V \Delta \textrm{step}^t(\underline{\sigma}, \bold{a}, q_{\textrm{start}}) = \begin{cases} \lambda_0 & \underline{\sigma} = \underline{\sigma}_0 \\
\lambda_1 & \text{otherwise}
\end{cases}
\]
\item[(d)] $\bold{b}$ assigns probability $1$ to $\lambda_1 \in \Sigma$.
\end{itemize}
The observer is told that the correct output is $\lambda_1$ and that flipping \emph{any subset of bits} of his initial belief in $\underline{\sigma}_0$ will bring his predicted output into line with this new information. The question answered by gradient descent is: \emph{which bits to flip}? As Example \ref{example:shiftmachinesecond} above shows, if hypothesis (c) does not hold then the dynamics of the gradient descent can be more complicated than what we describe below. 

Under standard probability the loss is
\begin{align*}
L^{\operatorname{std}}_{\mathscr{D}} &= -\ln\Big[ 1 - \prod_{i=1}^n ( 1 - h_i ) \Big] + \lambda R\,.
\end{align*}
If we assume that initially $\bold{h}_0$ has all coordinates equal to $h$ then
\[
\frac{\partial L^{\operatorname{std}}_{\mathscr{D}}}{\partial h_j} = \frac{- (1-h)^{n-1}}{1 - (1-h)^{n}} + \lambda \frac{\partial R}{\partial h_j}\,.
\]
The gradient descent will increase all coordinates by the same amount, so we will eventually reach the vertex of $[0,1]^U$ where all coordinates are set to one: that is, the observer \emph{flips all the bits}. We contrast this to the naive probability
\begin{align*}
L_{\mathscr{D}} &= -\ln\Big[ 1 - \prod_{i=1}^n ( 1 - h_i )^{n_i} \Big] + \lambda R\,.
\end{align*}
Supposing again that $\bold{h}_0$ has all coordinates equal to $h$, and writing $N = \sum_{i=1}^n n_i$,
\[
\frac{\partial L_{\mathscr{D}}}{\partial h_j} = \frac{- n_j (1-h)^{N-1}}{1 - (1-h)^N} + \lambda \frac{\partial R}{\partial h_j}\,.
\]
The coordinates with high $n_j$ will tend to increase fastest, and so the gradient vector field will converge to a point where coordinates with high degree are set to $1$. Thus the naive observer preferentially flips bits with high degree.

\appendix

\section{Properties of naive probability}\label{section:appnv}

The definition of the naive probabilistic extension given in Proposition \ref{prop:plain_prob} extends in a natural way to component-wise plain proofs:

\begin{corollary}\label{cor:plain_prob} Suppose given a proof
\[
\psi: {!} A_1, \ldots, {!} A_r \vdash {!} B_1 \otimes \cdots \otimes {!} B_s
\]
which is component-wise plain, and that for each plain component $\psi_j$ of $\psi$ we are given a proof $\pi_j$ from which $\psi_j$ may be constructed as in Setup \ref{setup:psiandpi}. The $A_i$-degree of $\psi_j$ in this construction we denote by $n^j_i$. Moreover, suppose that
\begin{itemize}
\item $\cat{P}_i$ is a finite set of proofs of $A_i$ for $1 \le i \le r$,
\item $\cat{Q}_j$ is a finite set of proofs of $B_j$ for $1 \le j \le s$,
\item such that, for each $1 \le j \le s$, there is an inclusion
\[
\Big\{ \pi_j( X_1,\ldots,X_r ) \l X_i \in \cat{P}_i^{n^j_i} \Big\} \subseteq \cat{Q}_j\,,
\]
\item and for $1 \le j \le s$ the indexed set $\big\{ \dntn{\nu} \}_{\nu \in \cat{Q}_j}$ is linearly independent in $\dntn{B_j}$.
\end{itemize}
Then there is a unique function
\[
\Delta \psi: \prod_{i=1}^r \Delta \cat{P}_i \lto \prod_{j=1}^s \Delta \cat{Q}_j
\]
which makes the diagram
\[
\xymatrix@C+2pc@R+1pc{
\bigotimes_{i=1}^r {!} \dntn{A_i} \ar[r]^-{\dntn{\psi}} & \bigotimes_{j=1}^s {!} \dntn{B_j}\\
\prod_{i=1}^r \Delta \cat{P}_i\ar[u]^-{\iota} \ar[r]_-{\Delta \psi} & \prod_{j=1}^s \Delta \cat{Q}_j \ar[u]_-{\iota}
}
\]
commute.
\end{corollary}
\begin{proof}
Since $\psi$ is component-wise plain the denotation $\dntn{\psi}$ may be constructed from the denotations of its components $\psi_j: {!} A_1, \ldots, {!} A_r \vdash B_j$ by the diagram
\[
\xymatrix@C+3pc{
\bigotimes_{i=1}^r {!} \dntn{A_i} \ar[r]^-{\otimes_i \Delta^{s-1}} & \Big[ \bigotimes_{i=1}^r {!} \dntn{A_i} \Big]^{\otimes s} \ar[r]^-{ \otimes_i \operatorname{prom}\dntn{\psi_j} } & \bigotimes_{j=1}^s {!} \dntn{B_j}
}
\]
which sends
\[
\bigotimes_{i=1}^r \vacu_{\zeta_i} \longmapsto \bigotimes_{j=1}^s \vacu_{ \dntn{\psi_j} \big( \bigotimes_{i=1}^r \vacu_{\zeta_i} \big)}\,.
\]
Now the conclusion follows by applying Proposition \ref{prop:plain_prob} to each $\psi_j$. The uniqueness is a consequence of Lemma \ref{lemma:injectiveiota}, which shows the right hand vertical map is injective.
\end{proof}

\begin{definition} We refer to $\Delta \psi$ as the \emph{naive probabilistic extension} of $\psi$. 
\end{definition}

\begin{remark} By construction the diagram
\[
\xymatrix@C+2pc@R+1pc{
\bigotimes_{i=1}^r {!} \dntn{A_i} \ar[r]^-{\dntn{\psi}} & \bigotimes_{j=1}^s {!} \dntn{B_j} \ar[r]^-{w_j} & \dntn{B_j}\\
\prod_{i=1}^r \Delta \cat{P}_i\ar[u]^-{\iota} \ar[r]_-{\Delta \psi} & \prod_{j=1}^s \Delta \cat{Q}_j \ar[u]_-{\iota} \ar[r]_-{\proj_j} & \Delta \cat{Q}_j \ar[u]_-{\dntn{-}}
}
\]
commutes, where $w_j$ is the tensor product of the dereliction map ${!} \dntn{B_j} \lto \dntn{B_j}$ with the counits ${!} \dntn{B_i} \lto k$ for all $i \neq j$, and $\proj_j$ is the $j$th projection. Since the top row is the denotation $\dntn{\psi_j}$ of the $j$th plain component of $\psi$, this shows that
\[
\proj_j\, \circ \,\Delta \psi = \Delta \psi_j\,.
\]
\end{remark}

\begin{remark} It is easy to see that the diagram
\[
\xymatrix@C+2pc@R+1pc{
\prod_{i=1}^r \Delta \cat{P}_i \ar[r]^-{\Delta \psi} & \prod_{j=1}^s \Delta \cat{Q}_j \\
\prod_{i=1}^r \cat{P}_i \ar[u]^{\inc} \ar[r]_-{\psi} & \prod_{j=1}^s \cat{Q}_j \ar[u]_{\inc}
}
\]
commutes.
\end{remark}


\begin{lemma}\label{lemma:cut_componentwiseplainnv} Suppose given two component-wise plain proofs
\begin{align*}
\psi: {!} A_1,\ldots,{!} A_r &\vdash {!} B_1 \otimes \cdots \otimes {!} B_s\,,\\
\phi: {!} B_1,\ldots,{!} B_s &\vdash {!} C_1 \otimes \cdots \otimes {!} C_t\,.
\end{align*}
and let $\phi \l \psi$ denote their cut \cite[Proposition 3.6]{clift_murfet2}. This proof is component-wise plain and $\Delta( \phi \l \psi ) = \Delta( \phi) \circ \Delta(\psi)$.
\end{lemma}
\begin{proof}
First we have to explain precisely what we mean by the naive probabilistic extension of $\phi \l \psi$, since this depends on choices of sets of proofs. By definition the two squares in the diagram
\[
\xymatrix@C+2pc@R+1pc{
\bigotimes_{i=1}^r {!} \dntn{A_i} \ar[r]^-{\dntn{\psi}} & \bigotimes_{j=1}^s {!} \dntn{B_j} \ar[r]^-{\dntn{\phi}} & \bigotimes_{k=1}^t {!} \dntn{C_k}\\
\prod_{i=1}^r \Delta \cat{P}_i\ar[u]^-{\iota} \ar[r]_-{\Delta \psi} & \prod_{j=1}^s \Delta \cat{Q}_j \ar[u]_-{\iota} \ar[r]_-{\Delta \phi} & \prod_{k=1}^t \Delta \cat{R}_k \ar[u]_-{\iota}
}
\]
commute, where the $\cat{R}_k$ are the chosen proofs of $C_k$. By \cite[Proposition 3.6]{clift_murfet2} the cut $\phi \l \psi$ is component-wise plain, and the proof of that lemma shows furthermore that the components of the cut are plain proofs
\[
(\phi \l \psi)_k : s\, {!} A_1, \ldots, s\, {!} A_r \vdash C_k\,.
\]
which may be constructed as in Setup \ref{setup:psiandpi} from proofs $\pi_k$ with $A_i$-degrees $m^k_i$ such that
\[
\Big\{ \pi_k( X_1, \ldots, X_r ) \l X_i \in \cat{P}_i^{m^k_i} \Big\} \subseteq \cat{R}_k
\]
holds for each $k$. Hence $\Delta( \phi \l \psi )$ exists and by uniqueness it must be equal to $\Delta \phi \circ \Delta \psi$.
\end{proof}

\section{Standard probability}\label{section:stdprob}

We defined the naive probabilistic extension of $\psi$ using the naive encoding $\iota$ of probability distributions over input proofs (Definition \ref{defn:iota_notation}). In this section we study the standard encoding, that is, the first vector in \eqref{eq:twovectors_dist}.

\begin{definition}\label{defn:iota_notation_std} Given $A_1,\ldots,A_r$ and for each $i$ a finite set of proofs $\cat{P}_i$ of $A_i$, define
\begin{gather*}
\iota^{\std}: \Delta\big( \cat{P}_1 \times \cdots \times \cat{P}_r \big) \lto {!} \dntn{A_1} \otimes \cdots \otimes {!} \dntn{A_r}\\
\iota^{\std}(v) = \sum_{\rho_1 \in \cat{P}_1} \cdots \sum_{\rho_r \in \cat{P}_r} v_{(\rho_1,\ldots,\rho_r)} \bigotimes_{i=1}^r \vacu_{\dntn{\rho_i}}\,.
\end{gather*}
\end{definition}

\begin{lemma}\label{lemma:injectiveiotastd} If for all $1 \le i \le r$ the denotation map $\cat{P}_i \lto \dntn{A_i}$ is injective, then $\iota^{\std}$ is also injective.
\end{lemma}
\begin{proof} It suffices to show that the vectors $W_\rho = \bigotimes_{i=1}^r \vacu_{\dntn{\rho_1}}$ are linearly independent, as $\rho$ ranges over $\prod_{i=1}^r \cat{P}_i$. But these are distinct group-like elements by \cite[\S 2.3]{clift_murfet2} and the hypothesis, so this follows from \cite[Proposition 3.2.1(b)]{sweedler}.
\end{proof}

The following setup differs from that in Corollary \ref{cor:plain_prob}. We suppose given a proof
\[
\psi: {!} A_1, \ldots, {!} A_r \vdash {!} B_1 \otimes \cdots \otimes {!} B_s
\]
which \emph{need not be} component-wise plain, and suppose that
\begin{itemize}
\item $\cat{P}_i$ is a finite set of proofs of $A_i$ for $1 \le i \le r$,
\item $\cat{Q}_j$ is a finite set of proofs of $B_j$ for $1 \le j \le s$,
\item such that there is an inclusion
\[
\Big\{ \psi( x_1,\ldots,x_r ) \l x_i \in \cat{P}_i \Big\} \subseteq \prod_{j=1}^s \cat{Q}_j\,,
\]
\item and for $1 \le j \le s$ the denotation map $\cat{Q}_j \lto \dntn{B_j}$ is injective.
\end{itemize}
Then by Lemma \ref{lemma:injectiveiotastd} there is a unique $\Delta^{\std} \psi$ making the top square in the diagram
\[
\xymatrix@C+2pc@R+1pc{
\bigotimes_{i=1}^r {!} \dntn{A_i} \ar[r]^-{\dntn{\psi}} & \bigotimes_{j=1}^s {!} \dntn{B_j}\\
\Delta \prod_{i=1}^r \cat{P}_i\ar[u]^-{\iota^{\std}} \ar[r]^-{\Delta^{\std} \psi} & \Delta \prod_{j=1}^s \cat{Q}_j \ar[u]_-{\iota^{\std}}
}
\]
Moreover, it is easy to see that the function $\Delta^{\std} \psi$ is given by
\[
\Delta^{\std}\psi(v) = \sum_{\tau_1 \in \cat{Q}_1} \cdots \sum_{\tau_s \in \cat{Q}_s} \sum_{\substack{\rho \text{ s.t.}\\\psi(\rho) = \tau}} v_\rho \cdot (\tau_1,\ldots,\tau_s)\,.
\]
We call this the \emph{standard probabilistic extension} of $\psi$.

\section{Boolstep versus relstep}\label{section:smooth_model}

In the main text the encoding $\pRelstep$ of \cite[\S 5.1]{clift_murfet2} is used to define the naive probabilistic extension of the step function. This encoding takes as input a sequence of booleans and returns a sequence of booleans. A second encoding $\pBoolstep$ of this kind was given in \cite[\S 4.2]{clift_murfet2}, and in this appendix we show that the naive probabilistic extensions associated to these two encodings are the same.

\subsection{The boolstep encoding}

First we recall how the $\pBoolstep$ encoding is derived from the encoding of the step function which uses binary integers or more general $s$-lists. The aim is to give enough detail that the reader can swap in the $\pBoolstep$ encoding in place of the $\pRelstep$ encoding throughout Section \ref{section: turing machines}. The notation is as in Section \ref{section:prob_ext_tm}, so that $s = |\Sigma|$ and $n = |Q|$ and we fix a Turing machine $M$ with transition function $\delta$. For any type $A$ there is a proof
\[
\pStep_A: \tTur_{A^{s+1}} \vdash \tTur_A
\]
which encodes the step function \cite[Appendix A.2]{clift_murfet2}, where
\[
\tTur_A = {!} {_s} \tList_A \otimes {!} {_s} \tList_A \otimes {!} {_n} \tBool_A\,.
\]
For the definition of the type $\tList_A$ see \cite[Definition A.2]{clift_murfet2}. 

The algorithm $\pBoolstep$ can be described informally as follows. Given $a, b, c, d, t \ge 1$ a sequence of $s$-booleans of length $a + b$ can be concatenated into a pair of $s$-lists representing the left part of the tape (of length $a$) and the right part of the tape (of length $b$). Together with the initial state this gives a proof of $\tTur$ which may be fed as input to the cut of $t$ copies of $\pStep$. The resulting proof of $\tTur$ may be input into a proof which reads off the $c$ squares on the left part of the tape closest to the tape head, and the $d$ squares of the right part of the tape closest to the tape head. This yields a proof \cite[Definition 4.27]{clift_murfet2}
\be
{^t}\pBoolstep^{a,b,c,d}_A : (a+b) \,{!} {}_s \tBool_B, {!} {_n} \tBool_B \vdash \big({!} {}_s \tBool_A\big)^{\otimes c + d} \otimes {!} {_n} \tBool_A
\ee
for some power $B = A^{g(c,d,t)}$ such that
\[
{^t}\pBoolstep^{a,b,c,d}_A( \sigma_{-a+1},\ldots,\sigma_b, q ) = ( \tau_{-c+1},\ldots,\tau_d, q' )
\]
if and only if the Turing machine $M$ initialised in state $q$ and with tape
\[
\ldots, \Box, \sigma_{-a+1},\ldots, \underline{\sigma_0}, \sigma_1, \ldots, \sigma_b, \Box, \ldots
\]
with the underline indicating the position of the head, when run for $t$ steps is in state $q'$ and the contents of the tape in relative positions $[-c+1,d]$ are
\[
\tau_{-c+1}\,, \ldots\,, \underline{\tau_0} \,, \tau_1 \,, \ldots \,, \tau_d \,.
\]
These proofs are component-wise plain and so the denotations are morphisms of coalgebras
\be
{!} \dntn{{}_s \tBool_B}^{\otimes a+b} \otimes {!} \dntn{{}_n \tBool_B} \lto {!} \dntn{{}_s \tBool_A}^{\otimes c+d} \otimes {!} \dntn{ {}_n \tBool_A}\,.
\ee
With $\cat{P}_{tape}$ and $\cat{P}_{state}$ being defined as in \eqref{eq:ptape},\eqref{eq:pstate} we have (by Corollary \ref{cor:plain_prob}) that the naive probabilistic extension $\Delta {^t}\pBoolstep^{a,b,c,d}_A$ is the unique function making the diagram
\[
\xymatrix@C+6pc@R+1.5pc{
{!} \dntn{{}_s \tBool_B}^{\otimes a+b} \otimes {!} \dntn{ {}_n \tBool_B } \ar[r]^-{\dntn{{^t}\pBoolstep^{a,b,c,d}_A}} &  {!} \dntn{{}_s \tBool_A}^{\otimes c+d} \otimes {!} \dntn{{}_n \tBool_A} \\
(\Delta \Sigma)^{a} \times (\Delta \Sigma)^b \times \Delta Q \ar[u]^-{\iota} \ar[r]_-{\Delta {^t}\pBoolstep^{a,b,c,d}_A} & (\Delta \Sigma)^c \times (\Delta \Sigma)^d \times \Delta Q \ar[u]_-{\iota}
}
\]
commute, where $\iota$ is from Definition \ref{defn:iota_notation}. In principle we know how to calculate this naive probabilistic extension, because we know the denotation of $\pStep_A$ from \cite[Lemma 4.19]{clift_murfet2}. However, there is a subtlety: in order to make use of this calculation from \cite{clift_murfet2} we have to arrange that the basetype $A$ has $\dim\dntn{A}$ sufficiently large that all the binary integers involved in computing $\dntn{\pBoolstep}$ from $\dntn{\pStep}$ are linearly independent. However, taking $\dim\dntn{A}$ sufficiently large we can explicitly compute $\Delta {^t} \pBoolstep_A^{a,b,c,d}$ and thereby show that it gives rise to the same naive probabilistic extension of the step function as $\pRelstep$. This is the main result of this appendix, and occurs as Corollary \ref{corollary:app1}.


\subsection{The naive probabilistic extension of step}

We denote by $\delta$ our Turing machine, by $\Sigma$ the tape alphabet and by $Q$ the set of states. We set $s = |\Sigma|$ and $n = |Q|$. Associated to $\delta$ is the \emph{string step function}
\[
\textrm{strstep}: \Sigma^* \times \Sigma^* \times Q \lto \Sigma^* \times \Sigma^* \times Q
\]
which is defined as follows: if $S \in \Sigma^*$ is the string on the left hand part of the tape (the rightmost symbol in $S$ being the symbol under the tape head) and $T \in \Sigma^*$ is the \emph{reversal} of the string on the right hand part of the tape (so the rightmost symbol in $T$ is adjacent to the tape head) and $q \in Q$ is the internal state of the Turing machine defined by $\delta$, then after one time step the Turing machine is in configuration $\textrm{step}(S,T,q)$. In \cite{clift_murfet2} this version of the step function is denoted ${_\delta} \textrm{step}$ rather than ${_\delta} \textrm{strstep}$, but in this paper the former notation denotes instead the map in \eqref{eq:stepfunction}. This function is encoded by the proof
\[
\pStep_A: \tTur_{A^3} \vdash \tTur_A\,.
\]
Next we discuss the naive probabilistic extension of this proof. To define a naive probabilistic extension we need to fix finite sets of input proofs and output proofs, which means we need to bound the lengths of the sequences involved.

With $S, T \in \Sigma^*$ and $q \in Q$ suppose $(S',T',q') = \textrm{strstep}(S,T,q)$. Then the lengths satisfy $|S'| \le |S| + 1$ and $|T'| \le |T| + 1$, and so we may define:

\begin{definition} For $a,b \ge 1$ let $\textrm{strstep}_{m,n}$ denote the unique function making
\be
\xymatrix@C+3.5pc{
\Sigma^* \times \Sigma^* \times Q \ar[r]^-{\textrm{strstep}} & \Sigma^* \times \Sigma^* \times Q \\
\Sigma^{\le a} \times \Sigma^{\le b} \times Q \ar[u] \ar[r]_-{\textrm{strstep}_{a,b}} & \Sigma^{\le a+1} \times \Sigma^{\le b+1} \times Q \ar[u]
}
\ee
commute, where the vertical maps are the inclusions and $\Sigma^{\le a} = \cup_{k \le a} \Sigma^k$.
\end{definition}

The idea is to restrict $\pStep$ to inputs $(S,T,q)$ where $|S| \le a$ and $|T| \le b$ and then apply Corollary \ref{cor:plain_prob} to define the corresponding naive probabilistic extension. This restriction is done formally by using the following sets of proofs:
\begin{align*}
\cat{P}_1 &= \big\{ \underline{S}_{A^3} \}_{S \in \Sigma^{\le a}} && \cat{Q}_1 = \big\{ \underline{S}_{A} \}_{S \in \Sigma^{\le a + 1}}\\
\cat{P}_2 &= \big\{ \underline{T}_{A^3} \}_{T \in \Sigma^{\le b}} && \cat{Q}_2 = \big\{ \underline{T}_{A} \}_{T \in \Sigma^{\le b + 1}}\\
\cat{P}_3 &= \big\{ \underline{q}_{A^3} \}_{q \in Q} && \cat{Q}_3 = \big\{ \underline{q}_{A} \}_{q \in Q}\,.
\end{align*}
To apply Corollary \ref{cor:plain_prob} we will need to know that the sets of denotations $\dntn{ \cat{Q}_j }$ are linearly independent for $1 \le j \le 3$. While this is clear for $\cat{Q}_3$, in the case of $\cat{Q}_1, \cat{Q}_2$ we are asking about the linear independence of denotations in $\dntn{{_s} \tList_A}$ which places a restriction on $A$.

\begin{remark} By \cite[Remark B.11]{clift_murfet2} as long as $\dim(\dntn{A}) > c/2$ the set
\[
\big\{ \dntn{ \underline{S}_A } \big\}_{S \in \Sigma^{\le c}}
\]
is linearly independent in $\dntn{{_s} \tList_A}$.
\end{remark}

\begin{proposition}\label{prop:deltastep} For any type $A$ with
\be\label{eq:dimagood}
\dim(\dntn{A}) > \tfrac{1}{2}\max\{a + 1,b + 1\}
\ee
there is a unique function $\Delta \pStep_{A,a,b}$ making the diagram
\[
\xymatrix@C+4.5pc@R+1pc{
\dntn{\tTur_{A^3}} \ar[r]^-{\dntn{\pStep_A}} & \dntn{\tTur_A}\\
\Delta \Sigma^{\le a} \times \Delta \Sigma^{\le b} \times \Delta Q \ar[u]^-{\iota} \ar[r]_-{\Delta  \pStep_{A,a,b}} & \Delta \Sigma^{\le a+1} \times \Delta \Sigma^{\le b+1} \times \Delta Q \ar[u]_-{\iota}
}
\]
commute.
\end{proposition}
\begin{proof}
Immediate from Corollary \ref{cor:plain_prob}.
\end{proof}




We repeat here the notation used in \cite[Lemma 4.19]{clift_murfet2}, where
\be\label{eq:alphabetagamma}
\alpha = \sum_{i=1}^s a_i \cdot S_i \sigma_i ,, \qquad
\beta = \sum_{j=1}^t b_j \cdot T_j \tau_j\,, \qquad
\gamma = \sum_{k=1}^r c_k \cdot q_k\,,
\ee
are generic elements of $\Delta \Sigma^{\le a}, \Delta \Sigma^{\le b}$ and $\Delta Q$ respectively, and $(\hat{\sigma}^k_i, \hat{q}_i^k, \hat{d}_i^k) = \delta(\sigma_i, q_k)$.

\begin{lemma}\label{lemma:formulas_for_psi} In the situation of the proposition, with $\alpha, \beta, \gamma$ as above, set
\[
(\alpha', \beta', \gamma') = \Delta \pStep_{A,a,b}( \alpha, \beta, \gamma )\,.
\]
Then
\begin{align*}
\alpha' &= \Big( \sum_{i,k} a_i c_k \delta_{\hat{d}_i^k = 0} \Big)\Big( \sum_i a_i \cdot S_i \Big) + \Big( \sum_{i,k} a_i c_k \delta_{\hat{d}^k_i = 1} \Big) \Big( \sum_{i,i',j,k} a_i a_{i'} b_j c_k \cdot S_i \hat{\sigma}_{i'}^k \tau_j \Big)\,,\\
\beta' &= \Big( \sum_{i,k} a_i c_k \delta_{\hat{d}^k_i = 1} \Big) \Big( \sum_j b_j \cdot T_j \Big) + \Big( \sum_{i,k} a_i c_k \delta_{\hat{d}^k_i = 0} \Big) \Big( \sum_{i,j,j',k} b_j b_{j'} a_i c_k \cdot T_j \tau_{j'}\hat{\sigma}^k_i \Big)\,,\\
\gamma' &= \sum_{i,k} a_i c_k \cdot \hat{q}_i^k\,.
\end{align*}
\end{lemma}
\begin{proof}
An immediate consequence of \cite[Lemma 4.19]{clift_murfet2}.
\end{proof}


\begin{lemma} There is a unique function
\[
\Delta \textrm{strstep}: \Delta( \Sigma^* ) \times \Delta( \Sigma^* ) \times \Delta Q \lto \Delta( \Sigma^* ) \times \Delta( \Sigma^* ) \times \Delta Q
\]
with the property that for any $a,b \ge 1$ and type $A$ satisfying \eqref{eq:dimagood}, the diagram
\[
\xymatrix@C+5pc@R+1pc{
\Delta( \Sigma^* ) \times \Delta( \Sigma^* )\times \Delta Q \ar[r]^-{\Delta \textrm{strstep}} & \Delta( \Sigma^* ) \times \Delta( \Sigma^* )\times \Delta Q \\
\Delta( \Sigma^{\le a} ) \times \Delta(\Sigma^{\le b}) \times Q \ar[u] \ar[r]_-{\Delta \pStep_{A,a,b}} & \Delta(\Sigma^{\le a+1}) \times \Delta(\Sigma^{\le b+1}) \times \Delta Q \ar[u]
}
\]
commutes, where the vertical maps are induced by the inclusions $\Sigma^{\le a} \subseteq \Sigma^*$ and $\Sigma^{\le b} \subseteq \Sigma^*$.
\end{lemma}
\begin{proof}
First of all we observe that if $A, B$ are two types satisfying \eqref{eq:dimagood} then $\Delta \pStep_{A,a,b} = \Delta \pStep_{B,a,b}$ by Lemma \ref{lemma:formulas_for_psi}. Secondly, note that if $a \le a'$ and $b \le b'$ then the diagram
\[
\xymatrix@C+4pc@R+1pc{
\Delta \Sigma^{\le a'} \times \Delta \Sigma^{\le b'} \times \Delta Q \ar[r]^-{\Delta \pStep_{A,a',b'}} & \Delta \Sigma^{\le a'+1} \times \Delta \Sigma^{\le b'+1} \times \Delta Q \\
\Delta \Sigma^{\le a} \times \Delta \Sigma^{\le b} \times Q \ar[u] \ar[r]_-{\Delta \pStep_{A,a,b}} & \Delta\Sigma^{\le a+1} \times \Delta\Sigma^{\le b+1} \times \Delta Q \ar[u]
}
\]
commutes, from which the claim follows.
\end{proof}

\begin{definition} We call $\Delta \textrm{strstep}$ the \emph{naive probabilistic extension} of $\textrm{strstep}$.
\end{definition}

We define for $u \le 0$ a function
\begin{gather*}
L_u: \Sigma^* \lto \Sigma\,,\\
L_u( \sigma_{-l} \cdots \sigma_{-1} \sigma_0 ) = \begin{cases} \sigma_u & u \ge -l \\ \Box & \text{otherwise} \end{cases}
\end{gather*}
and for $u \ge 1$ a function (note the reversed order)
\begin{gather*}
R_u: \Sigma^* \lto \Sigma\,,\\
R_u( \sigma_{l} \cdots \sigma_{1} ) = \begin{cases} \sigma_u & u \le l \\ \Box & \text{otherwise} \end{cases}
\end{gather*}
These functions induce functions $\Delta L_u, \Delta R_u: \Delta(\Sigma^*) \lto \Delta \Sigma$ on distributions. We define
\begin{gather*}
\textrm{patch}: \Delta (\Sigma^*) \times \Delta (\Sigma^*) \lto (\Delta \Sigma)^{\mathbb{Z}, \Box}\\
\textrm{patch}( \alpha, \beta )(v) = \begin{cases} \Delta(L_v)(\alpha) & v \le 0 \\
\Delta(R_v)(\beta) & v \ge 1 \end{cases}
\end{gather*}

\begin{theorem}\label{theorem:strstepequalsstep} The diagram
\[
\xymatrix@C+5pc@R+1pc{
\Delta( \Sigma^* ) \times \Delta( \Sigma^* )\times \Delta Q \ar[r]^-{\Delta \textrm{strstep}}\ar[d]_-{\textrm{patch} \times 1} & \Delta( \Sigma^* ) \times \Delta( \Sigma^* )\times \Delta Q \ar[d]^-{\textrm{patch} \times 1}\\
(\Delta \Sigma)^{\mathbb{Z}, \Box} \times \Delta Q \ar[r]_-{\Delta \textrm{step}} & (\Delta \Sigma)^{\mathbb{Z}, \Box} \times \Delta Q
}
\]
commutes.
\end{theorem}
\begin{proof}
By direct calculation using Lemma \ref{lemma:formulas_for_psi}.
\end{proof}


In Lemma \ref{lemma:nvstepext} we showed that the naive probabilistic extensions $\Delta {_a} \pRelstep$ assemble to define the map $\Delta \textrm{step}$. The next result shows that this map, and its powers (see Lemma \ref{lemma:prelsteptnv}) may also be assembled from the naive probabilistic extensions of the proofs $\pBoolstep$. In this sense, while $\pRelstep$ and $\pBoolstep$ are genuinely different encodings of the step function (see \cite[Remark 5.7]{clift_murfet2}) they determine the same propagation of uncertainty.

\begin{corollary}\label{corollary:app1} For each $t \ge 1$ the function $\Delta \textrm{step}^t$ is unique with the property that for all $a,b,c,d \ge 1$ and any type $A$ satisfying
\be\label{eq:dimboundcrazy}
3^{\max\{c,d\}+1} \dim(\dntn{A}) > \tfrac{1}{2} \max\{ a + t, b + t \}
\ee
the diagram
\[
\xymatrix@C+5pc@R+1.5pc{
(\Delta \Sigma)^{\mathbb{Z},\Box} \times \Delta Q \ar[r]^-{\Delta \textrm{step}^t} & (\Delta \Sigma)^{\mathbb{Z}, \Box} \times \Delta Q \ar[d]^-{\Pi_{c,d} \times 1}\\
(\Delta \Sigma)^{a} \times (\Delta \Sigma)^b \times \Delta Q \ar[u]^-{\textrm{fill}_{a,b} \times 1} \ar[r]_-{\Delta {^t}\pBoolstep^{a,b,c,d}_A} & (\Delta \Sigma)^c \times (\Delta \Sigma)^d \times \Delta Q
}
\]
commutes, where $\textrm{fill}_{a,b}$ fills the tape outside the region $[-a+1,b]$ with blanks, and $\Pi_{c,d}$ is the projection onto the region $[-c+1,d]$.
\end{corollary}
\begin{proof}
Let us first explain the strange constraint in \eqref{eq:dimboundcrazy}. By \cite[Definition 4.27]{clift_murfet2}
\[
{^t}\pBoolstep^{a,b,c,d} = \pUnpack^{c,d} \l {_\delta} \pStep^t \l \pPack^{a,b}\,.
\]
with various basetypes, which we omit from the notation. In calculating the denotations we make use of pairs of binary integers in $\Sigma^{\le a} \times \Sigma^{\le b}$ at the beginning (that is, immediately after $\pPack$) then $\Sigma^{\le a+1} \times \Sigma^{\le b+1}$ after one application of $\pStep$, and so on through to $\Sigma^{\le a+p} \times \Sigma^{\le b + p}$ just prior to applying $\pUnpack$. In order that the denotations of all these binary integers are linearly independent, it suffices by \cite[Remark B.11]{clift_murfet2} to have
\[
\dim(\dntn{A^{3^{e+1}}}) > \tfrac{1}{2} \max\{ a + p, b + p \}
\]
see also \cite[Remark 4.21]{clift_murfet2} where $e = \max\{ c, d \}$ and this yields \eqref{eq:dimboundcrazy}. With this hypothesis the polynomial function $F_\psi$ for $\pBoolstep$ may be computed by composing the polynomial functions for $t$ copies of $\pStep$. Using this and the theorem we may compute that
\begin{align*}
\Delta {^t}\pBoolstep^{a,b,c,d} &= \Delta\big( \pUnpack^{c,d} \l \pStep^t \l \pPack^{a,b} \big)\\
&= (\Pi_{c,d} \times 1 ) \circ (\textrm{patch} \times 1 ) \circ (\Delta \textrm{strstep})^t \circ \Delta \pPack^{a,b}\\
&= (\Pi_{c,d} \times 1 ) \circ (\Delta \textrm{step})^t \circ ( \textrm{patch} \times 1 ) \circ \Delta \pPack^{a,b}\\
&= (\Pi_{c,d} \times 1 ) \circ (\Delta \textrm{step})^t \circ (\textrm{fill}_{a,b} \times 1)
\end{align*}
as claimed.
\end{proof}

\bibliographystyle{amsalpha}
\providecommand{\bysame}{\leavevmode\hbox to3em{\hrulefill}\thinspace}
\providecommand{\href}[2]{#2}

\end{document}